\documentclass[twoside,11pt]{article}
\usepackage{geometry}
\usepackage{graphicx}
\usepackage{indentfirst}                                
\usepackage[colorlinks]{hyperref}
\hypersetup{linkcolor=blue,filecolor=black,urlcolor=blue, citecolor=black}   
\usepackage[numbers,sort&compress]{natbib}         

\ifxetex
\usepackage{letltxmacro}
\LetLtxMacro\SavedIncludeGraphics\includegraphics
\def\includegraphics#1#{
 \IncludeGraphicsAux{#1}%
}%
\newcommand*{\IncludeGraphicsAux}[2]{%
 \XeTeXLinkBox{%
  \SavedIncludeGraphics#1{#2}%
}}
\fi
\newcommand\orcidicon[1]{\href{https://orcid.org/#1}{\includegraphics[scale=0.02]{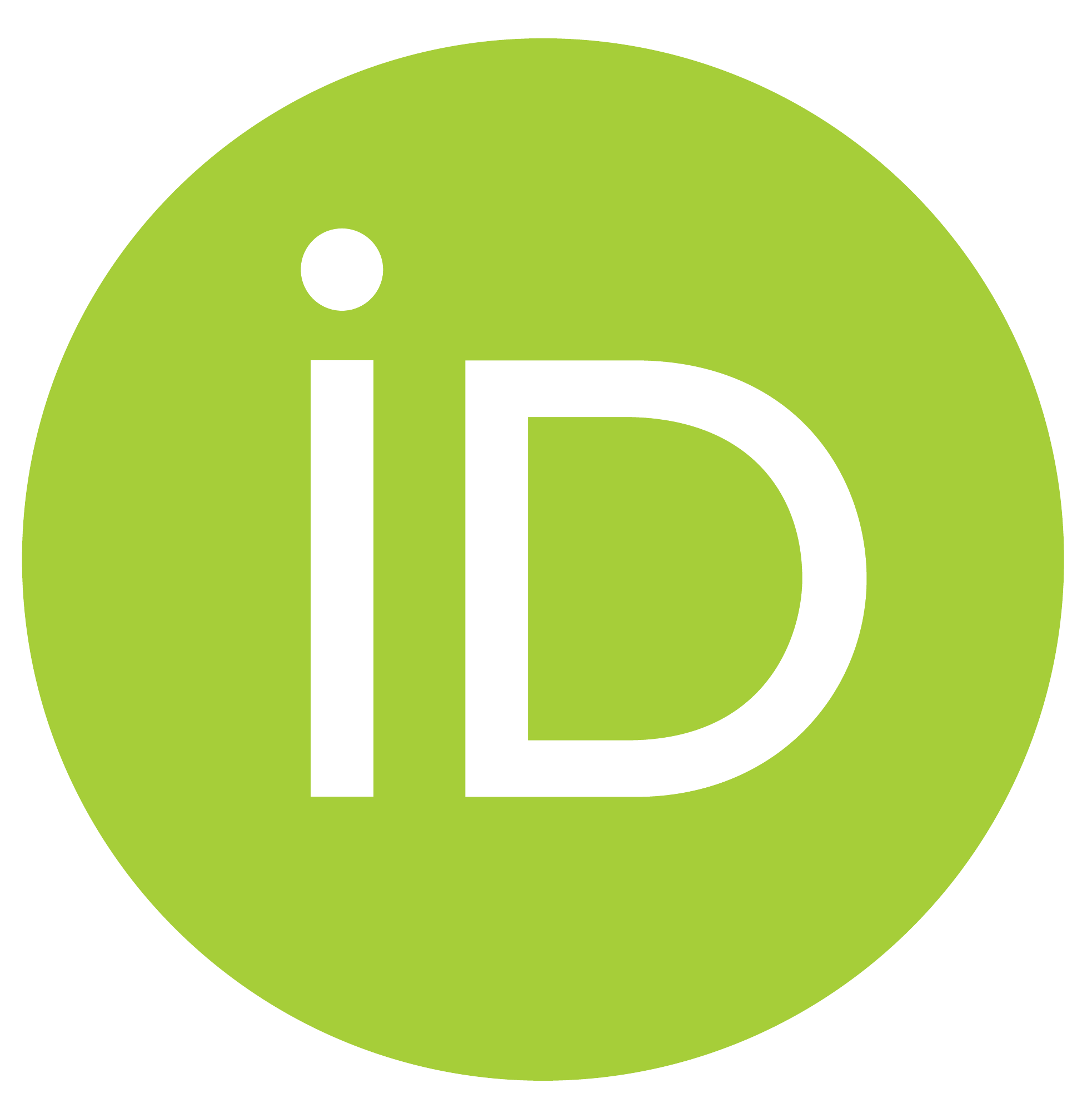}}}

\usepackage[titletoc,title]{appendix}
\usepackage{titlesec}
\titlespacing{\section}{0pt}{2.5ex}{1.5ex}
\titlespacing{\subsection}{0pt}{1.5ex}{1ex}
\titlespacing{\subsubsection}{0pt}{1.5ex}{1ex}
\titleformat{\section}{\large\bfseries\centering}{\thesection}{1em}{}
\titleformat{\subsection}[runin]{\bfseries}{\thesubsection.}{0.5em}{}[.\mbox{\ }]
\titleformat{\subsubsection}[runin]{\bfseries}{\thesubsubsection.}{0.4em}{}[.\mbox{\ }]

\usepackage{amsmath,amsfonts,amssymb,bm,mathrsfs,amsthm,amsxtra}
\numberwithin{equation}{section}
\newtheorem{lemma}{Lemma}[section]
\newtheorem{proposition}[lemma]{Proposition}
\newtheorem{theorem}{Theorem}[section]

\newtheorem{definition}{Definition}[section]

\usepackage{ulem}

\newcommand{\VERT}{\vert\kern-0.3ex\vert\kern-0.3ex\vert}
\def\d{\,\mathrm{d}}
\def\p{\partial}

\usepackage{accents}
\makeatletter
\def\widebar{\accentset{{\cc@style\underline{\mskip10mu}}}}
\makeatother

\newcommand\w[1]{\makebox[1em]{$#1$}}

\begin{document}
\title{\bf Well-posedness for Moving Interfaces with Surface Tension in Ideal Compressible MHD\let\thefootnote\relax\footnotetext{
The research of {\sc Yuri Trakhinin} was supported by Mathematical Center in Akademgorodok under Agreement No.~075-15-2019-1675 with the Ministry of Science and Higher Education of the Russian Federation.	
The research of {\sc Tao Wang} was supported by the Fundamental Research Funds for the Central Universities under Grant No.~2042022kf1183, the National Natural Science Foundation of China under Grants No.~11731008 and 11971359, and Hong Kong Institute for Advanced Study under Grant No.~9360157.
}
}

\author{
{\sc Yuri Trakhinin}\orcidicon{0000-0001-8827-2630}\thanks{Sobolev Institute of Mathematics, Koptyug av.~4, 630090 Novosibirsk, Russia; Novosibirsk State University, Pirogova str.~1, 630090 Novosibirsk, Russia. e-mail: trakhin@math.nsc.ru}
\qquad  \qquad {\sc Tao Wang}\orcidicon{0000-0003-4977-8465}\thanks{School of Mathematics and Statistics, Wuhan University, Wuhan 430072, China; Department of Mathematics, City University of Hong Kong, Hong Kong, China.
	e-mail: tao.wang@whu.edu.cn}
}

\date{\empty}

\maketitle

\vspace{-5mm}

{\footnotesize
\noindent{\bf Abstract}:\quad
We study the local well-posedness for an interface with surface tension that separates a perfectly conducting inviscid fluid from a vacuum. The fluid flow is governed by the equations of three-dimensional ideal compressible magnetohydrodynamics (MHD), while the vacuum magnetic and electric fields are supposed to satisfy the pre-Maxwell equations. The fluid and vacuum magnetic fields are tangential to the interface. This renders a nonlinear hyperbolic-elliptic coupled problem with a characteristic free boundary. We introduce some suitable regularization to establish the solvability and tame estimates for the linearized problem. Combining the linear well-posedness result with a modified Nash--Moser iteration scheme, we prove the local existence and uniqueness of solutions of the nonlinear problem. The non-collinearity condition required by Secchi and Trakhinin [\textit{Nonlinearity} 27(1): 105--169, 2014] for the case of zero surface tension becomes unnecessary in our result, which verifies the stabilizing effect of surface tension on the evolution of moving vacuum interfaces in ideal compressible MHD.

\vspace{2mm}
 \noindent{\bf Keywords}:\quad
 Ideal compressible MHD,
 pre-Maxwell equations,
 surface tension,
 moving interface,
 well-posedness

\vspace{2mm}
 \noindent{\bf Mathematics Subject Classification (2020)}:\quad
 76W05,
 35L65,
 35R35

\vspace{2mm}
\tableofcontents
}

\vspace{2mm}

\section{Introduction}\label{sec:intro}

We study the local well-posedness for an interface $\Sigma(t)$ with surface tension that separates a perfectly conducting inviscid fluid from a vacuum.
In the moving domain $\Omega^+(t)\subset \mathbb{R}^3$ occupied by the fluid,
we consider the following equations of ideal compressible magnetohydrodynamics (MHD) (see {\sc Landau--Lifshitz} \cite[\S 65]{LL84MR766230}):
\begin{align}\label{MHD1}
\left\{
\begin{aligned}
&    \p_t \rho +\nabla\cdot (\rho v )=0,\\[0.5mm]
&\p_t (\rho v )+\nabla\cdot (\rho v \otimes v -H \otimes H )+\nabla q =0,\\[0.5mm]
&\p_t H-\nabla\times(v\times H)=0,\\
&\p_t(\rho \mathfrak{e}+\tfrac{1}{2}\rho|v|^2+ \tfrac{1}{2}|H|^2) +\nabla\cdot
\left( v  (\rho \mathfrak{e}+\tfrac{1}{2}\rho|v|^2+p)+H\times(v\times H) \right)=0,
\end{aligned}
\right.
\end{align}
 together with the divergence constraint
\begin{align} \label{divH1}
  \nabla\cdot H=0.
\end{align}
Here density $\rho$, fluid velocity $v=(v_1,v_2,v_3)^{\mathsf{T}}$, magnetic field $H=(H_1,H_2,H_3)^{\mathsf{T}}$, and pressure $p$ are unknown functions of time $t$ and space variable $x=(x_1,x_2,x_3)$.
We denote by $q=p+\frac{1}{2}|H|^2$ the total pressure.
The internal energy $\mathfrak{e}$ and the density $\rho$ are given smooth functions of the pressure $p$ and the entropy $S$.
The thermodynamic variables are related through the Gibbs relation $\vartheta \d S=\d \mathfrak{e}+p\d (1/\rho)$
with $\vartheta>0$ being the temperature.

In the vacuum region $\Omega^-(t)\subset \mathbb{R}^3$,
the magnetic field ${h}=(h_1,h_2,h_3)^{\mathsf{T}}$ and the electric field ${e}=(e_1,e_2,e_3)^{\mathsf{T}}$
are assumed to satisfy the pre-Maxwell equations
\begin{alignat}{3}
& \nabla\times {h}=0,\qquad && \nabla \cdot {h}=0,
\label{pre-M1a}
\\
&\nabla\times {e}=-\p_t h,\qquad && \nabla\cdot {e}=0,
\label{pre-M1b}
\end{alignat}
obtained by neglect of the displacement current in Maxwell's equations \cite{BFG57}.
The vacuum electric field ${e}$ in \eqref{pre-M1a}--\eqref{pre-M1b}
can be degraded to a secondary variable, so that only one basic variable is needed, viz.~${h}$,  satisfying the elliptic system \eqref{pre-M1a}.

Let $\Omega:=(-1,1)\times\mathbb{T}^2$ be the reference domain occupied by the conducting fluid and the vacuum, where $\mathbb{T}^2$ denotes the $2$-torus and will be thought of as the unit square with periodic boundary conditions.
We assume that the moving interface $\Sigma(t)$ has the form of a graph
$\Sigma(t):=\{x_1=\varphi(t,x')\}\cap \Omega$
with $x'=(x_2,x_3)$,
where the interface function $\varphi:\mathbb{R}_+\times \mathbb{T}^2\to (-1,1)$ is to be determined.
Let us denote by $\Sigma^{\pm}:=\{\pm 1\}\times \mathbb{T}^2$ the fixed boundaries of the domains $\Omega^{\pm}(t):=\{x_1\gtrless \varphi(t,x') \}\cap \Omega$, respectively.
Then the boundary conditions read as
\begin{subequations}
\label{BC1}
\begin{alignat}{3}
&\p_t \varphi=v\cdot N
& & \textrm{on } \Sigma(t),
\label{BC1a}\\
&q-\tfrac{1}{2}|{h}|^2=\mathfrak{s}\mathcal{H}(\varphi)
&\qquad & \textrm{on } \Sigma(t),
\label{BC1b}\\
& H\cdot N=0,\quad
{h}\cdot N=0
&\quad & \textrm{on } \Sigma(t),
\label{BC1c}\\
&H_1=0,\quad v_1=0 && \textrm{on } \Sigma^+,
\label{BC1d}\\
& {h}\times \mathbf{e}_1 =\bm{j}_{\rm c} &&\textrm{on } \Sigma^-,
\label{BC1e}
\end{alignat}
\end{subequations}
with $N:=(1,-\p_2\varphi,-\p_3\varphi)^{\mathsf{T}}$ and $\mathbf{e}_1:=(1,0,0)^{\mathsf{T}}$.
Here, $\bm{j}_{\rm c}\in\mathbb{R}^3 $ is a given surface current,
$\mathfrak{s}>0$ is the constant coefficient of surface tension,
and $\mathcal{H}(\varphi)$ is twice the mean curvature of $\Sigma(t)$ defined by
\begin{align}
  \label{H.cal:def}
  \mathcal{H}(\varphi):= \mathrm{D}_{x'}\cdot\left( \frac{\mathrm{D}_{x'}\varphi }{\sqrt{1+|\mathrm{D}_{x'}\varphi|^2 }}\right)
  \quad \textrm{with }
  \mathrm{D}_{x'}:=
  \begin{pmatrix} \p_2\\ \p_3 \end{pmatrix}.
\end{align}
Condition \eqref{BC1a} states that
the interface moves with the motion of the conducting fluid,
which renders the free interface $\Sigma(t)$ {\it characteristic}.
Condition \eqref{BC1b} results from the presence of surface tension and the balance of the normal stresses at the interface \cite{Delhaye1974}.
Note that the effect of surface tension becomes especially important in modelling the flows of liquid metals \cite{SDGX07MR2356385}.
Condition \eqref{BC1c} means that the fluid and vacuum magnetic fields are tangential to the interface.
Conditions \eqref{BC1d} are the perfectly conducting wall and impermeable conditions.
Energy flows into the system and the system is not isolated from the outside due to condition \eqref{BC1e} \cite[\S 4.6]{GKP19}.

We supplement \eqref{MHD1}--\eqref{pre-M1a} and \eqref{BC1} with the initial data
\begin{align} \label{IC1}
  \varphi|_{t=0}=\varphi_0\quad  \textrm{on } \mathbb{T}^2,\qquad
  U|_{t=0}=U_0 \quad  \textrm{in } \Omega^+(0),
\end{align}
where $\|\varphi_0\|_{L^{\infty}(\mathbb{T}^2)}<1$ and $U:=(q,v,H,S)^{\mathsf{T}}\in \mathbb{R}^8
$.
Remark that the equation \eqref{divH1} and the first boundary conditions in \eqref{BC1c}--\eqref{BC1d} can be regarded as initial constraints \cite[Appendix A]{T09ARMAMR2481071}.
It is important to point out that if the interface function is given, then the vacuum magnetic field ${h}$ can be uniquely determined from the elliptic problem \eqref{pre-M1a}, \eqref{BC1c} and \eqref{BC1e} \cite{BFG57}.

In the absence of surface tension,
{\sc Secchi}--{\sc Trakhinin} \cite{ST13MR3148595,ST14MR3151094} obtained the linear and nonlinear well-posedness for the free boundary problem \eqref{MHD1}--\eqref{pre-M1a}, \eqref{BC1}, and \eqref{IC1} with $\mathfrak{s}=0$,
provided the fluid and vacuum magnetic fields are not collinear at each point of the interface.
The non-collinearity condition stems from the stability analysis of compressible current-vortex sheets ({\it cf.}~\cite{CW08MR2372810,T05MR2187618,T09ARMAMR2481071}).
On the one hand, this condition enhances the regularity of the moving interface $\Sigma(t)$, since it allows to express 
$\p_t\varphi$ and $\nabla \varphi$ as functions
of the traces of the velocity and magnetic fields.
But on the other hand, it excludes some important cases such as the case of vanishing vacuum magnetic field.
Hence, in \cite{TW21MR4201624} we considered the free-boundary problem in ideal compressible MHD and proved the first local well-posedness result for the case of vanishing vacuum magnetic field under the generalized Taylor sign condition.
We mention that the local well-posedness is still unsolved for nontrivial vacuum magnetic field without the non-collinearity condition {(see \cite{T16MR3503661} for futher discussions).

It is known that surface tension provides a stabilizing effect on the motion of free vacuum interfaces; see, for instance, {\sc Coutand--Shkoller} \cite{CS07MR2291920} and {\sc Shatah--Zeng} \cite{SZ08MR2388661,SZ11MR2763036} for the incompressible Euler equations with surface tension, and {\sc Coutand et al.}~\cite{CHS13MR3139610} for the compressible isentropic case.
Motived by these works, the authors \cite{TW21b} recently investigated the free-boundary ideal compressible MHD equations with surface tension and constructed the unique solution without assuming any Taylor-type sign condition. 
The result obtained in \cite{TW21b} corresponds to the special case of vanishing vacuum magnetic field.
Therefore, it is natural to examine the stabilizing effect of surface tension on the evolution of moving interfaces in ideal compressible MHD for general vacuum magnetic field,
that is, to study the local well-posedness for problem \eqref{MHD1}--\eqref{pre-M1a}, \eqref{BC1}, and \eqref{IC1} with $\mathfrak{s}>0$.

Different from the nonlinear hyperbolic problem studied in \cite{TW21b},
the problem \eqref{MHD1}--\eqref{pre-M1a}, \eqref{BC1}, and \eqref{IC1} under consideration is a nonlinear hyperbolic-elliptic coupled problem with a characteristic free boundary. For its resolution, we consider very general constitutive relations satisfying the physical assumption that the sound speed is positive. The first step in our analysis is to reformulate the free-boundary problem \eqref{MHD1}--\eqref{pre-M1a}, \eqref{BC1}, and \eqref{IC1} into an equivalent fixed-boundary problem by use of a simple lift of the graph.

We establish the solvability and high-order energy estimates for the linearized problem around a certain basic state by passing to the limit from a well-chosen regularization.
For this purpose, we rewrite the linearized vacuum equations into a div-curl system,
so that we can introduce some suitable decompositions and scalar potential $\xi$ to obtain the reduced problem \eqref{ELP3} with homogeneous boundary conditions and homogeneous vacuum equations as in \cite{T16MR3503661}.
Similar to our previous work \cite{TW21b}, 
the $L^2$ estimate for the problem \eqref{ELP3} is not closed.
To deal with this situation, we design an elaborate $\varepsilon$--regularization that has a unique solution satisfying uniform-in-$\varepsilon$ energy estimates in certain Sobolev spaces of sufficiently large regularity.

More precisely, 
we add some terms in the fluid equations and the boundary conditions to close the $L^2$ estimates for both the regularization \eqref{Reg} and its dual problem \eqref{dual}.
This allows us to derive the existence of $L^2$ weak solutions to the $\varepsilon$--regularization for any small fixed parameter $\varepsilon>0$ by applying the duality argument.
Then we build for the regularized problem \eqref{Reg} high-order energy estimates uniformly in $\varepsilon$.
Noting that the fluid variables $W$ satisfy a symmetric hyperbolic problem with characteristic boundary,
as in \cite{TW21MR4201624,TW21b}, we work in the anisotropic Sobolev spaces $H_*^m$ with different regularity in the normal and tangential directions, first introduced by {\sc Chen} \cite{C07MR2289911} (see {\sc Secchi} \cite{S96MR1405665} for a general theory).
In the estimate of tangential derivatives ({\it cf.}~\S \ref{sec:tangential}),
we have to deal with the troublesome boundary  term, {\it i.e.}, the integral of $Q_{2c}$ ({\it cf.}~\eqref{Q2:es1}), due to the introduction of term $\varepsilon \Delta_{x'}^2\psi$ in the boundary condition \eqref{Reg.c}, where $\Delta_{x'}^2$ denotes the biharmonic operator in the tangential space coordinates.
To overcome this difficulty, we add another regularized term $-\varepsilon \Delta_{x'}^2\xi$ in the boundary condition \eqref{Reg.e}, so that the good term $\mathcal{J}_2$ ({\it cf.}~\eqref{Q2a:es1}--\eqref{J2:es}) can be used to estimate the integral of $Q_{2c}$ as in \eqref{Q2c:es}.
Moreover, for the regularized problem \eqref{Reg}, the $L^2(\Sigma_t)$-norm of $\p_t^m \psi$ rather than its instant $L^2(\Sigma)$-norm is controllable because of the boundary condition \eqref{Reg.c}.
As such, we consider the estimate for $Q_4$ in the cases with $\alpha_0<m$ and $\alpha=m$ separately; see \S \ref{sec:tangential} for the details.
The elliptic equation for the gradient of the potential $\xi$ helps to gain the normal derivatives of $\nabla\xi$, which along with the spatial boundary regularity enhanced from surface tension enables us to obtain high-order energy estimates for the $\varepsilon$--regularization \eqref{Reg} uniformly in $\varepsilon$.
Then we achieve the resolution and high-order energy estimate for the linearized problem \eqref{ELP1} by passing to the limit $\varepsilon\to 0$.

Remark that our energy estimate \eqref{tame} for the linearized problem is a so-called tame estimate, since it exhibits a {\it fixed} loss of regularity from the basic state and source terms to the solution.
Based on the linear well-posedness result, we can construct local solutions for the nonlinear problem by an appropriate modification of the Nash--Moser iteration scheme developed by {\sc H\"{o}rmander} \cite{H76MR0602181} and {\sc Coulombel--Secchi} \cite{CS08MR2423311}. In particular,
a smooth intermediate state should be introduced and estimated, so that the state around which we linearize at each iteration step can satisfy certain constraints for the linear solvability.




The manuscript is organized as follows. In Section \ref{sec:main}, we formulate the nonlinear problems and present the main result of this paper, {\it i.e.}, Theorem \ref{thm:main}. Section \ref{sec:linear} is devoted to proving the energy estimates and unique solvability for the linearized problem around a suitable basic state ({\it cf.}~Theorem \ref{thm:linear}). In Section \ref{sec:Nash-Moser}, we give the proof of local well-posedness for the nonlinear problem.

\section{Main Result}\label{sec:main}
In this section we reformulate the nonlinear free-boundary problem \eqref{MHD1}--\eqref{pre-M1a}, \eqref{BC1}, and \eqref{IC1} into an equivalent fixed-boundary problem,
state the main result of this paper, 
and present the notation for later use.

\subsection{Nonlinear problems}
We consider very general, smooth constitutive relations $\rho=\rho(p,S)$ and $\mathfrak{e}=\mathfrak{e}(p,S)$ for the perfectly conducting fluid.
All we need is that the sound speed $a: =  \sqrt{p_{\rho}  (\rho, S)}$ is positive for all $\rho\in (\rho_*,\rho^*)$,
where $\rho_*$ and $\rho^*$ are some nonnegative constants.
Thanks to the constraint \eqref{divH1} and the Gibbs relation, for smooth solutions $U=(q,v,H,S)^{\mathsf{T}}$,
we can reduce the fluid equations \eqref{MHD1} to the symmetric hyperbolic system
\begin{align}
  \label{MHD:vec}
  A_0^+(U)\p_t U+\sum_{i=1}^{3}A_i^+(U)\p_i U=0\qquad \textrm{in } \Omega^+(t),
\end{align}
where the coefficient matrices are defined as
\setlength{\arraycolsep}{4pt}
\begin{align}
\nonumber  
  &A_0^+(U):=
  \begin{pmatrix}
 \dfrac{1}{\rho a^2} & 0 & -\dfrac{1}{\rho a^2} H^{\mathsf{T}} & 0\\[2.5mm]
 0 & \rho {I}_3 & {O}_3 & 0\\[1.5mm]
 -\dfrac{1}{\rho a^2} H & {O}_3 & {I}_3+\dfrac{1}{\rho a^2}H\otimes H & 0\\[1.5mm]
 \w{0}  & \w{0}  & \w{0} & \w{1}
  \end{pmatrix},
\end{align}
and
\begin{align}
\nonumber  
  &A_i^+(U):=
  \begin{pmatrix}
 \dfrac{v_i}{\rho a^2} & \mathbf{e}_i^{\mathsf{T}} & -\dfrac{v_i}{\rho a^2} H^{\mathsf{T}} & 0\\[2.5mm]
 \mathbf{e}_i & \rho v_i {I}_3 & -H_i {I}_3 & 0\\[1.5mm]
 -\dfrac{v_i}{\rho a^2} H & -H_i {I}_3 & v_i{I}_3+\dfrac{v_i}{\rho a^2}H\otimes H & 0\\[1.5mm]
 \w{0}  & \w{0}  & \w{0} & \w{v_i}
  \end{pmatrix}
  \ \textrm{ for $i=1,2,3$.}
\end{align}
Here and below, ${O}_m$ and ${I}_{m}$ are the zero and identity matrices of order $m$, respectively,
$\mathbf{e}_1:=(1,0,0)^{\mathsf{T}}$, $\mathbf{e}_2:=(0,1,0)^{\mathsf{T}}$, and $\mathbf{e}_3:=(0,0,1)^{\mathsf{T}}$.
Moreover, it is convenient to rewrite the vacuum equations \eqref{pre-M1a} as
\begin{align}
  \label{pre-M:vec}
\sum_{i=1}^{3}A_i^-\p_i {h}=0\qquad \textrm{in } \Omega^-(t),
\end{align}
where $A_1^-$, $A_2^-$, and $A_3^-$ are the constant matrices given by
\begin{align}
\nonumber  
A_1^-:=\begin{pmatrix}
0 & 0 &0 \\
0 & 0 &-1 \\
0 & 1 &0 \\
1 & 0 &0
\end{pmatrix},
\quad
A_2^-:=\begin{pmatrix}
  0 & 0 &1 \\
  0 & 0 &0 \\
  -1 & 0 &0 \\
 0 & 1 &0
\end{pmatrix},
\quad
A_3^-:=\begin{pmatrix}
  0 & -1 &0 \\
  1 & 0 &0 \\
  0 & 0 &0 \\
  0 & 0 &1
\end{pmatrix}.
\end{align}

Let us reformulate the free-boundary problem \eqref{MHD:vec}, \eqref{pre-M:vec}, \eqref{BC1}, and \eqref{IC1} into an equivalent fixed-boundary problem.
To this end, we take the lifting function $\Phi$ as
\begin{align} \nonumber  
  \Phi(t,x):=x_1+ \chi( x_1)\varphi(t,x'),
\end{align}
where $\chi\in C^{\infty}_0(-1,1)$ 
is a cut-off function satisfying
\begin{align} \label{chi:def}
   \|\chi'\|_{L^{\infty}(\mathbb{R})} < \frac{4}{\|\varphi_0\|_{L^{\infty}(\mathbb{T}^2)}+3},
\qquad
  \chi\equiv 1\quad \textrm{on }[-\delta_0,\delta_0]
\end{align}
for some small constant $\delta_0>0$.
Then the change of variables $(t,x)\mapsto (\tilde{t},\tilde{x})$ with
$t=\tilde{t}$ and $x=(\Phi(\tilde{t},\tilde{x}),\tilde{x}_2,\tilde{x}_3)$
maps $\Sigma^{\pm}$  and $\Sigma(t)$ to $\Sigma^{\pm}$ and $\Sigma:=\{0\}\times \mathbb{T}^2$, respectively.
The domains $\Omega^{+}(t)$ and $\Omega^{-}(t)$ are mapped to 
$\Omega^+:=(0,1)\times \mathbb{T}^2$ and $\Omega^-:=(-1,0)\times \mathbb{T}^2$,
respectively.
Since $\|\varphi_0\|_{L^{\infty}(\mathbb{T}^2)}<1$, we know 
that $\|\varphi\|_{L^{\infty}([0,T]\times\mathbb{T}^2)}\leq \frac{1}{2}(\|\varphi_0\|_{L^{\infty}(\mathbb{T}^2)}+1)<1$ for some small $T>0$.
As a result, the change of variables is admissible on the time interval $[0,T]$.
Let us introduce
\begin{align*}
  \widetilde{U}(\tilde{t},\tilde{x})=U(t,x),\qquad
  \tilde{h}(\tilde{t},\tilde{x})={h}(t,x).
\end{align*}
Then the free-boundary problem \eqref{MHD:vec}, \eqref{pre-M:vec}, \eqref{BC1}, and \eqref{IC1} can be reduced to the following nonlinear fixed-boundary problem:
\begin{subequations} \label{NP1}
  \begin{alignat}{2}
 \label{NP1a}
 &\mathbb{L}_+(U,\Phi) :=L_+(U,\Phi)U =0
 & &\textrm{in  }  [0,T]\times
 \Omega^+,\\
 \label{NP1b}
&\mathbb{L}_-({h},\Phi) :=L_-(\Phi){h} =0
& &\textrm{in  }  [0,T]\times
\Omega^-,\\
 \label{NP1c}
 &  \p_t \varphi =v\cdot N,\quad
 q-\tfrac{1}{2}|{h}|^2=\mathfrak{s}\mathcal{H}(\varphi),\quad
 {h}\cdot N=0
 &&\textrm{on  }    [0,T]\times \Sigma,\\
 \label{NP1d}
 &v_1=0\qquad \textrm{on  }  [0,T]\times \Sigma^+, \qquad\quad
{h}\times \mathbf{e}_1=\bm{j}_{\rm c}\qquad   & &\textrm{on  }  [0,T]\times \Sigma^-,\\
 \label{NP1e}
 &(U,\varphi)|_{t=0}=(U_0,\varphi_0),  &  &
  \end{alignat}
\end{subequations}
where we drop the tildes for notational simplicity.
The operators $L_{\pm}$ are defined as
\begin{align}
  \label{L+:def}
  &L_+(U,\Phi):=A_0^+(U)\partial_t+\widetilde{A}_1^+(U,\Phi)\partial_1+ A_2^+(U)\partial_2
  +A_3^+(U)\partial_3,\\
  \label{L-:def}
  &L_-(\Phi):=\widetilde{A}_1^-(\Phi)\partial_1+ A_2^-\partial_2+A_3^-\partial_3,
\end{align}
with $\widetilde{A}_1^-(\Phi):=  (A_1^--\partial_2\Phi A_2^--\partial_3\Phi A_3^-)/\partial_1\Phi$ and
\begin{align}
\nonumber   
  &\widetilde{A}_1^+(U,\Phi):=
  \frac{1}{\partial_1\Phi}\big(A_1^+(U)-\partial_t\Phi A_0^+(U)-\partial_2\Phi A_2^+(U)-\partial_3\Phi A_3^+(U)\big). 
\end{align}
For later use, we introduce the following boundary operators ({\it cf.}~\eqref{NP1c}--\eqref{NP1d}):
\begin{align}
&\label{B.bb:def}
\mathbb{B}_+(U,h,\varphi):=
\begin{pmatrix}
 \p_t \varphi -v\cdot N\\
 q-\frac{1}{2}|{h}|^2-\mathfrak{s}\mathcal{H}(\varphi)\\
 v_1
\end{pmatrix}, \quad
\mathbb{B}_-(h,\varphi):=
\begin{pmatrix}
 {h}\cdot N\\
{h}\times \mathbf{e}_1-\bm{j}_{\rm c}
\end{pmatrix}.
\end{align}

\subsection{Main result}
In the new coordinates, the equation \eqref{divH1} and the first conditions in \eqref{BC1c}--\eqref{BC1d} become
\begin{align}
 \label{inv1b}
 &  
 \frac{\p_1 H_1}{\partial_1\Phi}
 +\sum_{i=2}^3\Big(\partial_i -\frac{\partial_i\Phi }{\partial_1\Phi}\partial_1\Big)H_i =0\qquad  \textrm{in  }  [0,T]\times \Omega^+,
 \\[0.5mm]
 \label{inv2b}
 &  H\cdot N=0\quad \textrm{on }  [0,T]\times \Sigma,
 \qquad H_1=0\quad \textrm{on }  [0,T]\times \Sigma^+.
\end{align}
These last identities \eqref{inv1b}--\eqref{inv2b} are satisfied as long as they hold at the initial time,
for which the proof can be found in \cite[Appendix A]{T09ARMAMR2481071}.

Denote by $\lfloor s \rfloor$ the floor function of $s\in\mathbb{R}$ that maps $s$ to the greatest integer less than or equal to $s$.
We now state the main result of this paper, that is,
the local well-posedness theorem for nonlinear problem \eqref{NP1}.
\begin{theorem}
\label{thm:main}
Let $\bm{j}_{\rm c}\in H^{m+3/2}([0,T_0]\times \Sigma^-)$ for some $T_0>0$ and integer $m\geq 20$.
Suppose that the initial data $(U_0,\varphi_0)\in H^{m+3/2}(\Omega^+)\times H^{m+2}(\mathbb{T}^{2})$
satisfy $\|\varphi_0\|_{L^{\infty}(\mathbb{T}^2)}<1$, the constraints \eqref{inv1b}--\eqref{inv2b}, the hyperbolicity condition
\begin{align}
 \label{hyperbolicity}
 \rho_*<\inf_{\Omega^+} \rho(U_0)\leq \sup_{\Omega^+} \rho(U_0)<\rho^*,
\end{align}
and the compatibility conditions up to order $m$ \textnormal{(}see Definition \textnormal{\ref{def:1})}.
Then there exists a sufficiently small time $T>0$, such that the problem \eqref{NP1} has a unique solution $(U,h, \varphi)$ in $H^{\lfloor (m-9)/2\rfloor}([0,T]\times\Omega^+)\times H^{m-9}([0,T]\times\Omega^-)\times H^{m-9}([0,T]\times\mathbb{T}^{2})$ satisfying $\mathrm{D}_{x'}\varphi\in H^{m-9}([0,T]\times\mathbb{T}^{2})$.
\end{theorem}

\subsection{Notation}
We adopt the following notation throughout the paper:
\begin{list}{}{
  \setlength{\topsep}{1.5mm} 
  \setlength{\parsep}{\parskip} 
  \setlength{\partopsep}{0mm} 
  \setlength{\itemsep}{1.5mm} 
  \setlength{\labelwidth}{2em} 
  \setlength{\labelsep}{0.5em} 
  \setlength{\leftmargin}{2.5em} 
  \setlength{\itemindent}{0em}   
  \setlength{\listparindent}{0em}      
 }
  \item [(\rm a)]
We write $C$ for some universal positive constant
and $C(\cdot)$ for some positive constant depending on the quantities listed in the parenthesis.
Symbol $A\lesssim B$ denotes $A \leq C B$,
while $A\lesssim_{a_1,\ldots,a_m} B$ means that $A \leq C(a_1,\ldots,a_m)B$
for given parameters $a_1,\ldots,a_m$.
\item[(\rm b)]
Recall that $\Sigma^{\pm}:=\{\pm 1\}\times\mathbb{T}^2$ are
the boundaries of the reference domain $\Omega:=(-1,1)\times\mathbb{T}^2$
and $\Sigma:=\{0\}\times\mathbb{T}^2$ is the common boundary of
$\Omega^{\pm}:=\{0<\pm x_1<1\}\cap \Omega$.
For $T>0$, we set $\Omega_T:=(-\infty,T)\times \Omega$,
$\Sigma_T:=(-\infty,T)\times \Sigma,$ and
\begin{alignat*}{3}
\Omega_T^{\pm}:=(-\infty,T)\times \Omega^{\pm},
\quad \Sigma_T^{\pm}:=(-\infty,T)\times \Sigma^{\pm}.
\end{alignat*}
\item[(\rm c)]
We denote by $\p_t$ (or $\p_0$)
the time derivative $\frac{\p}{\p t}$.
Set
$\nabla:=(\p_1,\p_2,\p_3)^{\mathsf{T}}$
and $\mathrm{D}:=(\p_0,\p_1,\p_2,\p_3)^{\mathsf{T}}$,
where $\p_i:=\frac{\p}{\p x_i}$ for $i=1,2,3$.
For any  $\beta=(\beta_1,\ldots,\beta_n)\in\mathbb{N}^n$ and   $z=(z_1,\ldots,z_n)\in\mathbb{R}^n$,
we define
\begin{align*}
  &\beta!:=\beta_1!\cdots \beta_n!,
  \quad |\beta|:=\beta_1+\cdots+\beta_n,
  \quad z^{\beta}:=z_1^{\beta_1}\cdots z_n^{\beta_n},
  \\
  &\mathrm{D}_{z}:=\left(\frac{\p}{\p z_1},\ldots, \frac{\p}{\p z_n} \right)^{\mathsf{T}},
  \quad
  \mathrm{D}_{z}^{\beta}:=\left(\frac{\p}{\p z_1}\right)^{\beta_1}\cdots
  \left(\frac{\p}{\p z_n} \right)^{\beta_n}.
\end{align*}

\item[(\rm d)] We abbreviate $z':=(z_2,z_3)^{\mathsf{T}}$ for $z:=(z_1,z_2,z_3)^{\mathsf{T}}$,
so that $x':=(x_2,x_3)$.
We use $\mathrm{D}_{x'}:=(\p_2,\p_3)^{\mathsf{T}}$,
$\Delta_{x'}:=\mathrm{D}_{x'}\cdot \mathrm{D}_{x'}$, and $\Delta_{x'}^2:=\Delta_{x'}\Delta_{x'}$
to denote the gradient, Laplacian, and biharmonic operators in the tangential space coordinates $x'$, respectively.
For any integer $m\geq 2$,  we use
$  \mathrm{D}_{x'}^{m}:=(\p_2^m,\p_2^{m-1}\p_3,\ldots,\p_2 \p_3^{m-1}, \p_3^m)^{\mathsf{T}} $
to represent the vector of all partial derivatives in $x'$ of order $m$.

\item[(\rm e)] We denote
$  \mathrm{D}_*^{\alpha}:=\p_t^{\alpha_0} (\sigma \p_1)^{\alpha_1}\p_2^{\alpha_2}  \p_3^{\alpha_3} \p_1^{\alpha_{4}}$
for $ \alpha:=(\alpha_0,\ldots,\alpha_{4})\in\mathbb{N}^{5}$ and $\sigma:=x_1(1-x_1)$.
For $m\in\mathbb{N}$ and $I\subset \mathbb{R}$, the anisotropic Sobolev space $H_*^{m}(I\times \Omega^+)$ is defined as
\begin{align*}
  &H_*^m(I\times\Omega^+):=
  \{ u\in L^2(I\times\Omega^+):\, \mathrm{D}_*^{\alpha} u\in L^2(I\times\Omega^+)
  \textrm{ for } \langle \alpha \rangle\leq m  \},
\end{align*}
and equipped with the norm ${\|}\cdot{\|}_{H^m_*(I\times \Omega^+)}$, where
\begin{align}  \nonumber 
  {\|}u{\|}_{H^m_*(I\times \Omega^+)}^2:=
\sum_{\langle \alpha\rangle\leq m} \|\mathrm{D}_*^{\alpha} u\|_{L^2(I \times\Omega^+)}^2,
\quad
\langle \alpha \rangle :=\sum_{i=0}^3\alpha_i +2\alpha_{4}.
\end{align}
\item[(\rm e)]
For any $m\in \mathbb{N}$, we denote by $\mathring{\rm c}_m$
a generic and smooth matrix-valued function of
$\{(\mathrm{D}^{\alpha} \mathring{U},\mathrm{D}^{\alpha} \mathring{{h}},\mathrm{D}^{\alpha}\mathring{\Psi}
):|\alpha|\leq m\}$,
where $\mathrm{D}^{\alpha}:=\p_t^{\alpha_0} \p_1^{\alpha_1}\p_2^{\alpha_2} \p_3^{\alpha_3} $ for $\alpha:=(\alpha_0,\ldots,\alpha_{3})\in\mathbb{N}^{4}$.
The exact form of $\mathring{\rm c}_m$ may vary at different places.
\end{list}

\section{Linear Well-posedness}\label{sec:linear}
This section is devoted to showing the high-order energy estimates and unique solvability for the linearization of the problem \eqref{NP1} around a suitable basic state $(\mathring{U},\mathring{{h}},\mathring{\varphi})$.

\subsection{Main theorem for the linearized problem}
We assume that the basic state $(\mathring{U}(t,x),\mathring{{h}}(t,x),\mathring{\varphi}(t,x'))$
with $\mathring{U}=(\mathring{q} ,\mathring{v},\mathring{H},\mathring{S})^{\mathsf{T}}$
and $\mathring{{h}}=(\mathring{h}_1,\mathring{h}_2,\mathring{h}_3)^{\mathsf{T}}$
is sufficiently smooth
and satisfies
\begin{alignat}{3}
  \label{bas1a}
&\|\mathring{\varphi}\|_{L^{\infty}(\Sigma_T)}\leq
\tfrac{1}{2}(\|{\varphi}_0\|_{L^{\infty}( \mathbb{T}^2)} +1) <1,
\\
  \label{bas1b}
  &
  \rho_*<  \rho(\mathring{U})<\rho^*\qquad \textrm{on }\overline{\Omega_T^+},\\
  \label{bas1c}
 & \|\mathring{U}\|_{W^{3,\infty}(\Omega_T^+)}+
 \|\mathring{{h}}\|_{W^{3,\infty}(\Omega_T^-)}+
 \|\mathring{\varphi}\|_{W^{4,\infty}(\Sigma_T)}\leq K
\end{alignat}
for some constant $K>0$.
We also assume that
\begin{alignat}{3}
\label{bas1d}
&\p_t \mathring{\varphi}=\mathring{v}\cdot \mathring{N},
\quad \mathring{{h}}\cdot \mathring{N}=0,
\quad \mathring{H}\cdot \mathring{N}=0
\qquad &&\textrm{on }\Sigma_T,\\	
\label{bas1e}
&  \p_1 \mathring{{h}}\cdot \mathring{N}+\p_2\mathring{h}_2+\p_3\mathring{h}_3=0
\quad &&\textrm{on }\Sigma_T,
\\
\label{bas1f}
&
\mathring{H}_1=0,\quad
\mathring{v}_1=0
\quad  \textrm{on }\Sigma_T^+,
\qquad
\mathring{{h}}\times \mathbf{e}_1 =\bm{j}_{\rm c}
\quad &&\textrm{on }\Sigma_T^-,
\end{alignat}
where
$\mathring{N}:=(1,-\p_2\mathring{\varphi},-\p_3\mathring{\varphi})^{\mathsf{T}}$.
Constraint \eqref{bas1e} comes from restricting the last equation in \eqref{NP1b} on boundary $\Sigma_T$ for the basic state.
Moreover, $\p_1 \mathring{\Phi}>0$ on $\overline{\Omega_T}$ for
\[
\mathring{\Phi}(t,x):=x_1+\mathring{\Psi}(t,x),\qquad
\mathring{\Psi}(t,x):=\chi(x_1)\mathring{\varphi}(t,x').
\]

Let us introduce the good unknowns of {\sc Alinhac} \cite{A89MR976971}:
\begin{align} \label{good}
  \dot{V}:=V-\frac{\Psi}{\partial_1 \mathring{\Phi}}\partial_1\mathring{U},
\quad \dot{{h}}:={h}-\frac{\Psi}{\p_1 \mathring{\Phi}}\p_1 \mathring{{h}},
\end{align}
for $V=(q,v,H,S)^{\mathsf{T}} $ and $\Psi(t,x):=\chi(x_1)\psi(t,x')$.
Then the linearized operators for the interior equations \eqref{NP1a}--\eqref{NP1b} around the basic state $(\mathring{U},\mathring{{h}},\mathring{\varphi})$ are defined and simplified as
\begin{align}
  \mathbb{L}'_+\big(\mathring{U},\mathring{\Phi}\big)(V,\Psi)
 :=&\left.\frac{\mathrm{d}}{\mathrm{d}\theta}
  \mathbb{L}_{+}\big(\mathring{U} +{\theta}V ,
  \mathring{\Phi} +{\theta}\Psi \big)\right|_{\theta=0}
\nonumber%
  \\
  =& \,{L}_{+}\big(\mathring{U},\mathring{\Phi}\big)\dot{V}+\mathcal{C}_{+}(\mathring{U},\mathring{\Phi})\dot{V}  +\frac{\Psi}{\partial_1\mathring{\Phi}}
  \partial_1\mathbb{L}_+(\mathring{U},\mathring{\Phi} ),
  \label{Alinhac1}
\\[0.5mm]
 \mathbb{L}_{-}'\big(\mathring{{h}},\mathring{\Phi}\big)({h},\Psi)
 :=& \left.\frac{\mathrm{d}}{\mathrm{d}\theta} \mathbb{L}_{-}\big(\mathring{{h}} +{\theta}{h} , \mathring{\Phi} +{\theta}\Psi \big)\right|_{\theta=0}
=L_-\big(\mathring{\Phi}\big)\dot{{h}} +\frac{\Psi}{\partial_1\mathring{\Phi}}\partial_1\mathbb{L}_-(\mathring{{h}},\mathring{\Phi} ),
\label{Alinhac2}
\end{align}
where
\begin{align} \label{C.cal+:def}
  \mathcal{C}_+({U},{\Phi})V:=
  \sum_{k=1}^{8}V_k\Bigg(
   \frac{\p \widetilde{A}^{+}_1}{\p {U_k}}({U},{\Phi}) \partial_1 {U}
  +\sum_{i=0,2,3}\frac{\p A^{+}_i}{\p {U_k}}({U}) \partial_i {U}
  \Bigg).
\end{align}
To linearize the boundary conditions \eqref{NP1c},
we recall \eqref{H.cal:def} and calculate
\begin{align}
  \left.\frac{\mathrm{d}}{\mathrm{d}\theta} \mathcal{H}(\mathring{\varphi} +\theta \psi)\right|_{\theta=0}=
  \mathrm{D}_{x'}\cdot
  \frac{\mathrm{d}}{\mathrm{d}\theta}
  \bigg(\frac{\mathrm{D}_{x'}(\mathring{\varphi} +\theta \psi)}{\sqrt{1+|\mathrm{D}_{x'}(\mathring{\varphi} +\theta \psi)|^2}}\bigg)
  \bigg|_{\theta=0}
=\mathrm{D}_{x'}\cdot
  \big(\mathring{B}\mathrm{D}_{x'}\psi \big),
\nonumber
\end{align}
where $\mathring{B}$ is the positive definite matrix defined by
\begin{align}
\label{B.ring:def}
  \mathring{B}:=
\frac{I_3}{|\mathring{N}|}
- \frac{\mathrm{D}_{x'}\mathring{\varphi}\otimes \mathrm{D}_{x'}\mathring{\varphi}}{|\mathring{N}|^3}.
\end{align}
Consequently, for the boundary operators $\mathbb{B}_{\pm}$ defined by \eqref{B.bb:def}, we have
\begin{align}
 \nonumber
& \mathbb{B}_+'(\mathring{U} ,\mathring{h}, \mathring{\varphi})(V,h,\psi)
 :=\left.\frac{\mathrm{d}}{\mathrm{d}\theta}
 \mathbb{B}_+\big(\mathring{U} +\theta V ,\,\mathring{h} +\theta h ,\,\mathring{\varphi} +\theta \psi \big) \right|_{\theta=0}\\
&\qquad\qquad\qquad\qquad\ \  =  \begin{pmatrix}
  (\p_t + \mathring{v}' \cdot \mathrm{D}_{x'}) \psi-v\cdot\mathring{N}\\[0.5mm]
  q-\mathring{h}\cdot h-\mathfrak{s} \mathrm{D}_{x'}\cdot
  \big(\mathring{B}\mathrm{D}_{x'}\psi \big)\\[0.5mm]
  v_1
 \end{pmatrix},
 \label{B'+.bb:def}\\
&\mathbb{B}_-'(\mathring{h}, \mathring{\varphi})(h,\psi)
:=\left.\frac{\mathrm{d}}{\mathrm{d}\theta}
\mathbb{B}_-\big(\mathring{h} +\theta h ,\,\mathring{\varphi} +\theta \psi \big) \right|_{\theta=0}
= \begin{pmatrix}
 h\cdot\mathring{N}-  \mathring{h}' \cdot \mathrm{D}_{x'} \psi\\[0.5mm]
h\times \mathbf{e}_1
\end{pmatrix},
\label{B'-.bb:def}
\end{align}
where we denote $z':=(z_2,z_3)^{\mathsf{T}}$ for any vector $z:=(z_1,z_2,z_3)^{\mathsf{T}}$.

Apply the good unknowns \eqref{good}
and neglect the last terms in \eqref{Alinhac1}--\eqref{Alinhac2}
to obtain the following effective linear problem:
\begin{subequations} \label{ELP1}
  \begin{alignat}{3}
 &\mathbb{L}'_{e+}(\mathring{U}, \mathring{\Phi}) \dot{V}
 :=L_+(\mathring{U}, \mathring{\Phi})\dot{V}
 +\mathcal{C}_+( \mathring{U},\mathring{\Phi})\dot{V} =f^+
 &\quad  &\textrm{in } \Omega^+_T,
 \label{ELP1a}\\
 &L_-( \mathring{\Phi}) \dot{{h}}=f^-
&\quad  &\textrm{in } \Omega^-_T,
\label{ELP1b}\\
 \label{ELP1c}   &
  \mathbb{B}'_{e+}(\mathring{U}, \mathring{{h}}, \mathring{\varphi}) (\dot{V},\dot{{h}},\psi)
=g^+
  &&\textrm{on } \Sigma_T^2\times \Sigma_{T}^+,\\
 \label{ELP1d}
&   \mathbb{B}'_{e-}( \mathring{{h}}, \mathring{\varphi}) ( \dot{{h}},\psi)
=g^-
&&\textrm{on } \Sigma_T\times \Sigma_{T}^-,\\
 &(\dot{V},\psi)=0,
\qquad \dot{{h}}=0
    && \textrm{if } t<0,
 \label{ELP1e}
  \end{alignat}
\end{subequations}
where operators $L_{\pm}$ and $\mathcal{C}_+$ are defined by \eqref{L+:def}--\eqref{L-:def} and \eqref{C.cal+:def}, respectively.
According to the identities
\begin{align*}
 \mathbb{B}'_{e+}(\mathring{U}, \mathring{{h}}, \mathring{\varphi}) (\dot{V},\dot{{h}},\psi)
 =\mathbb{B}_+'(\mathring{U} ,\mathring{h}, \mathring{\varphi})(V,h,\psi)\ \textrm{ and }\
 \mathbb{B}'_{e-}(\mathring{{h}}, \mathring{\varphi}) (\dot{{h}},\psi)
 =\mathbb{B}_-'(\mathring{h}, \mathring{\varphi})(h,\psi),
\end{align*}
we get from definition \eqref{good} and constraint \eqref{bas1e} that
\begin{align}
\label{B'e+:def}
&\mathbb{B}'_{e+}(\mathring{U}, \mathring{{h}}, \mathring{\varphi}) (\dot{V},\dot{{h}},\psi)
:=\begin{pmatrix}
(\p_t  + \mathring{v}' \cdot \mathrm{D}_{x'}+\mathring{b}_1)\psi-\dot{v}\cdot\mathring{N}  \\[1mm]
\dot{q}-\mathring{{h}}\cdot \dot{{h}}
-\mathfrak{s} \mathrm{D}_{x'}\cdot
\big(\mathring{B}\mathrm{D}_{x'}\psi \big)+\mathring{b}_2  \psi \\[1mm]
\dot{v}_1
\end{pmatrix},\\[1mm]
\label{B'e-:def}
&\mathbb{B}'_{e-}(\mathring{{h}}, \mathring{\varphi}) (\dot{{h}},\psi)
:=\begin{pmatrix}
\dot{{h}}\cdot \mathring{N}
 -\mathrm{D}_{x'}\cdot(\mathring{h}'  \psi)
 \\[1mm]
\dot{{h}}\times \mathbf{e}_1
\end{pmatrix},
\end{align}
for $\mathring{b}_1:=-\p_1 \mathring{v}\cdot \mathring{N}$ and
$\mathring{b}_2:=\p_1\mathring{q}-\mathring{{h}}\cdot \p_1\mathring{{h}}.$
In \eqref{ELP1c} we employ the notation $\Sigma_{T}^2\times \Sigma_{T}^+$ to denote that
the first two components of this vector equation are taken on $\Sigma_{T}$ and the third component on $\Sigma_{T}^+$. The similar notation applies also for \eqref{ELP1d}.
The source terms $f^{\pm}$ and $g^{\pm}$ are supposed to vanish in the past, so that the second equation in \eqref{ELP1e} follows from \eqref{ELP1b}, \eqref{ELP1d}, and the first equation in \eqref{ELP1e}.

The well-posedness result for the effective linear problem \eqref{ELP1} is presented in the following theorem. Hereafter, we use the short-hand notation
\begin{align}
\|({U},{h},{\varphi})\|_{m} :=\, & \|{U} \|_{H_*^{m}(\Omega_{T}^+)}+\|{h} \|_{H^{m}(\Omega_{T}^-)} +\|{\varphi} \|_{H^{m}(\Sigma_{T})},  \label{norm:ring}
\\
\|(g^+,g^-)\|_{ H^{m} \times H^{m+1} }:=\,&  \|(g_1^+,g_2^+) \|_{H^{m}(\Sigma_{T})} +\|g_3^+\|_{H^{m}(\Sigma_{T}^+) }   \nonumber \\
& +\|g_1^-\|_{  H^{m+1}(\Sigma_{T}) }+\|g_2^-\|_{  H^{m+1}(\Sigma_{T}^-) }.
 \label{norm:g}
\end{align}
\begin{theorem}
\label{thm:linear}
Let $K_0>0$ and $m\in\mathbb{N}$ with $m\geq 6$.
Then there exist constants $T_0>0$ and $C(K_0)>0$, such that if
$(\mathring{U},\mathring{h},\mathring{\varphi}) \in H_*^{{m+4}}(\Omega_T^+)\times H^{{m+4}}(\Omega_T^-) \times  H^{{m+4}}(\Sigma_T)$ satisfies \eqref{bas1a}--\eqref{bas1f} and
\begin{align}  \label{H:thm.linear}
{\|}\mathring{U}{\|}_{H_*^{{10}}(\Omega_T^+)} +\| \mathring{h}\|_{H^{{10}}(\Omega_T^-)} +\|\mathring{\varphi}\|_{H^{{10}}(\Sigma_T)}\leq {K}_0,
\end{align}
and source terms $(f^+,f^-)\in H_*^{m}(\Omega_T^+)\times H^{m+1}(\Omega_T^-) $,
$g^+ \in  H^{m+1}(\Sigma_{T})^2\times  H^{m+1}(\Sigma_{T}^+)$, and $g^-\in  H^{m+2}(\Sigma_{T})\times H^{m+2}(\Sigma_{T}^-)$
vanish in the past and satisfy the compatibility conditions \eqref{h.nat:com} for some $0<T\leq T_0$,
then problem \eqref{ELP1} admits a unique solution $(\dot{V},\dot{h},\psi)\in H_*^{m}(\Omega_T^+)\times H^{m}(\Omega_T^-)\times H^{m}(\Sigma_T)$
satisfying
\begin{align}
\nonumber
&\|\dot{V}\|_{H_*^{m}(\Omega_T^+)}+
\|\dot{h}\|_{H^{m}(\Omega_T^-)}+\|\Psi\|_{H^{m}(\Omega_T)}+
\|(\psi,\mathrm{D}_{x'}\psi)\|_{H^{m}(\Sigma_T)}\\
&\quad \lesssim_{K_0}
\|(\mathring{U},\mathring{h},\mathring{\varphi})\|_{m+4}\left(\|f^+  \|_{H_*^{6}(\Omega_T^+) }+\|f^-  \|_{H^{7}(\Omega_T^-) }+\|(g^+,g^-)\|_{H^{7} \times H^{8} }\right)
\nonumber\\
& \qquad \quad\  +\|f^+  \|_{H_*^{m}(\Omega_T^+) }+\|f^-  \|_{H^{m+1}(\Omega_T^-) }
+\|(g^+,g^-)\|_{H^{m+1} \times H^{m+2} }.
  \label{tame}
\end{align}
\end{theorem}

The rest part of this section is devoted to proving the above theorem.

\subsection{Reformulation}
To reduce the problem \eqref{ELP1}, we compute
that the vacuum equations \eqref{ELP1b} are equivalent to
\begin{align}
 \label{div-curl1}
 \begin{pmatrix}
  \nabla\times \big(\p_1 \mathring{\Phi} \mathring{\eta}^{-\mathsf{T}} \dot{{h}} \big)\\[1mm]
  \nabla\cdot \big(\mathring{\eta}\dot{{h}} \big)
 \end{pmatrix}
 = 
 \begin{pmatrix}
  \mathring{\eta} & 0\\[1mm]
  0 & \p_1\mathring{\Phi}
 \end{pmatrix} f^-
 =:\tilde{f}^-
 \qquad \textrm{in }\Omega^-_T,
\end{align}
where $\mathring{\eta}$ is the invertible matrix defined by
\begin{align}
  \label{eta.ring:def}
\mathring{\eta}:=\begin{pmatrix}
1 & -\p_2\mathring{\Phi} & -\p_3\mathring{\Phi}\\[0.5mm]
0 & \p_1\mathring{\Phi} & 0\\[0.5mm]
0 &0 & \p_1\mathring{\Phi}
\end{pmatrix}.
\end{align}
Then we decompose $\dot{{h}}$ as $\dot{{h}}={{h}}_{\flat}+{{h}}_{\natural}$, where ${{h}}_{\natural}$ is required to solve the following div-curl boundary value problem ({\it cf.}~\eqref{div-curl1}, \eqref{ELP1d}, and \eqref{B'e-:def}):
\begin{align}
\label{div-curl2}
\left\{
\begin{aligned}
&\begin{pmatrix}
  \nabla\times \big(\p_1 \mathring{\Phi} \mathring{\eta}^{-\mathsf{T}} {{h}}_{\natural} \big)\\[1mm]
  \nabla\cdot \big(\mathring{\eta}{{h}}_{\natural} \big)
\end{pmatrix}
=\tilde{f}^-
\quad  \textrm{in }\Omega^-_T,\\
&{{h}}_{\natural}\cdot \mathring{N}=g_1^- \quad \textrm{on }\Sigma_T,
\qquad
  {{h}}_{\natural}\times \mathbf{e}_1=g_2^- \quad \textrm{on }\Sigma^-_T,\\[1mm]
&(x_2,x_3)\to {{h}}_{\natural}(t,x_1,x_2,x_3)\quad \textrm{is $1$-periodic}.
\end{aligned}
\right.
\end{align}
The resolution and $H^2$-estimate for \eqref{div-curl2} have been proved in {\sc Secchi--Trakhinin} \cite[\S 6]{ST14MR3151094}.
To get the $H^m$-estimate for $m\geq 6$,
as in the proof of \cite[Theorem 13]{ST14MR3151094},
we decompose
$\p_1 \mathring{\Phi} \mathring{\eta}^{-\mathsf{T}} {{h}}_{\natural} =\zeta^{\natural}+\nabla \xi_{\natural}$
for the vector $\zeta^{\natural}$ and the scalar $\xi_{\natural}$ satisfying
\begin{align}
\nonumber
& \left\{
\begin{aligned}
&\nabla\times \zeta^{\natural} =(\tilde{f}_1^-,\,\tilde{f}_2^-,\,\tilde{f}_3^-)^{\mathsf{T}},
\quad
\nabla\cdot \zeta^{\natural} =0\qquad  \textrm{in }\Omega^-_T,\\
&\zeta^{\natural}_1=0\quad \textrm{on }\Sigma_T,
\qquad  \zeta^{\natural} \times \mathbf{e}_1=g_2^- \quad \textrm{on }\Sigma^-_T,\\
&(x_2,x_3)\to \zeta^{\natural}(t,x_1,x_2,x_3)\quad \textrm{is $1$-periodic},
\end{aligned}
\right.
\\[1mm] 
 \nonumber
& \left\{
 \begin{aligned}
  &
  \nabla\cdot (\mathring{A}\nabla\xi_{\natural}) =\tilde{f}_4^- -\nabla \cdot (\mathring{A}\zeta^{\natural}) \qquad  \textrm{in }\Omega^-_T,\\
  &(\mathring{A}\nabla\xi_{\natural})_1=g_1^--(\mathring{A}\zeta^{\natural})_1\quad \textrm{on }\Sigma_T,
  \qquad  \xi_{\natural}=0 \quad \textrm{on }\Sigma^-_T,\\
  &(x_2,x_3)\to \xi_{\natural}(t,x_1,x_2,x_3)\quad \textrm{is $1$-periodic},
 \end{aligned}
 \right.
\end{align}
where
$(\mathring{A}\nabla\xi_{\natural})_1$ denotes the first component of  vector $\mathring{A}\nabla\xi_{\natural}$ and
$\mathring{A}$ is the positive definite matrix defined by
\begin{align} \label{A.ring:def}
\mathring{A} := \frac{1}{\p_1\mathring{\Phi}}\mathring{\eta}\mathring{\eta}^{\mathsf{T}}.
\end{align}
Then we can obtain the next lemma by using the elliptic regularization and the Moser-type calculus inequalities.
\begin{lemma} \label{h.nat:lem}
Let $(\tilde{f}^-,g_1^-,g_2^-)$ belong to
$ H^{m-1}(\Omega^-_T)\times H^{m-1/2}(\Sigma_T)\times H^{m-1/2}(\Sigma^-_T)$ for some integer $m\geq 6$.
Assume that the compatibility conditions
\begin{align}
\label{h.nat:com}
g_2^-\cdot  \mathbf{e}_1|_{\Sigma^-}=0 ,\qquad
\int_{\Sigma^-} u \cdot g_2^-
=\int_{\Omega^-} u \cdot (\tilde{f}_1^-,\,\tilde{f}_2^-,\,\tilde{f}_3^-)^{\mathsf{T}}
\end{align}
hold for all vectors $u\in H^1(\Omega^-)$ satisfying $u_2=u_3=0$ on $\Sigma$
and $\nabla\times u=0$ in $\Omega^-$.
Then the problem \eqref{div-curl2} has a unique solution ${{h}}_{\natural}$ in $H^m(\Omega^-_T)$ and
\begin{align}
  \|{{h}}_{\natural}\|_{H^m(\Omega^-_T)}
  \lesssim _{K}
\,&
\left(1+\|\mathring{\varphi}\|_{H^{m+1}(\Sigma_{T})}\right)\|(\tilde{f}^-, g_1^-,g_2^-)\|_{H^{5}(\Omega^-_T)\times H^{5}(\Sigma_T)\times H^{5}(\Sigma^-_T)}
\nonumber\\
&
\label{h.nat:est}
+  \|\tilde{f}^-\|_{H^{m-1}(\Omega^-_T)}
+  \|g_1^-\|_{H^{m-1/2}(\Sigma_T)}
+  \|g_2^-\|_{H^{m-1/2}(\Sigma^-_T)}.
\end{align}
\end{lemma}

To transform \eqref{ELP1} into a problem with homogeneous boundary conditions, 
we introduce the decomposition $\dot{V}=V_{\flat}+V_{\natural}$ with
$V_{\natural}:=(q_{\natural}, v_1^{\natural},0)^{\mathsf{T}}\in\mathbb{R}^8$ satisfying
\begin{align}
\label{V.natural}
q_{\natural}:=\mathfrak{R}_T\big(g_2^+ +\mathring{{h}}\cdot {{h}}_{\natural}\big),\quad
v_1^{\natural}:=\chi(x_1)\mathfrak{R}_T(-g_1^+)
+\chi(1-x_1)\widetilde{\mathfrak{R}}_T g_3^+,
\end{align}
where $\mathfrak{R}_T:H^m(\Sigma_T)\to H_*^{m+1}(\Omega_T^+)$ and $\widetilde{\mathfrak{R}}_T:H^m(\Sigma_T^+)\to H_*^{m+1}(\Omega_T^+)$ are the continuous extension operators (see \cite{OSY94MR1289186} for more details).
It follows from \eqref{ELP1}, \eqref{div-curl2}, and \eqref{V.natural} that vectors $V_{\flat}$ and ${{h}}_{\flat}$ solve the problem
\begin{subequations} \label{ELP2}
  \begin{alignat}{3}
 \label{ELP2a}
 &\mathbb{L}'_{e+}(\mathring{U}, \mathring{\Phi}) {V}=\tilde{f}^+
 := f^+-\mathbb{L}_{e+}'(\mathring{U},\mathring{\Phi})V_{\natural}
&\qquad &  \textrm{in } \Omega^+_T,\\
 \label{ELP2b}
&\nabla\times \big(\p_1 \mathring{\Phi} \mathring{\eta}^{-\mathsf{T}} {{h}} \big)
=0,\quad
\nabla\cdot \big(\mathring{\eta}{{h}} \big)=0
&& \textrm{in }\Omega^-_T,\\
 \label{ELP2c}
 &  \mathbb{B}'_{e+}(\mathring{U}, \mathring{{h}}, \mathring{\varphi}) (V,{{h}},\psi)=0
&& \textrm{on } \Sigma_T^2\times\Sigma_{T}^+,\\
 \label{ELP2d}
 &  \mathbb{B}'_{e-}(\mathring{{h}}, \mathring{\varphi}) ({{h}},\psi)=0
&& \textrm{on } \Sigma_T\times\Sigma_{T}^-,
\\
 \label{ELP2e}
&
 ({V},h,\psi) =0   &&\textrm{if } t<0,
  \end{alignat}
\end{subequations}
where we drop the subscript ``$\flat$'' for simplicity of notation
and operators $\mathbb{B}'_{e\pm }$ are given in \eqref{B'e+:def}--\eqref{B'e-:def}.

Let us derive a suitable reformulation for problem \eqref{ELP2}.
In view of \eqref{ELP2b}, 
we can introduce the scalar potential $\xi$ by
\begin{align}
\label{xi:def}
  \nabla \xi =\p_1 \mathring{\Phi} \mathring{\eta}^{-\mathsf{T}} {{h}}  \qquad \textrm{in }\Omega^-_T.
\end{align}
Then $\nabla\cdot (\mathring{A} \nabla \xi) =\nabla\cdot (\mathring{\eta} {{h}} )=0$ in $\Omega^-_T$,
where $\mathring{A}$ is the positive definite matrix defined by \eqref{A.ring:def}.
To rewrite \eqref{ELP2c}--\eqref{ELP2d} in terms of $\nabla \xi$,
we get from
\eqref{xi:def}, \eqref{eta.ring:def}, and the second condition in \eqref{bas1d} that
\begin{align} \label{bdy:id1}
  \mathring{{h}}\cdot {{h}}= \mathring{\eta}\mathring{{h}}\cdot \nabla\xi
  =\mathring{{h}}'\cdot\mathrm{D}_{x'}\xi,\quad
 h\cdot \mathring{N}=(\mathring{A}\nabla\xi)_1
  \qquad \textrm{on }\Sigma_T.
\end{align}
Moreover, as in \cite{TW21b,TW21MR4201624}, we set 
\begin{align}
  W:=\,& (q, v_1-\p_2\mathring{\Phi}v_2-\p_3\mathring{\Phi}v_3, v_2, v_3, H, S)^{\mathsf{T}}
\nonumber
 \\[0.5mm]
 =\, &J( \mathring{\Phi})^{-1}V
 \quad  \textrm{with\quad } {J}(\mathring{\Phi}):=
\begin{pmatrix}
\begin{matrix}
1 & 0 & 0 & 0\\
0 & 1 & \partial_2 {\mathring{\Phi}}  & \partial_3 {\mathring{\Phi}}  \\
0 & 0 & 1 & 0 \\
0 & 0 & 0 & 1
\end{matrix}
&O_4\\[2mm]
O_4 & I_4
\end{pmatrix}.
  \label{W:def}
\end{align}
In view of \eqref{ELP2}--\eqref{W:def}, it suffices to study the reduced problem
\begin{subequations}
  \label{ELP3}
  \begin{alignat}{3}
 \label{ELP3a}
 &{{\bf L}}W:=\sum_{i=0}^3{\bm{A}}_i\p_i W +{\bm{A}}_4 W =\bm{f}
 :=J(\mathring{\Phi})^{\mathsf{T}}\tilde{f}^+
 &&\textnormal{in }\Omega^+_T,\\
 \label{ELP3b}
 &\nabla\cdot (\mathring{A}\nabla\xi)=0
 &&\textnormal{in }\Omega^-_T,\\
  \label{ELP3c}
 &W_2=(\p_t+\mathring{v}'\cdot\mathrm{D}_{x'}+\mathring{b}_1)\psi
 &&\textnormal{on }\Sigma_T,\\
 \label{ELP3d}
 &W_1=\mathring{{h}}'\cdot\mathrm{D}_{x'}\xi+\mathfrak{s}\mathrm{D}_{x'}\cdot
 \big(\mathring{B}\mathrm{D}_{x'}\psi \big)
 -\mathring{b}_2 \psi
 &&\textnormal{on }\Sigma_T,\\
  \label{ELP3e}
&(\mathring{A}\nabla\xi)_1=\mathrm{D}_{x'}\cdot(\mathring{{h}}' \psi)&\quad &\textnormal{on }\Sigma_T,\\
 \label{ELP3f}
&W_2=0\quad  \textnormal{on }\Sigma^+_T,\qquad
\xi=0\quad  \textnormal{on }\Sigma^-_T,\qquad
(W,\xi, \psi)|_{t<0} =0,
 && 
  \end{alignat}
\end{subequations}
where $  {\bm{A}}_{i}:=J( \mathring{\Phi})^{\mathsf{T}}{A}_{i}^+(\mathring{U})J( \mathring{\Phi})$ for $i=0,2,3$, $  {\bm{A}}_1:=J( \mathring{\Phi})^{\mathsf{T}}\widetilde{A}_1^+(\mathring{U},\mathring{\Phi})J( \mathring{\Phi}),$
and
$  {\bm{A}}_4:=J( \mathring{\Phi})^{\mathsf{T}}\mathbb{L}_{e+}'(\mathring{U},\mathring{\Phi})J( \mathring{\Phi}).$
We deduce from \eqref{bas1d} and \eqref{bas1f} that
\begin{align}
  \label{decom}
  {\bm{A}}_1\big|_{\Sigma}=
  {\bm{A}}_1\big|_{\Sigma^+}=
  \begin{pmatrix}
 0 & 1 &0 \\
 1 & 0 & 0\\
 0 & 0 & O_6
  \end{pmatrix}
  =:{\bm{A}}_1^{(1)}.
\end{align}
Define ${\bm{A}}_1^{(0)}:={\bm{A}}-{\bm{A}}_1^{(1)}$
so that ${\bm{A}}_1^{(0)}|_{\Sigma}={\bm{A}}_1^{(0)}|_{\Sigma^+}=0.$

\subsection{$L^2$ estimate for the regularization}

To solve problem \eqref{ELP3}, we introduce the following $\varepsilon$--regularization:
\begin{subequations}
  \label{Reg}
  \begin{alignat}{3}
 \label{Reg.a}
 &{{\bf L}}_{\varepsilon}W:=\sum_{i=0}^3{\bm{A}}_i\p_i W-\varepsilon\bm{J}\p_1 W  +{\bm{A}}_4 W =\bm{f}
 &&\textnormal{in }\Omega^+_T,\\
 \label{Reg.b}
 &\nabla\cdot (\mathring{A}\nabla\xi)=0
 &&\textnormal{in }\Omega^-_T,\\
 \label{Reg.c}
 &W_2=(\p_t+\mathring{v}'\cdot\mathrm{D}_{x'}+\mathring{b}_1)\psi
 +\varepsilon \Delta_{x'}^2 \psi
 &&\textnormal{on }\Sigma_T,\\
 \label{Reg.d}
 &W_1=\mathring{h}'\cdot\mathrm{D}_{x'}\xi+\mathfrak{s}\mathrm{D}_{x'}\cdot
  \big(\mathring{B}\mathrm{D}_{x'}\psi \big)
  -\mathring{b}_2 \psi
 &&\textnormal{on }\Sigma_T,\\
 \label{Reg.e}
 &(\mathring{A}\nabla\xi)_1=\mathrm{D}_{x'}\cdot(\mathring{h}' \psi)
 +\varepsilon  \Delta_{x'}\xi
 -\varepsilon  \Delta_{x'}^2 \xi
 &\quad &\textnormal{on }\Sigma_T,\\
 \label{Reg.f}
 &W_2=0\quad  \textnormal{on }\Sigma^+_T,\qquad
 \xi=0\quad  \textnormal{on }\Sigma^-_T,\qquad
 (W,\xi, \psi)|_{t<0}=0,
 &&
  \end{alignat}
\end{subequations}
where $\bm{J}:={\rm diag}\,(0,1,0,0,0,0,0,0)$, the matrices $\mathring{A}$ and $\mathring{B}$ are defined in \eqref{A.ring:def} and \eqref{B.ring:def}, respectively,
$\Delta_{x'}:=\mathrm{D}_{x'}\cdot \mathrm{D}_{x'}$,
and $\Delta_{x'}^2:=\Delta_{x'}\Delta_{x'}$.
As in \cite[\S 2.3]{TW21b} for the problem with vanishing vacuum magnetic field,
we add the term $-\varepsilon\bm{J}\p_1 W $ in \eqref{Reg.a} to derive the $L^2$ estimate for regularization \eqref{Reg}.
The terms $\varepsilon\Delta_{x'}^2\psi$ and $\varepsilon \Delta_{x'}\xi $ containing respectively in \eqref{Reg.c} and \eqref{Reg.e} allow us to obtain the $L^2$ estimate for the dual problem of \eqref{Reg}.
Moreover, the term $-\varepsilon \Delta_{x'}^2\xi$ will become especially useful in establishing the uniform-in-$\varepsilon$ energy estimates for the problem \eqref{Reg}.

Let us show the $L^2$ energy estimate for regularization \eqref{Reg}.
Taking the scalar product of \eqref{Reg.a} with $W$, we utilize \eqref{decom} and \eqref{Reg.f} to obtain
\begin{align} \label{es:0a}
  \int_{\Omega^+} {\bm{A}}_0W\cdot W
  +\varepsilon \|W_2\|_{L^2(\Sigma_t)}^2
    -2\int_{\Sigma_t}W_1 W_2  \lesssim_K \|(\bm{f},W)\|_{L^2(\Omega_t^+)}^2.
\end{align}
It follows from the boundary conditions \eqref{Reg.c}--\eqref{Reg.d} that
\begin{align}
\nonumber
  2 W_1W_2=
\,&
\underbrace{2\mathring{h}'\cdot\mathrm{D}_{x'}\xi \p_t \psi}_{\mathcal{T}_1}+
\underbrace{2\mathring{h}'\cdot\mathrm{D}_{x'}\xi (\mathring{v}'\cdot\mathrm{D}_{x'}+\mathring{b}_1)\psi}_{\mathcal{T}_2}
+\underbrace{2 \varepsilon  W_1 \Delta_{x'}^2 \psi  }_{\mathcal{T}_3}
\\
&+\underbrace{2 \big\{\mathfrak{s} \mathrm{D}_{x'}\cdot
  \big(\mathring{B}\mathrm{D}_{x'}\psi \big)-\mathring{b}_2 \psi \big\}
  (\p_t+\mathring{v}'\cdot\mathrm{D}_{x'}+\mathring{b}_1)\psi
}
_{\mathcal{T}_4}
\qquad \textrm{on }\Sigma_t.
\label{W1W2}
\end{align}

Using $\mathcal{T}_1=
2\mathrm{D}_{x'}\cdot\big(\xi\p_t(\mathring{h}'\psi) \big)
-2\xi\mathrm{D}_{x'}\cdot\p_t(\mathring{h}'\psi)
-2\mathrm{D}_{x'}\xi \cdot \p_t \mathring{h}' \psi$
and \eqref{Reg.e} yields
\begin{align}
\label{T1:es1}
-\int_{\Sigma_t}\mathcal{T}_1
=2\int_{\Sigma_t} \xi \p_t (\mathring{A}\nabla\xi)_1
+\varepsilon \int_{\Sigma} \left(|\mathrm{D}_{x'} \xi |^2+|\Delta_{x'} \xi|^2\right)
+2\int_{\Sigma_t} \p_t \mathring{h}'  \psi \cdot \mathrm{D}_{x'}\xi  .
\end{align}
Passing in the first term on the right-hand side of  \eqref{T1:es1} to the volume integral,
we utilize \eqref{Reg.f} and \eqref{Reg.b} to infer
\begin{align}
\nonumber
 2\int_{\Sigma_t} \xi \p_t (\mathring{A}\nabla\xi)_1
&=2\int_{\Omega_t^-} \p_1\big(\xi \p_t (\mathring{A}\nabla\xi)_1\big)
=2\int_{\Omega_t^-} \nabla\cdot \big(\xi \p_t (\mathring{A}\nabla\xi)\big)\\
& =
2\int_{\Omega_t^-} \nabla \xi\cdot  \p_t(\mathring{A} \nabla\xi)
=\int_{\Omega^-}   \mathring{A} \nabla\xi \cdot \nabla \xi
+\int_{\Omega_t^-} \p_t\mathring{A} \nabla\xi \cdot \nabla \xi,
\label{T1:es2}
\end{align}
where the last identity results from the fact that the matrix $\mathring{A}$ is symmetric ({\it cf.}~\eqref{A.ring:def}).
The second term on the right-hand side of \eqref{T1:es1} is a good term,
while the last term will be estimated below.

Regarding the term $\mathcal{T}_2$ defined in \eqref{W1W2}, we compute
\begin{align}
\mathring{h}'\cdot\mathrm{D}_{x'}\xi ( \mathring{v}'\cdot\mathrm{D}_{x'})\psi
&=\mathring{v}'\cdot\mathrm{D}_{x'}\xi ( \mathring{h}'\cdot\mathrm{D}_{x'})\psi
-(\mathring{h}_2\mathring{v}_3-\mathring{h}_3\mathring{v}_2)
\big(\p_2(\psi \p_3\xi) - \p_3(\psi \p_2\xi) \big)
\nonumber\\
& =\mathring{v}'\cdot\mathrm{D}_{x'}\xi ( \mathring{h}'\cdot\mathrm{D}_{x'})\psi
+\mathrm{D}_{x'}\cdot (\mathring{\rm c}_0 \psi \mathrm{D}_{x'} \xi)
+ \mathring{\rm c}_1 \psi\mathrm{D}_{x'}\xi,
\label{h.v:id1}
\end{align}
where for any $m\in \mathbb{N}$ we denote by $\mathring{\rm c}_m$
a generic and smooth matrix-valued function of
$\{(\mathrm{D}^{\alpha} \mathring{U},\mathrm{D}^{\alpha} \mathring{{h}},\mathrm{D}^{\alpha}\mathring{\Psi}):|\alpha|\leq m\}$.
Using \eqref{h.v:id1} and \eqref{Reg.e} leads to
\begin{align}
\nonumber
\int_{\Sigma_t}\mathcal{T}_2
&=2\int_{\Sigma_t}
\left\{\mathring{v}'\cdot\mathrm{D}_{x'}\xi \big(  \mathrm{D}_{x'}\cdot (\mathring{h}' \psi)
- \mathrm{D}_{x'}\cdot \mathring{h}' \psi \big)
+ \mathring{\rm c}_1\psi \mathrm{D}_{x'}\xi \right\} \\
& =2\int_{\Sigma_t}
\left\{ \mathring{v}'\cdot\mathrm{D}_{x'}\xi
\big( (\mathring{A}\nabla\xi)_1-\varepsilon  \Delta_{x'}\xi
+\varepsilon  \Delta_{x'}^2  \xi \big)
+ \mathring{\rm c}_1 \psi\mathrm{D}_{x'}\xi
\right\}.
\label{T2:es1}
\end{align}
Let us make the estimate for each term in \eqref{T2:es1}.
Similar to \eqref{T1:es2}, 
we discover
\begin{align}
2\int_{\Sigma_t} \mathring{v}'\cdot\mathrm{D}_{x'}\xi (\mathring{A}\nabla\xi)_1
  =\, &\sum_{i=2,3}\int_{\Omega_t^-}
\left(2\nabla\mathring{v}_i\p_i\xi \cdot (\mathring{A} \nabla\xi)
-\p_i(\mathring{v}_i \mathring{A})\nabla\xi\cdot\nabla\xi\right).
\label{T2:es2}
\end{align}
Since
\begin{align*}
2\int_{\Sigma_t}\mathring{v}'\cdot\mathrm{D}_{x'}\xi  \Delta_{x'}^2  \xi
 =\int_{\Sigma_t} \big(2[\Delta_{x'} , \mathring{v}'\cdot\mathrm{D}_{x'}]\xi\Delta_{x'} \xi
-\mathrm{D}_{x'}\cdot \mathring{v}'|\Delta_{x'} \xi|^2\big),
\end{align*}
we have
\begin{align}
&2\varepsilon  \int_{\Sigma_t}
 \mathring{v}'\cdot\mathrm{D}_{x'}\xi
\left(  \Delta_{x'}^2  \xi - \Delta_{x'}\xi \right)
 \lesssim_K \varepsilon   \|(\mathrm{D}_{x'}\xi,\Delta_{x'} \xi)\|_{L^2(\Sigma_t)}^2,
\label{T2:es3}
\\
&\int_{\Sigma_t}  \mathring{\rm c}_1 \psi\mathrm{D}_{x'}\xi
=\int_{\Sigma_t} \big( \mathring{\rm c}_2 \psi + \mathring{\rm c}_1\mathrm{D}_{x'} \psi \big) \xi
\lesssim_K \|(\psi,\mathrm{D}_{x'}\psi)\|_{L^2(\Sigma_t)}^2
+\|\xi\|_{L^2(\Sigma_t)}^2.
\label{T2:es4}
\end{align}
For the last term in \eqref{T2:es4}, we employ integration by parts and Poincar\'{e}'s inequality (see, {\it e.g.}, {\sc Evans} \cite[\S 5.8.1]{E10MR2597943}) to get
\begin{align}
  \|\xi\|_{L^2(\Sigma)}^2\lesssim \|(\xi,\p_1\xi)\|_{L^2(\Omega^-)}^2
  \lesssim \|\nabla\xi\|_{L^2(\Omega^-)}^2,\quad
\|\xi\|_{L^2(\Sigma_t)}^2
\lesssim \|\nabla\xi\|_{L^2(\Omega_t^-)}^2.
  \label{xi:es1}
\end{align}

Utilizing the boundary condition \eqref{Reg.d} yields
\begin{align}
\nonumber
-\int_{\Sigma_t}\mathcal{T}_3
=\;&
2\varepsilon\mathfrak{s} \int_{\Sigma_t}\Delta_{x'}
\big(\mathring{B}\mathrm{D}_{x'}\psi \big)  \cdot \Delta_{x'} \mathrm{D}_{x'}\psi
\\
&+2\varepsilon \int_{\Sigma_t}\Delta_{x'}\mathrm{D}_{x'}\psi \cdot \mathrm{D}_{x'}(\mathring{h}'\cdot \mathrm{D}_{x'}\xi)
+2\varepsilon \int_{\Sigma_t}\Delta_{x'}\psi \Delta_{x'}(\mathring{b}_2\psi).
\nonumber
\end{align}
Recalling definition \eqref{B.ring:def} for matrix $\mathring{B}$,
we obtain 
\begin{align}
\int_{\Sigma_t}\mathcal{T}_3
\leq -\varepsilon \mathfrak{s}  \int_{\Sigma_t}\frac{|\Delta_{x'}\mathrm{D}_{x'} \psi|^2 }{|\mathring{N}|^3}
+\varepsilon C(K)
\sum_{|\alpha|\leq 2} \|(\mathrm{D}_{x'}^{\alpha}\psi,\mathrm{D}_{x'}\xi,
\mathrm{D}_{x'}^2\xi )\|_{L^2(\Sigma_t)}^2,
\label{T3:es1}
\end{align}
where $\mathrm{D}_{x'}^{m}:=(\p_2^m,\p_2^{m-1}\p_3,\ldots,\p_2 \p_3^{m-1}, \p_3^m)^{\mathsf{T}} $
denotes the vector of all partial derivatives in $x'$ of order $m\geq 2$.

A lengthy 
calculation implies ({\it cf.}~\cite[(2.20)]{TW21b})
\begin{align}
\label{T4:es1}
\int_{\Sigma_t}\mathcal{T}_4
\leq   -\mathfrak{s}\int_{\Sigma}  \dfrac{|\mathrm{D}_{x'}\psi|^2}{|\mathring{N}|^3}
- \int_{\Sigma} \mathring{b}_2 \psi^2
+C(K)\|(\psi, \mathrm{D}_{x'}\psi)\|_{L^2(\Sigma_t)}^2.
\end{align}

Plugging \eqref{T1:es1}--\eqref{T1:es2} and \eqref{T2:es1}--\eqref{T4:es1} into \eqref{es:0a}--\eqref{W1W2} and using the identity $\|\mathrm{D}_{x'}^2\xi\|_{L^2(\Sigma)}\leq \|\Delta_{x'}\xi\|_{L^2(\Sigma)}$ imply
\begin{align}
\nonumber
\|W(t)\|_{L^2(\Omega^+)}^2
&+\|\mathrm{D}_{x'}\psi(t)\|_{L^2(\Sigma)}^2
+\|(\xi,\nabla\xi)(t)\|_{L^2(\Omega^-)}^2
\\
&+\varepsilon\|(W_2,\mathrm{D}_{x'}^3\psi)\|_{L^2(\Sigma_t)}^2
+\varepsilon\|(\mathrm{D}_{x'}\xi,\mathrm{D}_{x'}^2\xi)(t)\|_{L^2(\Sigma)}^2
\nonumber
\\
 \lesssim_K
\|(&\bm{f},W)\|_{L^2(\Omega_t^+)}^2
+\|(\psi, \mathrm{D}_{x'}\psi)\|_{L^2(\Sigma_t)}^2
+\|\nabla\xi\|_{L^2(\Omega_t^-)}^2
\nonumber \\
&
+\varepsilon\|(\mathrm{D}_{x'}^{2} \psi,\mathrm{D}_{x'} \xi,\mathrm{D}_{x'}^2\xi)\|_{L^2(\Sigma_t)}^2
+\|\psi(t)\|_{L^2(\Sigma)}^2.
\label{es:0b}
\end{align}
We emphasize that the $L^2$ estimate \eqref{es:0b} is valid also for the case $\varepsilon=0$, that is, for the linear problem \eqref{ELP3}.

To control the last term in \eqref{es:0b},
we multiply \eqref{Reg.c} with $\psi$ to infer
\begin{align}
\|\psi(t)\|_{L^2(\Sigma)}^2
  +2\varepsilon \|\Delta_{x'}\psi\|_{L^2(\Sigma_t)}^2
 \leq \boldsymbol{\epsilon}\varepsilon \|W_2\|_{L^2(\Sigma_t)}^2
  +C(K, \boldsymbol{\epsilon}\varepsilon)\|\psi\|_{L^2(\Sigma_t)}^2
\label{es:0c}
\end{align}
for all $\boldsymbol{\epsilon}>0$.
Combining \eqref{es:0b} with \eqref{es:0c}, taking $\boldsymbol{\epsilon}>0$ small enough, and employing Gr\"{o}nwall's inequality, we have
\begin{multline}
\|W(t)\|_{L^2(\Omega^+)}^2+\|(\xi,\nabla\xi)(t)\|_{L^2(\Omega^-)}^2
  +\|(\psi,\mathrm{D}_{x'}\psi,\mathrm{D}_{x'}\xi,\mathrm{D}_{x'}^2\xi)(t)\|_{L^2(\Sigma)}^2
  \\
 +\|(W_2,W_1,\mathrm{D}_{x'}^2\psi,\mathrm{D}_{x'}^3\psi)\|_{L^2(\Sigma_t)}^2
  \lesssim_{K,\varepsilon}
  \|\bm{f}\|_{L^2(\Omega_t^+)}^2.
  \label{es:0}
\end{multline}
This is the desired $\varepsilon$-dependent $L^2$ energy estimate for regularization \eqref{Reg}.

\subsection{Existence for the regularization}
We prove the existence of solutions to the regularization \eqref{Reg} by applying the duality argument. For this purpose, we introduce the dual problem of \eqref{Reg}, which reads as
\begin{subequations} \label{dual}
  \begin{alignat}{3}
\label{dual.a}
 &{{\bf L}}_{\varepsilon}^*W^*
 :=\bigg(\!-\sum_{i=0}^{3}\bm{A}_i\p_i+\varepsilon\bm{J}\p_1+\bm{A}_4^{\mathsf{T}}-\sum_{i=0}^{3}\p_i\bm{A}_i\bigg)W^*=\bm{f}^*
 &\quad &\textnormal{in }\Omega^+,\\
& \nabla\cdot (\mathring{A}\nabla \xi^*)=0
 &\quad &\textnormal{in }\Omega^-,
 \label{dual.b}
 \\
 &\p_tw^*+\mathrm{D}_{x'}\cdot(\mathring{v}'w^*)
 -\varepsilon\Delta_{x'}^2 w^*
 -\mathring{b}_1 w^*
 &\quad &
 \nonumber\\
 &\qquad
  -\mathring{h}'\cdot \mathrm{D}_{x'}\xi^*
 +\mathring{b}_2 W_2^*
  -\mathfrak{s}\mathrm{D}_{x'}\cdot
   \big(\mathring{B}\mathrm{D}_{x'}W_2^* \big)=0
 &\quad &\textnormal{on }\Sigma,  \label{dual.c}\\
&(\mathring{A}\nabla\xi^*)_1=
\mathrm{D}_{x'}\cdot(\mathring{h}'W_2^*) +\varepsilon\Delta_{x'} \xi^*-\varepsilon\Delta_{x'}^2 \xi^*
 &\quad &\textnormal{on }\Sigma,  \label{dual.d}\\
 \label{dual.e}
 &
 W_2^*=0\quad \textrm{on }\Sigma^+,\qquad
 \xi^*=0\quad \textrm{on }\Sigma^-,\qquad
 (W^*,\xi^*)|_{{t>T} }=0, &&
  \end{alignat}
\end{subequations}
with $w^*:=W_1^*-\varepsilon W_2^*$.
The conditions \eqref{dual.c}--\eqref{dual.e} are imposed to ensure that
\begin{align*}
  &\int_{\Omega_{T}^+}\left({{\bf L}}_{\varepsilon}W\cdot W^* -W\cdot {{\bf L}}_{\varepsilon}^* W^*\right)
+\int_{\Omega_{T}^-}\left( \xi^* \nabla\cdot (\mathring{A}\nabla \xi)
-\xi\nabla\cdot (\mathring{A}\nabla \xi^*)\right)\\
&
  =
\int_{\Sigma_{T}^+}W_1W_2^*
-\int_{\Sigma_{T}}\left( W_2w^*+W_1W_2^*
-\xi^*(\mathring{A}\nabla\xi)_1+\xi(\mathring{A}\nabla\xi^*)_1\right)
-\int_{\Sigma_{T}^-}\xi^*\p_1\xi
 =0,
\end{align*}
where we have used \eqref{Reg.c}--\eqref{Reg.f}.
Passing then to the back time $\tilde{t}:={T}-t$, we find that
$\widetilde{W}^*(\tilde{t},x):={W}^*(t,x)$ and
$\tilde{\xi}^*(\tilde{t},x):={\xi}^*(t,x)$ satisfy
\begin{subequations} \label{dual2}
  \begin{alignat}{3}
  \label{dual2.a}
  & \bigg(\bm{A}_0\p_t-\sum_{i=1}^{3}\bm{A}_i\p_i+\varepsilon\bm{J}\p_1+\bm{A}_4^{\mathsf{T}}-\sum_{i=0}^{3}\p_i\bm{A}_i \bigg)W^*=\bm{f}^*
  &\quad &\textnormal{in }\Omega^+,\\
  & \nabla\cdot (\mathring{A}\nabla \xi^*)=0
  &\quad &\textnormal{in }\Omega^-,
  \label{dual2.b}
  \\
  &\p_tw^*-\mathrm{D}_{x'}\cdot(\mathring{v}'w^*)
  +\varepsilon\Delta_{x'}^2 w^*
  +\mathring{b}_1 w^*
  &\quad &
  \nonumber\\
  &\qquad
     +\mathring{h}'\cdot \mathrm{D}_{x'}\xi^*
  -\mathring{b}_2 W_2^*
  +\mathfrak{s}\mathrm{D}_{x'}\cdot
  \big(\mathring{B}\mathrm{D}_{x'}W_2^* \big)=0
  &\quad &\textnormal{on }\Sigma,  \label{dual2.c}\\
  &(\mathring{A}\nabla\xi^*)_1=
  \mathrm{D}_{x'}\cdot(\mathring{h}'W_2^*) +\varepsilon\Delta_{x'} \xi^*-\varepsilon\Delta_{x'}^2 \xi^*
  &\quad &\textnormal{on }\Sigma,  \label{dual2.d}\\
  \label{dual2.e}
  &
  W_2^*=0\quad \textrm{on }\Sigma^+,\qquad
  \xi^*=0\quad \textrm{on }\Sigma^-,\qquad
  (W^*,\xi^*)|_{{t<0} }=0, &&
  \end{alignat}
\end{subequations}
where for convenience we have dropped the tildes.
Taking the scalar product of \eqref{dual2.a} with ${W}^*$ and recalling $w^*:=W_1^*-\varepsilon W_2^*$, we use \eqref{decom} and \eqref{dual2.e} to get
\begin{align} \label{dual:es1}
  \int_{\Omega^+} {\bm{A}}_0{W}^*\cdot {W}^*
  +\int_{\Sigma_{t}} \big(\varepsilon |W^*_2|^2+2w^*W^*_2 \big)
  \lesssim_{K} \|({\bm{f}}^*,{W}^*)\|_{L^2(\Omega_{t}^+)}^2.
\end{align}
It follows from \eqref{dual2.b} and \eqref{dual2.d}--\eqref{dual2.e} that
\begin{align}
\nonumber
\int_{\Omega_t^-} \mathring{A}\nabla\xi^*\cdot  \nabla\xi^*
&=\int_{\Omega_t^-} \nabla\cdot(\xi^*\mathring{A}\nabla\xi^*)
=\int_{\Sigma_t} \xi^*(\mathring{A}\nabla\xi^*)_1\\
&=-\int_{\Sigma_t} W_2^*\mathring{h}'\cdot \mathrm{D}_{x'}\xi^*
-\varepsilon\int_{\Sigma_t}\left(|\mathrm{D}_{x'}\xi^*|^2+|\Delta_{x'}\xi^*|^2\right),
\nonumber
\end{align}
from which we have
\begin{align}
\|\nabla\xi^*\|^2_{L^2(\Omega_{t}^-)}
+  \|(\mathrm{D}_{x'}\xi^*,\mathrm{D}_{x'}^2\xi^*)\|^2_{L^2(\Sigma_{t})}
\lesssim_{K,\varepsilon}
\|W_2^*\|^2_{L^2(\Sigma_{t})}.
\label{dual:es2}
\end{align}
Multiplying the boundary condition \eqref{dual2.c} by ${w}^*$ leads to
\begin{align}
\nonumber
&\|{w}^*(t)\|_{L^2(\Sigma)}^2
  +2\varepsilon\|\Delta_{x'}w^*\|_{L^2(\Sigma_{t})}^2\\
&\quad   \leq
\boldsymbol{\epsilon}\varepsilon
\|(\mathrm{D}_{x'}w^*,\mathrm{D}_{x'}^{2}w^*)\|_{L^2(\Sigma_{t})}^2
  +C(K,\boldsymbol{\epsilon}\varepsilon)
  \|(w^*,W^*_2,\mathrm{D}_{x'}\xi^*)\|_{L^2(\Sigma_{t})}^2
 \nonumber
\end{align}
for all $\boldsymbol{\epsilon}>0$.
Substitute
$\|(\mathrm{D}_{x'}w^*,\mathrm{D}_{x'}^{2}w^*)\|_{L^2(\Sigma_{t})}\lesssim
\|(w^*,\Delta_{x'}w^*)\|_{L^2(\Sigma_{t})}$
into the last estimate and take $\boldsymbol{\epsilon}>0$ suitably small
to infer
\begin{align}
\|{w}^*(t)\|_{L^2(\Sigma)}^2
+\|(\mathrm{D}_{x'}w^*,\mathrm{D}_{x'}^{2}w^*)\|_{L^2(\Sigma_{t})}^2
\lesssim_{K,\varepsilon} \|(w^*,W^*_2,\mathrm{D}_{x'}\xi^*)\|_{L^2(\Sigma_{t})}^2.
  \label{dual:es3}
\end{align}
Then we combine \eqref{dual:es1}--\eqref{dual:es3},
utilize \eqref{xi:es1} with $\xi$ replaced by $\xi^*$,  and apply Gr\"{o}nwall's inequality
to obtain
\begin{align}
  \nonumber
&\|W^*(t)\|^2_{L^2(\Omega^+)}
+  \|{w}^*(t)\|_{L^2(\Sigma)}^2
+\|(\xi^*,\nabla\xi^*)\|^2_{L^2(\Omega_{t}^-)}
\\
&\qquad
+  \|(W_2^*,\mathrm{D}_{x'} w^*,\mathrm{D}_{x'}^2 w^*,\mathrm{D}_{x'} \xi^*,\mathrm{D}_{x'}^2 \xi^*)\|^2_{L^2(\Sigma_{t})}
\lesssim_{K,\varepsilon}\|{\bm{f}}^*\|_{L^2(\Omega_{t}^+)}^2.
  \label{dual:es}
\end{align}

With the $\varepsilon$-dependent $L^2$ estimates \eqref{es:0} and \eqref{dual:es}, we can deduce the existence of weak solutions $(W,\xi)\in L^2(\Omega_{T}^+)\times L^2(\Omega_{T}^-)$ to regularization \eqref{Reg}
for any small but fixed parameter $\varepsilon\in (0,1)$ by the standard duality argument in \cite{CP82MR0678605}.
Regarding \eqref{Reg.c} as a fourth-order parabolic equation for $\psi$ with given source term $W_2|_{x_1=0}\in L^2(\Sigma_T)$ and zero initial data $\psi |_{t=0}=0$,
as in \cite[Theorem 5.2]{CP82MR0678605}, we can obtain that the Cauchy problem for this parabolic equation has a unique solution $\psi\in C ([0,T],H^4(\mathbb{T}^2))\cap C^1 ([0,T],L^2(\mathbb{T}^2))$.
Therefore, for any small and fixed parameter $\varepsilon>0$, we obtain the existence of solutions $(W,\xi, \psi )\in L^2(\Omega_T^+)\times  L^2(\Omega_T^-)\times L^2((-\infty ,T];H^4(\mathbb{T}^2))$ to the regularized problem \eqref{Reg}. 

\subsection{Uniform energy estimates}
We now show the uniform-in-$\varepsilon$ high-order energy estimates for solutions to the regularization \eqref{Reg}.
Let $m\geq 1$ be an integer and $\alpha=(\alpha_0,\alpha_1,\alpha_2,\alpha_3,\alpha_4)\in\mathbb{N}^5$ satisfy $\langle \alpha \rangle:=\sum_{i=0}^{3}\alpha_i+2\alpha_4\leq m$.
For clear presentation,  we divide this section into five parts.

\subsubsection{Prelude}
Applying $\mathrm{D}_*^{\alpha}:=\p_t^{\alpha_0}(\sigma \p_1)^{\alpha_1}\p_2^{\alpha_2}\p_3^{\alpha_3}\p_1^{\alpha_4}$
to \eqref{Reg.a} with $\sigma:=x_1(1-x_1)$
and taking the scalar product of the resulting equations
with $\mathrm{D}_*^{\alpha} W$, we utilize \eqref{decom} and \eqref{Reg.f} to deduce
\begin{align}\label{es:HO}
\int_{\Omega^+}{\bm{A}}_0\mathrm{D}_*^{\alpha}W\cdot \mathrm{D}_*^{\alpha}W
+\varepsilon\|\mathrm{D}_*^{\alpha}W_2\|_{L^{2}(\Sigma_t)}^2
={\mathcal{Q}_{\alpha}(t)}
+  \mathcal{R}_{\alpha}(t)
\end{align}
for
\begin{align}
&
\label{Q.alpha}
\mathcal{Q}_{\alpha}(t):=
2\int_{\Sigma_t}\mathrm{D}_*^{\alpha}W_1\mathrm{D}_*^{\alpha}W_2
+\int_{\Sigma_t^+}\left(\varepsilon|\mathrm{D}_*^{\alpha}W_2|^2-2 \mathrm{D}_*^{\alpha}W_1\mathrm{D}_*^{\alpha}W_2 \right),\\
\nonumber  
 & \mathcal{R}_{\alpha}(t):=
\int_{\Omega_t^+}\mathrm{D}_*^{\alpha}W\cdot \bigg(
  2\mathrm{D}_*^{\alpha}(\bm{f}- {\bm{A}}_4 W)
  -\sum_{i=0}^{3}\big(2[\mathrm{D}_*^{\alpha},{\bm{A}}_i\p_i]W
  -\p_i {\bm{A}}_i\mathrm{D}_*^{\alpha}W\big) \bigg).
\end{align}
To estimate the integral $\mathcal{R}_{\alpha}$, we obtain from \eqref{Reg.a} and \eqref{decom} that
\begin{align}
\label{d1W}
  \begin{pmatrix}
  \p_1 W_2\\\p_1 W_1-\varepsilon \p_1W_2 \\0
  \end{pmatrix}
  =\bm{f}-{\bm{A}}_4 W-\sum_{i=0,2,3}{\bm{A}}_i\p_i W-{\bm{A}}_1^{(0)}\p_1W,
\end{align}
where matrix ${\bm{A}}_1^{(0)}$ vanishes on the boundaries $\Sigma_T$ and  $\Sigma_{T}^+$.
Then we can follow the proof of \cite[Lemma 3.5]{TW21MR4201624} and use decomposition \eqref{decom} to infer
\begin{align}
  \label{est:R}
  \mathcal{R}_{\alpha}(t)\lesssim_K
\mathcal{M}_1(t):=
 {\|}(\bm{f},W){\|}_{H_*^m(\Omega_{t}^+)}^2
  +\mathring{\rm C}_{m+4}  \|(\bm{f},W)\|_{W_*^{2,\infty}(\Omega_t^+)}^2,
\end{align}
for all $\langle\alpha\rangle \leq m$, where 
$\mathring{\rm C}_{m}:=1+\|(\mathring{U},\mathring{h},\mathring{\varphi}) \|_{m}^2$ ({\it cf.}~\eqref{norm:ring}) and
\begin{align*}
\|u\|_{W_*^{2,\infty}(\Omega_t^+) }:=\sum_{\langle \alpha\rangle\leq 1}\| \mathrm{D}_*^{\alpha} u\|_{W^{1,\infty}(\Omega_t^+) }.
\end{align*}

\subsubsection{Case $\alpha_1>0$}
Since $\mathcal{Q}_{\alpha}(t)=0$ for $\alpha_1>0$, we plug \eqref{est:R} into \eqref{es:HO} to get
\begin{align}
  \sum_{\langle\alpha\rangle\leq m,\, \alpha_1>0}
  \|\mathrm{D}_*^{\alpha}W(t)\|_{L^2(\Omega^+)}^2
  \lesssim_K \mathcal{M}_1(t),
  \label{es:HO1}
\end{align}
where $\mathcal{M}_1(t)$ is given in \eqref{est:R}.

\subsubsection{Case $\alpha_1=0$ and $\alpha_4>0$}
It follows from the identity \eqref{d1W} that
\begin{align}
  \mathcal{Q}_{\alpha}(t)\lesssim
  \sum_{i=0,2,3}\|\mathrm{D}_*^{\alpha-{\mathbf{e}}} (\bm{f},  {\bm{A}}_4 W,{\bm{A}}_i\p_i W,{\bm{A}}_1^{(0)}\p_1W)\|_{L^2(\Sigma_t\cup \Sigma_t^+)}^2
\label{es:HO2a}
\end{align}
for ${\mathbf{e}}:=(0,0,0,0,1).  $
Use the trace theorem ({\it cf.}~\cite{OSY94MR1289186}) and the Moser-type calculus inequalities ({\it cf.}~\cite[Theorem B.3]{MST09MR2604255}) for anisotropic Sobolev spaces to obtain
\begin{align}
&\|\mathrm{D}_*^{\alpha-{\mathbf{e}}} (\bm{f},  {\bm{A}}_4 W )\|_{L^2(\Sigma_t\cup \Sigma_t^+)}^2
\lesssim_{K} \mathcal{M}_1(t),
\label{es:HO2b}
 \end{align} and
\begin{align}
&\|\mathrm{D}_*^{\alpha-{\mathbf{e}}} ( {\bm{A}}_i\p_i W )\|_{L^2(\Sigma_t\cup \Sigma_t^+)}^2
 \lesssim_{K}
\sum_{0<\beta\leq \alpha-\mathbf{e}}
\|(\p_i \mathrm{D}_*^{\alpha-{\mathbf{e}}} W,\,\mathrm{D}_*^{ \beta}{\bm{A}}_i \mathrm{D}_*^{\alpha-{\mathbf{e}}-\beta} \p_i W) \|_{L^2(\Sigma_t\cup \Sigma_t^+)}^2
\nonumber\\
&\qquad \lesssim_{K}
\| \mathrm{D}_*^{\alpha-{\mathbf{e}}} W \|_{H^1(\Sigma_t\cup \Sigma_t^+)}^2
+\sum_{0<\beta\leq \alpha-\mathbf{e}}
\|\mathrm{D}_*^{ \beta}{\bm{A}}_i \mathrm{D}_*^{\alpha-{\mathbf{e}}-\beta} \p_i W \|_{H_*^2(\Omega_{t}^+)}^2
\nonumber\\
&\qquad \lesssim_{K}  \|W\|_{H_*^m(\Omega_{t}^+)}^2
+\mathring{\rm C}_{m+4} \|W\|_{W_*^{2,\infty}(\Omega_{t}^+ )}^2
\lesssim_{K} \mathcal{M}_1(t)
\qquad \textrm{for }i=0,2,3.
\label{es:HO2c}
\end{align}
Since $(\mathrm{D}_*^{\beta}{\bm{A}}_1^{(0)})|_{\Sigma_T\cup \Sigma_T^+}=0$ for $\beta=(\beta_0,\beta_1,\beta_2,\beta_3,\beta_4)\in\mathbb{N}^5$ with $\beta_4=0$,
we derive
 \begin{align}
 \nonumber  &\|\mathrm{D}_*^{\alpha-{\mathbf{e}}} ({\bm{A}}_1^{(0)}\p_1W)\|_{L^2(\Sigma_t\cup \Sigma_t^+)}^2
 \lesssim
 \sum_{{\mathbf{e}}\leq \beta\leq \alpha-{\mathbf{e}} } \|\mathrm{D}_*^{\beta} {\bm{A}}_1^{(0)}\mathrm{D}_*^{\alpha-{\mathbf{e}}-\beta} \p_1W\|_{L^2(\Sigma_t\cup \Sigma_t^+)}^2   \\
 &\qquad \lesssim
 \sum_{{\mathbf{e}}\leq \beta\leq \alpha-{\mathbf{e}} } {\|}\mathrm{D}_*^{\beta-{\mathbf{e}}}(\p_1 {\bm{A}}_1^{(0)})\mathrm{D}_*^{\alpha-\beta} W{\|}_{H_*^2(\Omega_{t}^+)}^2
 \lesssim_{K} \mathcal{M}_1(t).
\label{es:HO2d}
\end{align}
Substituting \eqref{est:R} and \eqref{es:HO2a}--\eqref{es:HO2d} into \eqref{es:HO} implies
\begin{align}
 \sum_{\substack{\langle\alpha\rangle \leq m,\, \alpha_1=0,\,\alpha_4>0}}
  \left(
  \|\mathrm{D}_*^{\alpha}W(t)\|_{L^2(\Omega^+)}^2
  +\varepsilon\|\mathrm{D}_*^{\alpha}W_2\|_{L^2(\Sigma_{t})}^2
  \right)
  \lesssim_K \mathcal{M}_1(t).
  \label{es:HO2}
\end{align}

\subsubsection{Case $\alpha_1=\alpha_4=0$}\label{sec:tangential}
We have $\mathrm{D}_*^{\alpha}=\p_t^{\alpha_0}\p_2^{\alpha_2}\p_2^{\alpha_3}$ and $\alpha_{0}+\alpha_2+\alpha_3\leq m$.
It follows from the boundary conditions \eqref{Reg.f} and \eqref{Reg.d} that
\begin{align}
\mathcal{Q}_{\alpha}(t)=2\int_{\Sigma_{t}} \mathrm{D}_*^{\alpha}W_1\mathrm{D}_*^{\alpha}W_2
=\sum_{i=1}^{4}\int_{\Sigma_{t}}  Q_i
,
\label{Q.alpha:id}
\end{align}
where
\begin{alignat}{3}
&Q_1:= 2 [\mathrm{D}_*^{\alpha}, \mathring{h}'\cdot\mathrm{D}_{x'}]\xi  \mathrm{D}_*^{\alpha} W_2,
\qquad
&&Q_2:= 2 \mathring{h}'\cdot\mathrm{D}_{x'}\mathrm{D}_*^{\alpha}\xi  \mathrm{D}_*^{\alpha} W_2,
\label{Q1Q2}\\
&Q_3:=2\mathfrak{s}  \mathrm{D}_*^{\alpha} \mathrm{D}_{x'}\cdot
\big(\mathring{B}\mathrm{D}_{x'}\psi \big) \mathrm{D}_*^{\alpha} W_2,
\qquad && Q_4:=
-2 \mathrm{D}_*^{\alpha}(\mathring{b}_2 \psi) \mathrm{D}_*^{\alpha} W_2
.
\label{Q3Q4}
\end{alignat}
Let us present the estimates for $Q_i$ in the following four steps.

\vspace{3mm}
\noindent{\bf Step 1:  estimate for $Q_1$.} \
Passing to the volume integral and using \eqref{Reg.f} yield
\begin{align}
  \int_{\Sigma_{t}} Q_1
  =\, &
  -2 \int_{\Omega_{t}^+}\p_1\Big([\mathrm{D}_*^{\alpha}, \mathring{h}'_{\sharp}\cdot\mathrm{D}_{x'}]\xi_{\sharp}  \mathrm{D}_*^{\alpha} W_2\Big)\nonumber \\
  =\,&
  \underbrace{-2 \int_{\Omega_{t}^+}\p_1 [\mathrm{D}_*^{\alpha}, \mathring{h}'_{\sharp}\cdot\mathrm{D}_{x'}]\xi_{\sharp}  \mathrm{D}_*^{\alpha} W_2}_{\mathcal{Q}_{1a}}
  \underbrace{-2\int_{\Omega_{t}^+} [\mathrm{D}_*^{\alpha}, \mathring{h}'_{\sharp}\cdot\mathrm{D}_{x'}]\xi_{\sharp}  \mathrm{D}_*^{\alpha} \p_1W_2}_{\mathcal{Q}_{1b}},
  \label{Q1:es1}
\end{align}
where we denote $\mathring{h}'_{\sharp}(t,x_1,x'):=\mathring{h}'(t,-x_1,x')$ and $\xi_{\sharp}(t,x_1,x'):=\xi(t,-x_1,x')$.
It follows directly from Cauchy's inequality that
\begin{align}
  \mathcal{Q}_{1a}
  \lesssim
  \| \p_1[\mathrm{D}_*^{\alpha}, \mathring{\rm c}_0  ]\mathrm{D}_{x'} \xi \|_{L^2(\Omega_{t}^-)}^2
  +\|W_2 \|_{H_*^m(\Omega_t^+)}^2.
  \label{Q1a:es}
\end{align}
If $\langle\alpha\rangle\leq m-1$,  then the integral $\mathcal{Q}_{1b}$ can be estimated as
\begin{align}
  \mathcal{Q}_{1b}
  \lesssim
 \|   [\mathrm{D}_*^{\alpha}, \mathring{\rm c}_0 ] \mathrm{D}_{x'}\xi \|_{L^2(\Omega_{t}^-)}^2
 +\|\mathrm{D}_*^{\alpha}\p_1 W_2 \|_{L^2(\Omega_t^+)}^2.
\label{Q1b:es1}
\end{align}
If $\langle\alpha\rangle= m$, then we choose $\beta<\alpha$ with $\langle \beta\rangle=m-1$ and employ integration by parts to derive
\begin{align}
  \mathcal{Q}_{1b}
  \lesssim\,&
\int_{\Omega^+}  \left|[\mathrm{D}_*^{\alpha}, \mathring{h}'_{\sharp}\cdot\mathrm{D}_{x'}]\xi_{\sharp}  \mathrm{D}_*^{\beta}\p_1 W_2 \right|
+ \|   [\mathrm{D}_*^{\alpha}, \mathring{h}' \cdot\mathrm{D}_{x'}] \xi\|_{H^1(\Omega_{t}^-)}^2
+ \|\mathrm{D}_*^{\beta}\p_1 W_2 \|_{L^2(\Omega_t^+)}^2
\nonumber
\\
  \lesssim\,&
  \boldsymbol{\epsilon}  \|\mathrm{D}_*^{\beta}\p_1 W_2 \|_{L^2(\Omega^+)}^2
  +C(\boldsymbol{\epsilon})\|  [\mathrm{D}_*^{\alpha}, \mathring{\rm c}_0]\mathrm{D}_{x'}\xi\|_{H^1(\Omega_{t}^-)}^2
+ \|\mathrm{D}_*^{\beta}\p_1 W_2 \|_{L^2(\Omega_t^+)}^2
\label{Q1b:es2}
\end{align}
for all $\boldsymbol{\epsilon}>0$.
To estimate the last terms in \eqref{Q1b:es1}--\eqref{Q1b:es2},
we compute from identities \eqref{d1W} and \eqref{decom} that
for all $\gamma=(\gamma_0,0,\gamma_2,\gamma_3,0)$ with $\langle \gamma\rangle\leq m-1$,
\begin{align}
  \nonumber
  \|\mathrm{D}_*^{\gamma}\p_1 W_2(t)\|_{L^2(\Omega^+)}^2
  \lesssim_K \;&
  \sum_{\langle \beta\rangle\leq m} \| \mathrm{D}_*^{\beta} W(t)\|_{L^2(\Omega^+)}^2
  +  \|(\bm{f},{\bm{A}}_4 W)\|_{H_*^m(\Omega_t^+)}^2 \\
  &+\sum_{i=0,2,3}\big\| \big([\mathrm{D}_*^{\gamma},{\bm{A}}_i\p_i] W ,[\mathrm{D}_*^{\gamma},{\bm{A}}_1^{(0)}\p_1] W \big)(t)\big\|_{L^2(\Omega^+)}^2
  \nonumber\\
  \lesssim_K \;&
  \sum_{\langle \beta\rangle\leq m} \| \mathrm{D}_*^{\beta} W(t)\|_{L^2(\Omega^+)}^2
  +\mathcal{M}_1(t),
  \label{d1W:es1}\\[0.5mm]
\|\mathrm{D}_*^{\gamma}\p_1 W_2 \|_{L^2(\Omega_t^+)}^2
\lesssim_K \;&
\mathcal{M}_1(t),
\quad
\label{d1W:es2}
\end{align}
where $\mathcal{M}_1(t)$ is defined in \eqref{est:R}.
Plugging \eqref{Q1a:es}--\eqref{Q1b:es2} into \eqref{Q1:es1},
we use \eqref{d1W:es1}--\eqref{d1W:es2} and the Moser-type calculus inequalities to discover
\begin{align}
\int_{\Sigma_t} Q_1
  \lesssim_K
  \boldsymbol{\epsilon}\sum_{\langle \beta\rangle\leq m} \| \mathrm{D}_*^{\beta} W(t)\|_{L^2(\Omega^+)}^2
  +C(\boldsymbol{\epsilon}) \mathcal{M}_2(t)
  +C(\boldsymbol{\epsilon}) \mathcal{M}_1(t)
  \label{Q1:es}
\end{align}
for all $\boldsymbol{\epsilon}>0$,
where 
\begin{align}
  \mathcal{M}_2(t):=\|\nabla \xi \|_{H^m(\Omega_t^-)}^2
  +\mathring{\mathrm{C}}_{m+4} \|\nabla \xi\|_{L^{\infty}(\Omega_t^-)}^2.
  \label{M2.cal}
\end{align}

\vspace{1mm}
\noindent{\bf Step 2:  estimate for $Q_2$.} \
For $Q_2$ defined in \eqref{Q1Q2},
from \eqref{Reg.c}, we have
\begin{align}
Q_2
  =\, &\underbrace{2
    \mathrm{D}_{x'}\mathrm{D}_*^{\alpha} \xi\cdot \mathrm{D}_*^{\alpha}\p_t (\mathring{h}'\psi )
  }_{Q_{2a}}
  +\underbrace{2
    \mathring{h}'\cdot \mathrm{D}_{x'} \mathrm{D}_*^{\alpha} \xi
    (\mathring{v}'\cdot \mathrm{D}_{x'} ) \mathrm{D}_*^{\alpha}  \psi
  }_{Q_{2b}}
+\underbrace{2\varepsilon
  \mathring{h}'\cdot \mathrm{D}_{x'} \mathrm{D}_*^{\alpha} \xi\;\!
  \mathrm{D}_*^{\alpha}\Delta_{x'}^2\psi
}_{Q_{2c}}
  \nonumber \\[0.5mm]
  & \underbrace{-2
\mathrm{D}_{x'}\mathrm{D}_*^{\alpha} \xi\cdot [\mathrm{D}_*^{\alpha}\p_t ,\mathring{h}']\psi
+
2
\mathring{h}'\cdot \mathrm{D}_{x'} \mathrm{D}_*^{\alpha} \xi
 \left\{[\mathrm{D}_*^{\alpha},\mathring{v}'\cdot \mathrm{D}_{x'}]  \psi
 +\mathrm{D}_*^{\alpha}(\mathring{b}_1\psi) \right\}
  }_{Q_{2d}} .
\label{Q2:es1}
\end{align}
In view of the boundary condition \eqref{Reg.e}, we find
\begin{align}
\int_{\Sigma_{t}}  {Q}_{2a}
=  \underbrace{-2\int_{\Sigma_{t}}\mathrm{D}_*^{\alpha}\xi
 \mathrm{D}_*^{\alpha}\p_t(\mathring{A}\nabla\xi)_1}_{\mathcal{J}_1}
+
\underbrace{2\varepsilon \int_{\Sigma_{t}}\mathrm{D}_*^{\alpha}\xi
 \mathrm{D}_*^{\alpha}\p_t(\Delta_{x'}\xi -\Delta_{x'}^2\xi) }_{\mathcal{J}_2}.
\label{Q2a:es1}
\end{align}
Pass $\mathcal{J}_1$ to the volume integral and use the elliptic equation \eqref{Reg.b} to derive
\begin{align}
\nonumber
  \mathcal{J}_{1}
  =\, &-2   \int_{\Omega_{t}^-} \nabla\cdot\left( \mathrm{D}_*^{\alpha} \p_t  (\mathring{A}\nabla\xi)   \mathrm{D}_*^{\alpha}  \xi\right)
  =-2   \int_{\Omega_{t}^-} \p_t  \mathrm{D}_*^{\alpha} (\mathring{A}\nabla\xi)   \cdot  \mathrm{D}_*^{\alpha} \nabla \xi \\
=\, &
-\int_{\Omega^-} \mathring{A}\mathrm{D}_*^{\alpha}\nabla\xi\cdot\mathrm{D}_*^{\alpha}\nabla\xi
+\int_{\Omega_{t}^-}\Big(\p_t\mathring{A}\mathrm{D}_*^{\alpha}\nabla\xi
-2 [\p_t\mathrm{D}_*^{\alpha}, \mathring{A}] \nabla\xi  \Big)\cdot \mathrm{D}_*^{\alpha}\nabla\xi.
\nonumber 
\end{align}
Regarding term $\mathcal{J}_{2}$, we see that it is a good term, since
\begin{align}
  \mathcal{J}_{2}
=-  \varepsilon \int_{\Sigma}
(|\mathrm{D}_*^{\alpha}\mathrm{D}_{x'}\xi|^2+
|\mathrm{D}_*^{\alpha}\Delta_{x'}\xi|^2)
  \leq  -\varepsilon
  \|(\mathrm{D}_*^{\alpha}\mathrm{D}_{x'}\xi,\mathrm{D}_*^{\alpha}\mathrm{D}_{x'}^2\xi)(t)
  \|_{L^2(\Sigma)}^2.
 \label{J2:es}
\end{align}
Substituting the above estimates for $\mathcal{J}_{1}$ and $\mathcal{J}_{2}$ into \eqref{Q2a:es1} yields
\begin{align}
\int_{\Sigma_{t}}  {Q}_{2a}
+
\int_{\Omega^-} \mathring{A}\mathrm{D}_*^{\alpha}\nabla\xi\cdot\mathrm{D}_*^{\alpha}\nabla\xi
+
\varepsilon
\|(\mathrm{D}_*^{\alpha}\mathrm{D}_{x'}\xi,\mathrm{D}_*^{\alpha}\mathrm{D}_{x'}^2\xi)(t)
\|_{L^2(\Sigma)}^2
\lesssim_K \mathcal{M}_2(t),
\label{Q2a:es2}
\end{align}
where $\mathcal{M}_2(t)$ is defined by \eqref{M2.cal}.
For term $Q_{2b}$ given in \eqref{Q2:es1},
we use identity \eqref{h.v:id1} with $\xi$ and $\psi$ replaced respectively by $\mathrm{D}_*^{\alpha}\xi$ and $\mathrm{D}_*^{\alpha}\psi$ to get
\begin{align}
\int_{\Sigma_{t}} Q_{2b}
=
\underbrace{2 \int_{\Sigma_{t}}\mathring{v}'\cdot\mathrm{D}_{x'}\mathrm{D}_*^{\alpha}\xi \;\!
\mathrm{D}_*^{\alpha}\mathrm{D}_{x'} \cdot ( \mathring{h}'\psi)}_{\mathcal{J}_3}
+\mathcal{J}_4+\mathcal{J}_5,
\label{Q2b:es1}
\end{align}
where
\begin{align} \nonumber
\mathcal{J}_4
:=-2\sum_{i=2,3} \int_{\Sigma_{t}}\mathring{v}'\cdot\mathrm{D}_{x'}\mathrm{D}_*^{\alpha}\xi \;\!
[\mathrm{D}_*^{\alpha}\p_i, \mathring{h}_i] \psi,\quad
\mathcal{J}_5:=
\int_{\Sigma_{t}}
\mathring{\rm c}_1\mathrm{D}_*^{\alpha}\psi
\mathrm{D}_*^{\alpha}\mathrm{D}_{x'}\xi.
\end{align}
In light of \eqref{Reg.e}, we obtain
\begin{align}
\mathcal{J}_3
=
\underbrace{2 \int_{\Sigma_{t}}\mathring{v}'\cdot\mathrm{D}_{x'}\mathrm{D}_*^{\alpha}\xi
 \mathrm{D}_*^{\alpha}(\mathring{A}\nabla\xi)_1}_{\mathcal{J}_{3a}}
\underbrace{-2\varepsilon \int_{\Sigma_{t}}
 \mathring{v}'\cdot\mathrm{D}_{x'}\mathrm{D}_*^{\alpha}\xi
 \mathrm{D}_*^{\alpha}(\Delta_{x'}-\Delta_{x'}^2)\xi }_{\mathcal{J}_{3b}}.
\label{J3:es1}
\end{align}
It follows from equation \eqref{Reg.b} and integration by parts that
\begin{align}
\mathcal{J}_{3a}
=\, &2\int_{\Omega_{t}^{-}}  \nabla \cdot \left(\mathring{v}_{\sharp}'\cdot\mathrm{D}_{x'}\mathrm{D}_*^{\alpha}\xi  \mathrm{D}_*^{\alpha}(\mathring{A}\nabla\xi) \right)
=2\int_{\Omega_{t}^{-}}  \nabla  \left(\mathring{v}_{\sharp}'\cdot\mathrm{D}_{x'}\mathrm{D}_*^{\alpha}\xi \right)
\cdot \mathrm{D}_*^{\alpha}(\mathring{A}\nabla\xi)
\nonumber\\
=\, &
\int_{\Omega_{t}^{-}}
\mathring{\rm c}_{2} \nabla \mathrm{D}_*^{\alpha}\xi
\cdot \left\{\mathrm{D}_*^{\alpha}(\mathring{\rm c}_1\nabla\xi )
+[\mathrm{D}_{x'}\mathrm{D}_*^{\alpha}, \mathring{\rm c}_1]\nabla\xi  \right\}
\lesssim_K \mathcal{M}_2(t),
\label{J3a:es}
\end{align}
where we denote $\mathring{v}'_{\sharp}(t,x_1,x'):=\mathring{v}'(t,-x_1,x')$
and $\mathcal{M}_2(t)$ is defined by \eqref{M2.cal}.
And
\begin{align}
\mathcal{J}_{3b}
&  =
2\varepsilon  \int_{\Sigma_{t}}
\left\{
\mathrm{D}_{x'} (\mathring{v}'\cdot\mathrm{D}_{x'}) \mathrm{D}_*^{\alpha}\xi   \cdot
\mathrm{D}_{x'}  \mathrm{D}_*^{\alpha} \xi
+
\Delta_{x'} (\mathring{v}'\cdot\mathrm{D}_{x'}) \mathrm{D}_*^{\alpha}\xi
\Delta_{x'}  \mathrm{D}_*^{\alpha} \xi \right\}
\nonumber\\
 & 
\lesssim_K
\varepsilon
\|(\mathrm{D}_{x'}   \xi, \mathrm{D}_{x'}^2   \xi) \|_{H^m(\Sigma_{t})}^2.
\label{J3b:es}
\end{align}
For the terms $\mathcal{J}_4,\mathcal{J}_5$ and the integral of $Q_{2d}$ defined in \eqref{Q2:es1},
we use the Moser-type calculus inequalities
and \eqref{xi:es1} with $\xi$ replaced by $\mathrm{D}_*^{\alpha}\xi$
to infer
\begin{align}
\mathcal{J}_4+\mathcal{J}_5
+ \int_{\Sigma_{t}} Q_{2d}
&=
\sum_{i=0,2,3}
\int_{\Sigma_{t}}
\left\{  \mathring{\rm c}_1[\mathrm{D}_*^{\alpha}\p_i, \mathring{\rm c}_0]\psi
+\mathring{\rm c}_1\mathrm{D}_*^{\alpha}(\mathring{\rm c}_1\psi)
\right\} \mathrm{D}_{x'}\mathrm{D}_*^{\alpha}\xi
\nonumber
\\
&=
\sum_{i=0,2,3}\int_{\Sigma_{t}}
\mathrm{D}_{x'}\cdot\left\{  \mathring{\rm c}_1[\mathrm{D}_*^{\alpha}\p_i, \mathring{\rm c}_0]\psi
+\mathring{\rm c}_1\mathrm{D}_*^{\alpha}(\mathring{\rm c}_1\psi)
\right\}\mathrm{D}_*^{\alpha}\xi
\nonumber \\[1.5mm]
&\lesssim_K    \mathcal{M}_3(t)+\|\mathrm{D}_*^{\alpha}\xi\|_{L^2(\Sigma_t)}^2
\lesssim_K \mathcal{M}_3(t)+\mathcal{M}_2(t),
\label{Q2d:es}
\end{align}
where 
\begin{align}
 \mathcal{M}_3(t):=\|( \psi,\mathrm{D}_{x'}\psi)\|_{H^m(\Sigma_t)}^2 +\mathring{\mathrm{C}}_{m+4} \|( \psi,\mathrm{D}_{x'}\psi)\|_{L^{\infty}(\Sigma_t)}^2.
 \label{M3.cal}
\end{align}
The integral of $Q_{2c}$ ({\it cf.}~\eqref{Q2:es1}) can be estimated as
\begin{align}
\int_{\Sigma_{t}} Q_{2c}
=\, &
-2\varepsilon  \int_{\Sigma_{t}}
\mathrm{D}_{x'}(\mathring{h}'\cdot \mathrm{D}_{x'} )\mathrm{D}_*^{\alpha} \xi\cdot
\mathrm{D}_*^{\alpha}\Delta_{x'}\mathrm{D}_{x'} \psi
\nonumber\\
\leq\, & \boldsymbol{\epsilon} \varepsilon \| \mathrm{D}_{x'}^3 \psi\|_{H^m(\Sigma_{t})}^2
+C(K, \boldsymbol{\epsilon} ) \varepsilon
\|(\mathrm{D}_{x'}    \xi, \mathrm{D}_{x'}^2    \xi) \|_{H^m(\Sigma_{t})}^2
\ \ \textrm{for all }\boldsymbol{\epsilon}>0.
\label{Q2c:es}
\end{align}
In view of decomposition \eqref{Q2:es1}, we combine \eqref{Q2a:es2}--\eqref{Q2d:es} with \eqref{Q2c:es} to get
\begin{align}
&\int_{\Sigma_{t}} Q_2
+
\int_{\Omega^-} \mathring{A}\mathrm{D}_*^{\alpha}\nabla\xi\cdot\mathrm{D}_*^{\alpha}\nabla\xi
+
\varepsilon
\|(\mathrm{D}_*^{\alpha}\mathrm{D}_{x'}\xi,\mathrm{D}_*^{\alpha}\mathrm{D}_{x'}^2\xi)(t)
\|_{L^2(\Sigma)}^2 \nonumber \\
&\quad \lesssim_K \mathcal{M}_2(t)+\mathcal{M}_3(t)
+\boldsymbol{\epsilon} \varepsilon
\| \mathrm{D}_{x'}^3 \psi\|_{H^m(\Sigma_{t})}^2
+C(K, \boldsymbol{\epsilon} ) \varepsilon
\|(\mathrm{D}_{x'}    \xi, \mathrm{D}_{x'}^2   \xi) \|_{H^m(\Sigma_{t})}^2
\label{Q2:es}
\end{align}
for all $\boldsymbol{\epsilon}>0$, where $\mathcal{M}_2(t)$ and $\mathcal{M}_3(t)$ are defined in \eqref{M2.cal} and \eqref{M3.cal}, respectively.

\vspace{3mm}
\noindent{\bf Step 3:  estimate for $Q_3$.} \
Next we consider the integral of $Q_3$ defined in \eqref{Q3Q4}.
Thanks to the boundary condition \eqref{Reg.c}, we infer
\begin{align}
\int_{\Sigma_{t}} Q_3
=
\underbrace{-2\mathfrak{s}\int_{\Sigma_{t}}  \mathrm{D}_*^{\alpha}
 \big(\mathring{B}\mathrm{D}_{x'}\psi \big) \cdot
 (\p_t+\mathring{v}'\cdot\mathrm{D}_{x'} )\mathrm{D}_*^{\alpha}\mathrm{D}_{x'}\psi}_{\mathcal{Q}_{3a}}+\mathcal{Q}_{3b}
+\mathcal{Q}_{3c},
\label{Q3:es1}
\end{align}
where 
\begin{align}
\nonumber
&\mathcal{Q}_{3b}:=
-2\mathfrak{s}\int_{\Sigma_{t}}  \mathrm{D}_*^{\alpha}
\big(\mathring{B}\mathrm{D}_{x'}\psi \big) \cdot
\Big\{ [\mathrm{D}_*^{\alpha}\mathrm{D}_{x'}, \mathring{v}'\cdot\mathrm{D}_{x'} ]\psi
+\mathrm{D}_*^{\alpha}\mathrm{D}_{x'}(\mathring{b}_1\psi)
 \Big\},\\
 &\mathcal{Q}_{3c}:=
 - 2\mathfrak{s}\varepsilon \int_{\Sigma_{t}} \Delta_{x'} \mathrm{D}_*^{\alpha}
 \big(\mathring{B}\mathrm{D}_{x'}\psi \big) \cdot
 \Delta_{x'}\mathrm{D}_*^{\alpha}\mathrm{D}_{x'}\psi  .
\nonumber
\end{align}
We have derived in \cite[\S 2.4]{TW21b} the estimate for $\mathcal{Q}_{3a}$ and $\mathcal{Q}_{3b}$ (denoted respectively as $\mathcal{Q}_{\alpha}^{(2)}(t)$ and $\mathcal{Q}_{\alpha}^{(4)}(t)$ therein),
which reads as
\begin{align}
\mathcal{Q}_{3a}
+\mathcal{Q}_{3b}
\leq  -\frac{\mathfrak{s}}{2}\int_{\Sigma} \frac{|\mathrm{D}_*^{\alpha}\mathrm{D}_{x'}\psi|^2}{|\mathring{N}|^3}
+C(K) \mathcal{M}_3(t).
\label{Q3a:es}
\end{align}
Noting that $\mathring{B}$ defined in \eqref{B.ring:def} is positive definite, we apply the Cauchy and Moser-type calculus inequalities to have
\begin{align}
\mathcal{Q}_{3c}
\leq\, &  -\mathfrak{s}\varepsilon \int_{\Sigma_{t}}
\mathring{B}  \mathrm{D}_*^{\alpha} \Delta_{x'}\mathrm{D}_{x'}\psi   \cdot
 \mathrm{D}_*^{\alpha} \Delta_{x'}\mathrm{D}_{x'}\psi
 +\varepsilon C(K)
 \|[\mathrm{D}_*^{\alpha} \Delta_{x'},\mathring{B}]\mathrm{D}_{x'}\psi \|_{L^2(\Sigma_{t})}^2\nonumber\\
\leq  \, &  -\mathfrak{s}\varepsilon \int_{\Sigma_{t}}
\mathring{B}  \mathrm{D}_*^{\alpha} \Delta_{x'}\mathrm{D}_{x'}\psi   \cdot
\mathrm{D}_*^{\alpha} \Delta_{x'}\mathrm{D}_{x'}\psi
+\varepsilon C(K)  \mathcal{M}_4(t),
\label{Q3c:es}
\end{align}
where
$$
 \mathcal{M}_4(t):=\|(\mathrm{D}_{x'}\psi, \mathrm{D}_{x'}^2\psi)\|_{H^m(\Sigma_t)}^2 +\mathring{\mathrm{C}}_{m+4} \|(\mathrm{D}_{x'}\psi,\mathrm{D}_{x'}^2\psi)\|_{L^{\infty}(\Sigma_t)}^2.
$$
Utilizing \eqref{Q3:es1}--\eqref{Q3c:es} and
$\|\mathrm{D}_{x'}^2\psi\|_{H^m(\Sigma_t)}^2\lesssim
 \|\mathrm{D}_{x'}\psi\|_{H^m(\Sigma_t)} \|\mathrm{D}_{x'}^3\psi\|_{H^m(\Sigma_t)} $
leads to
\begin{multline}
\int_{\Sigma_{t}} Q_3
+\frac{\mathfrak{s}}{2}\int_{\Sigma} \frac{|\mathrm{D}_*^{\alpha}\mathrm{D}_{x'}\psi|^2}{|\mathring{N}|^3}
+\mathfrak{s}\varepsilon \int_{\Sigma_{t}}
\mathring{B}  \mathrm{D}_*^{\alpha} \Delta_{x'}\mathrm{D}_{x'}\psi   \cdot
\mathrm{D}_*^{\alpha} \Delta_{x'}\mathrm{D}_{x'}\psi
 \\[1mm]
\lesssim_K
\boldsymbol{\epsilon} \varepsilon \|\mathrm{D}_{x'}^3\psi\|_{H^m(\Sigma_t)}^2
+C(\boldsymbol{\epsilon})
\mathcal{M}_3(t)
+\varepsilon\mathring{\mathrm{C}}_{m+4} \|\mathrm{D}_{x'}^2\psi\|_{L^{\infty}(\Sigma_t)}^2
\label{Q3:es}
\end{multline}
for all $\boldsymbol{\epsilon}>0$, where $\mathcal{M}_3(t)$ is defined in \eqref{M3.cal}.

Plugging \eqref{Q1:es}, \eqref{Q2:es}, and \eqref{Q3:es} into \eqref{Q.alpha:id},
we use \eqref{es:HO} and \eqref{est:R} to get
\begin{align}
\nonumber
&\|\mathrm{D}_*^{\alpha}W(t)\|_{L^2(\Omega^+)}^2
+\|\mathrm{D}_*^{\alpha}\nabla\xi(t)\|^2_{L^2(\Omega^-)}
+\|\mathrm{D}_*^{\alpha}\mathrm{D}_{x'}\psi(t)\|^2_{L^2(\Sigma)}
\\[2mm]
\nonumber
&\qquad
+\varepsilon\|(\mathrm{D}_*^{\alpha}W_2, \mathrm{D}_*^{\alpha} \mathrm{D}_{x'}^3\psi)\|_{L^2(\Sigma_{t})}^2
+\varepsilon
\|(\mathrm{D}_*^{\alpha}\mathrm{D}_{x'}\xi,\mathrm{D}_*^{\alpha}\mathrm{D}_{x'}^2\xi)(t)
\|_{L^2(\Sigma)}^2
 \\[2mm]
&   \lesssim_K
C(\boldsymbol{\epsilon}) \mathcal{M}(t)
+\varepsilon\mathring{\mathrm{C}}_{m+4} \|\mathrm{D}_{x'}^2\psi\|_{L^{\infty}(\Sigma_t)}^2
+C(K, \boldsymbol{\epsilon} ) \varepsilon \|(\mathrm{D}_{x'}    \xi, \mathrm{D}_{x'}^2   \xi) \|_{H^m(\Sigma_{t})}^2
\nonumber\\[1mm]
&  \qquad
+\boldsymbol{\epsilon}\varepsilon \| \mathrm{D}_{x'}^3 \psi\|_{H^m(\Sigma_{t})}^2
+\boldsymbol{\epsilon}\sum_{\langle \beta\rangle\leq m} \| \mathrm{D}_*^{\beta} W(t)\|_{L^2(\Omega^+)}^2
+\left|\int_{\Sigma_{t}} Q_4\right|
\label{es:HO3a}
\end{align}
for $\alpha_1=\alpha_4=0$, where
\begin{align}
\mathcal{M}(t):=\mathcal{M}_1(t)+\mathcal{M}_2(t)+\mathcal{M}_3(t).
\label{M.cal}
\end{align}

\vspace{1mm}
\noindent{\bf Step 4:  estimate for $Q_4$.} \
Let us now consider the final term $Q_4$ given in \eqref{Q3Q4}.
Utilize \eqref{Reg.c} to get
\begin{align}
 \int_{\Sigma_{t}} Q_4
 =\underbrace{-2 \int_{\Sigma_{t}}  \mathrm{D}_*^{\alpha}
  \big(\mathring{b}_2 \psi \big) \cdot
  (\p_t+\mathring{v}'\cdot\mathrm{D}_{x'} )\mathrm{D}_*^{\alpha} \psi}_{\mathcal{Q}_{4a}}+\mathcal{Q}_{4b}
+\mathcal{Q}_{4c},
 \label{Q4:es1}
\end{align}
where
\begin{align}
 \nonumber
 &\mathcal{Q}_{4b}:=
 -2 \int_{\Sigma_{t}}  \mathrm{D}_*^{\alpha}
 \big(\mathring{b}_2 \psi \big) \cdot
 \Big\{ [\mathrm{D}_*^{\alpha}, \mathring{v}'\cdot\mathrm{D}_{x'} ]\psi
 +\mathrm{D}_*^{\alpha}(\mathring{b}_1\psi)
 \Big\},\\
 &\mathcal{Q}_{4c}:=
2 \varepsilon \int_{\Sigma_{t}} \mathrm{D}_{x'} \mathrm{D}_*^{\alpha}
\big(\mathring{b}_2 \psi \big) \cdot
\mathrm{D}_*^{\alpha}\mathrm{D}_{x'}\Delta_{x'}\psi.
 \nonumber
\end{align}
The estimate for $\mathcal{Q}_{4a}$ and $\mathcal{Q}_{4b}$ can be obtained similar to that for the integrals $\mathcal{Q}_{\alpha}^{(1)}(t)$ and $\mathcal{Q}_{\alpha}^{(3)}(t)$ in \cite[\S 2.4]{TW21b}. Precisely, we can have
\begin{align}
 \left| \mathcal{Q}_{4a}
 +\mathcal{Q}_{4b} \right|
 \lesssim  \|\mathrm{D}_*^{\alpha} \psi (t)\|_{L^2(\Sigma)}^2+   \mathcal{M}_3(t).
 \label{Q4a:es1}
\end{align}
If $\alpha_0< m$, then we infer 
\begin{align}
 \|\mathrm{D}_*^{\alpha}\psi(t)\|_{L^2(\Sigma)}^2
 \lesssim \int_{\Sigma_t} |\mathrm{D}_*^{\alpha}\psi||\p_t\mathrm{D}_*^{\alpha}\psi|
 \lesssim \|( \psi,\mathrm{D}_{x'}\psi)\|_{H^m(\Sigma_t)}^2.
 \label{Q4b:es1}
\end{align}
Term $\mathcal{Q}_{4c}$ can be estimated by use of the Moser-type calculus inequalities as
\begin{align}
\left| \mathcal{Q}_{4c}  \right|
 \lesssim_K
 \boldsymbol{\epsilon} \varepsilon \|\mathrm{D}_{x'}^3\psi\|_{H^m(\Sigma_t)}^2
 +C(\boldsymbol{\epsilon})\varepsilon \mathcal{M}_3(t).
 \label{Q4c:es1}
\end{align}
Plugging \eqref{Q4a:es1}--\eqref{Q4c:es1} into \eqref{Q4:es1} implies
\begin{align}
\left|  \int_{\Sigma_{t}} Q_4   \right|
 \lesssim_K
\boldsymbol{\epsilon} \varepsilon \|\mathrm{D}_{x'}^3\psi\|_{H^m(\Sigma_t)}^2
+C(\boldsymbol{\epsilon}) \mathcal{M}_3(t)
\quad \textrm{if } \alpha_0<m.
 \label{Q4:es1b}
\end{align}

If $\alpha_0=m$, then we get from \eqref{Reg.c} that
\begin{align*}
Q_4=
-2  \p_t^m W_2
\left\{ [\p_t^m,\mathring{b}_2]\psi
+\mathring{b}_2\p_t^{m-1} \left(W_2
-(\mathring{v}'\cdot\mathrm{D}_{x'}+\mathring{b}_1)\psi
-\varepsilon  \Delta_{x'}^2 \psi
\right)
\right\},
\end{align*}
which leads to
\begin{align}
 \int_{\Sigma_{t}} Q_4
=\widetilde{\mathcal{Q}}_{4a}
+\widetilde{\mathcal{Q}}_{4b}
+
 2\varepsilon \int_{\Sigma_{t}}  \mathring{b}_2 \p_t^m W_2 \p_t^{m-1}\Delta_{x'}^2 \psi 
, \label{Q4:es2}
\end{align}
where
\begin{align*}
&\widetilde{\mathcal{Q}}_{4a}:=
-\int_{\Sigma}  \p_t^{m-1} W_2
\left\{ \mathring{b}_2\p_t^{m-1} W_2-2\mathring{b}_2
\p_t^{m-1}(\mathring{v}'\cdot\mathrm{D}_{x'}+\mathring{b}_1)\psi
+2[\p_t^{m},\mathring{b}_2]\psi
\right\}
 ,\\
&\widetilde{\mathcal{Q}}_{4b}:=
\int_{\Sigma_{t}}   \p_t^{m-1} W_2
\left\{ \p_t\mathring{b}_2\p_t^{m-1} W_2-2\p_t\left(\mathring{b}_2
\p_t^{m-1}(\mathring{v}'\cdot\mathrm{D}_{x'}+\mathring{b}_1)\psi
-[\p_t^{m},\mathring{b}_2]\psi\right)
\right\}.
\end{align*}
Applying integration by parts and using Cauchy's inequality yield
\begin{align}
\left|  \widetilde{\mathcal{Q}}_{4a}+\widetilde{\mathcal{Q}}_{4b}   \right|
 \lesssim_K \, &
 \|\p_t^{m-1}W_2(t)\|_{L^2(\Sigma)}^2+\|\p_t^{m-1}W_2\|_{L^2(\Sigma_t)}^2
 \nonumber\\
 &+ \big\| \mathring{b}_2\p_t^{m-1} (\mathring{v}'\cdot\mathrm{D}_{x'}+\mathring{b}_1)
 \psi -[\p_t^{m},\mathring{b}_2]\psi \big\|_{H^1(\Sigma_t)}^2.
\nonumber 
\end{align}
Noting from \eqref{d1W:es1} that
\begin{align}
\|\p_t^{m-1}W_2(t)\|_{L^2(\Sigma)}^2
&\lesssim_K \boldsymbol{\epsilon} \| \p_t^{m-1}\p_1 W_2 (t)\|^2_{L^2(\Omega^+)} +C(\boldsymbol{\epsilon}) \|   W_2  \|^2_{H_*^m(\Omega_t^+)}\nonumber\\
&\lesssim_K
\boldsymbol{\epsilon}\sum_{\langle \beta\rangle\leq m} \| \mathrm{D}_*^{\beta} W(t)\|_{L^2(\Omega^+)}^2
+C(\boldsymbol{\epsilon}) \mathcal{M}_1 (t) ,
\label{dtW:es}
\end{align}
we employ the Moser-type calculus inequalities to get
\begin{align}
\left|  \widetilde{\mathcal{Q}}_{4a}+\widetilde{\mathcal{Q}}_{4b}  \right|
 \lesssim_K
\boldsymbol{\epsilon}\sum_{\langle \beta\rangle\leq m} \| \mathrm{D}_*^{\beta} W(t)\|_{L^2(\Omega^+)}^2
+C(\boldsymbol{\epsilon}) \mathcal{M} (t)
\label{Q4a:es2b}
\end{align}
for all $\boldsymbol{\epsilon}>0$, where
$\mathcal{M}(t)$ is defined by \eqref{M.cal}.
Plugging \eqref{Q4a:es2b} into \eqref{Q4:es2}, we find
\begin{align}
\left| \int_{\Sigma_{t}} Q_4  \right|
\lesssim_K
\,& \boldsymbol{\epsilon}\varepsilon \|\p_t^m W_2\|_{L^2(\Sigma_{t})}^2
+C(\boldsymbol{\epsilon}, K)  \varepsilon
\|\p_t^{m-1}\mathrm{D}_{x'}^4 \psi\|_{L^2(\Sigma_{t})}^2\nonumber\\
&+\boldsymbol{\epsilon}\sum_{\langle \beta\rangle\leq m} \| \mathrm{D}_*^{\beta} W(t)\|_{L^2(\Omega^+)}^2
+C(\boldsymbol{\epsilon}) \mathcal{M}(t)
\quad \textrm{if }\alpha_0=m.
\label{Q4:es2b}
\end{align}

\vspace*{1mm}
The first and second terms on the right-hand side of \eqref{Q4:es2b} can be absorbed by the left-hand side of \eqref{es:HO3a} with $\mathrm{D}_*^{\alpha}=\p_t^m$ 
and $\mathrm{D}_*^{\alpha}=\p_t^{m-1}\mathrm{D}_{x'}$, respectively.
Therefore, we plug \eqref{Q4:es1b} and \eqref{Q4:es2b} into \eqref{es:HO3a},
let $\boldsymbol{\epsilon}>0$ be sufficiently small,
and take a suitable combination of the resulting identities to discover
\begin{align}
 \nonumber
 &\|\mathrm{D}_*^{\alpha}W(t)\|_{L^2(\Omega^+)}^2
 +\|\mathrm{D}_*^{\alpha}\nabla\xi(t)\|^2_{L^2(\Omega^-)}
 +\|\mathrm{D}_*^{\alpha}\mathrm{D}_{x'}\psi(t)\|^2_{L^2(\Sigma)}
 \\[1mm]
 \nonumber
 &\qquad
 +\varepsilon\|(\mathrm{D}_*^{\alpha}W_2, \mathrm{D}_*^{\alpha} \mathrm{D}_{x'}^3\psi)\|_{L^2(\Sigma_{t})}^2
 +\varepsilon
 \|(\mathrm{D}_*^{\alpha}\mathrm{D}_{x'}\xi,\mathrm{D}_*^{\alpha}\mathrm{D}_{x'}^2\xi)(t)
 \|_{L^2(\Sigma)}^2
 \\[1mm]
 &  \lesssim_K
 C(\boldsymbol{\epsilon}) \mathcal{M}(t)
 +\varepsilon\mathring{\mathrm{C}}_{m+4} \|\mathrm{D}_{x'}^2\psi\|_{L^{\infty}(\Sigma_t)}^2
 +C(K, \boldsymbol{\epsilon} ) \varepsilon \|(\mathrm{D}_{x'}    \xi, \mathrm{D}_{x'}^2   \xi) \|_{H^m(\Sigma_{t})}^2
 \nonumber\\[1mm]
 &  \qquad
 +\boldsymbol{\epsilon}\varepsilon \| \mathrm{D}_{x'}^3 \psi\|_{H^m(\Sigma_{t})}^2
 +\boldsymbol{\epsilon}\sum_{\langle \beta\rangle\leq m} \| \mathrm{D}_*^{\beta} W(t)\|_{L^2(\Omega^+)}^2
\qquad \textrm{for }\alpha_1=\alpha_4=0.
\label{es1d}
\end{align}

\subsubsection{Conclusion}
Combining \eqref{es1d} with \eqref{es:HO1} and \eqref{es:HO2},
we take $ \boldsymbol{\epsilon}>0$ small enough and use Gr\"{o}nwall's inequality to derive
\begin{align}
 \nonumber
 &\sum_{\langle\alpha\rangle\leq m}  \|\mathrm{D}_*^{\alpha}W(t)\|_{L^2(\Omega^+)}^2
 +\sum_{\langle\alpha\rangle\leq m,\ \alpha_1=\alpha_4=0}
\left(
\|\mathrm{D}_*^{\alpha}\nabla\xi(t)\|^2_{L^2(\Omega^-)}
+\|\mathrm{D}_*^{\alpha}\mathrm{D}_{x'}\psi(t)\|^2_{L^2(\Sigma)}
 \right)
 \nonumber
  \\[2mm]
 &
\qquad+\varepsilon\| (\mathrm{D}_{x'}^3\psi, \mathrm{D}_{x'}\xi, \mathrm{D}_{x'}^2\xi) \|_{H^m(\Sigma_{t})}^2
  \lesssim_K  \mathcal{M}(t)
 +\varepsilon\mathring{\mathrm{C}}_{m+4} \|\mathrm{D}_{x'}^2\psi\|_{L^{\infty}(\Sigma_t)}^2
 \label{es1}
\end{align}
where $\mathcal{M}(t)$ is defined by \eqref{M.cal} ({\it cf.}~\eqref{est:R}, \eqref{M2.cal}, and \eqref{M3.cal}).
To close the above estimate \eqref{es1}, we first obtain from 
\eqref{Reg.c} that
\begin{align}
\|\p_t^m \psi \|_{L^2(\Sigma_{t})}^2  &\lesssim
\|\p_t^{m-1}W_2-\varepsilon \p_t^{m-1}\Delta_{x'}^2\psi\|_{L^2(\Sigma_{t})}^2
+\|( \mathring{v}'\cdot\mathrm{D}_{x'}+\mathring{b}_1)\psi\|_{H^{m-1}(\Sigma_{t})}^2,
\nonumber
\end{align}
and hence
\begin{align}
 \mathcal{M}_3(t)\lesssim_K
\,&\mathcal{M}_1(t) + \varepsilon^2 \| \mathrm{D}_{x'}^3\psi\|_{H^m(\Sigma_{t})}^2
+\sum_{ \alpha_0 <m,\ \alpha_1=\alpha_4=0} \| \mathrm{D}_*^{\alpha}\psi\|_{L^2(\Sigma_t)}^2
\nonumber
\\
&+\|\mathrm{D}_{x'}\psi\|_{H^m(\Sigma_t)}^2
+\mathring{\mathrm{C}}_{m+4} \|( \psi,\mathrm{D}_{x'}\psi)\|_{L^{\infty}(\Sigma_t)}^2.
\label{es1b}
\end{align}
It follows from \eqref{A.ring:def} and \eqref{Reg.b} that
\begin{align*}
\p_1\p_1 \xi =\mathring{\rm c}_2\nabla \xi +\mathring{\rm c}_1  \mathrm{D}_{x'}\nabla\xi
\qquad\textrm{in }\Omega_{T}^-.
\end{align*}
Using the last identity and the Moser-type calculus inequalities, by induction in
$k=0,1,\ldots,m-1$, we can deduce that
\begin{align}
\|\mathrm{D}_*^{\gamma}\p_1^{k+1}\p_1\xi(t)\|^2_{L^2(\Omega^-)}\lesssim_{K} \mathcal{M}_2(t)+\sum_{\langle \alpha\rangle\leq m,\ \alpha_1=\alpha_{4}=0}
\|\mathrm{D}_*^{\alpha} \nabla\xi(t) \|^2_{L^2(\Omega^-)}
\label{es2a}
\end{align}
for all $\langle\gamma\rangle\leq m-k-1$.
In view of \eqref{Q4b:es1} and \eqref{es1}--\eqref{es2a},
we define the energy functional
\begin{align*}
\mathcal{I}(t):=
\,& \sum_{\langle\alpha\rangle\leq m}  \|\mathrm{D}_*^{\alpha}W(t)\|_{L^2(\Omega^+)}^2
+\sum_{|\beta|\leq m}\|\mathrm{D}^{\beta}\nabla\xi(t)\|^2_{L^2(\Omega^-)}\\
&+\sum_{\langle\alpha\rangle\leq m,\ \alpha_1=\alpha_4=0}
\|\mathrm{D}_*^{\alpha}\mathrm{D}_{x'}\psi(t)\|^2_{L^2(\Sigma)}
+\sum_{\alpha_0< m,\ \alpha_1=\alpha_4=0}
\|\mathrm{D}_*^{\alpha} \psi(t)\|^2_{L^2(\Sigma)}
\end{align*}
and find
\begin{align}
 &\mathcal{I}(t)
+\varepsilon\| (\mathrm{D}_{x'}^3\psi, \mathrm{D}_{x'}\xi, \mathrm{D}_{x'}^2\xi)\|_{H^m(\Sigma_{t})}^2
\lesssim_K   \int_0^t \mathcal{I}(\tau)\mathrm{d}\tau
+\mathcal{N}(t)
\label{es2}
\end{align}
for $\varepsilon>0$ small enough,
where
\begin{align*}
\mathcal{N}(t):=\|\bm{f}\|^2_{H^m_*(\Omega_{t}^+)}
+\mathring{\mathrm{C}}_{m+4}
\left(\|(\bm{f},W)\|_{W_*^{2,\infty}(\Omega_t^+)}^2+\|\nabla \xi\|_{L^{\infty}(\Omega_t^-)}^2+ \| \psi\|_{W^{2,\infty}(\Sigma_t)}^2\right).
\end{align*}
Apply Gr\"{o}nwall's inequality to \eqref{es2} and use the embedding inequalities  to infer
\begin{align}
&\mathcal{I}(t)
+\varepsilon\| (\mathrm{D}_{x'}^3\psi, \mathrm{D}_{x'}\xi, \mathrm{D}_{x'}^2\xi)\|_{H^m(\Sigma_{t})}^2
\lesssim_K \mathcal{N}(t)
\nonumber\\
&\  \  \lesssim_K  \|\bm{f}\|^2_{H^m_*(\Omega_{t}^+)}
+\mathring{\mathrm{C}}_{m+4}
 \left(\|(\bm{f},W)\|^2_{H^6_*(\Omega_{t}^+)}
+\|\nabla \xi\|^2_{H^6(\Omega_{t}^-)}
+\|\psi\|^2_{H^5(\Sigma_{t})}
\right)
\label{es3a}
\end{align}
for all $0\leq t\leq T$ provided $T,\varepsilon> 0$ are suitably small.
Integrating \eqref{es3a} over $[0,T]$,
we can find $T_0>0$ depending on $K_0$ ({\it cf.}~\eqref{H:thm.linear}),
such that
\begin{align}
&\|W\|^2_{H^m_*(\Omega_{T}^+)} + \|\nabla \xi\|^2_{H^m(\Omega_{T}^-)}
+\|\mathrm{D}_{x'}\psi\|^2_{H^m(\Sigma_{T})}+\|\psi\|^2_{H^{m-1}(\Sigma_{T})}
\nonumber\\
&\quad   \lesssim_{K_0}
\|\bm{f}\|^2_{H^m_*(\Omega_{T}^+)}+ \|\bm{f}\|^2_{H^6_*(\Omega_{T}^+)}
 \|(\mathring{U},\mathring{h},\mathring{\varphi}) \|_{m+4}^2
\quad \textrm{for  } 0\leq T\leq T_0,\  m\geq 6.
\label{es3b}
\end{align}
Combine \eqref{es1b}, \eqref{es2}, and \eqref{es3b} to get
\begin{align}
&\|W\|^2_{H^m_*(\Omega_{T}^+)} + \|(\xi,\nabla \xi)\|^2_{H^m(\Omega_{T}^-)}
+\|(\psi, \mathrm{D}_{x'}\psi)\|^2_{H^m(\Sigma_{T})}
+\varepsilon\| (\mathrm{D}_{x'}^3\psi, \mathrm{D}_{x'}\xi, \mathrm{D}_{x'}^2\xi)\|_{H^m(\Sigma_{T})}^2
\nonumber\\
&\quad \lesssim_{K_0}
\|\bm{f}\|^2_{H^m_*(\Omega_{T}^+)}+ \|\bm{f}\|^2_{H^6_*(\Omega_{T}^+)}
\|(\mathring{U},\mathring{h},\mathring{\varphi}) \|_{m+4}^2
\quad \textrm{for  } 0\leq T\leq T_0,\  m\geq 6,
\label{es3c}
\end{align}
provided $\varepsilon>0$ is sufficiently small.
Estimate \eqref{es3c} provides the desired uniform-in-$\varepsilon$ estimate for solutions to regularization \eqref{Reg}.

\subsection{Proof of Theorem \ref{thm:linear}}
The uniform-in-$\varepsilon$ high-order estimate \eqref{es3c} enables us to establish the solvability of problem \eqref{ELP3} by passing to the limit $\varepsilon\to 0$.
Indeed, according to \eqref{es3c}, we can extract a subsequence weakly convergent to $(W,\xi, \psi)\in H^m_*(\Omega_{T}^+)\times H^m(\Omega_{T}^-)\times H^m(\Sigma_{T})$ satisfying estimate \eqref{es3c} with $\varepsilon=0$.
Noting that $\p_1 W_2 $ and $\sqrt{\varepsilon}( \Delta_{x'}^2 \psi,  \Delta_{x'}\xi,  \Delta_{x'}^2 \xi)$ are uniformly bounded in $H^{m-2}_*(\Omega_{T}^+)$ and $H^{m-2}(\Sigma_{T})$, respectively,
the passage to the limit $\varepsilon\to 0$ in \eqref{Reg} verifies that $(W,\xi, \psi)$ solves  the reduced problem \eqref{ELP3}. Furthermore, the uniqueness of solutions follows from estimate \eqref{es3c} with $\varepsilon=0$.

Recall from \eqref{W:def}, \eqref{xi:def}, \eqref{ELP3a}, and \eqref{ELP2a} that
$$\dot{V}=V_{\natural}+J(\mathring{\Phi})W, \ \
\dot{h}=h_{\natural}+\p_1\mathring{\Phi}\mathring{\eta}^{\mathsf{T}}\nabla\xi,\ \
\bm{f}=J(\mathring{\Phi})^{\mathsf{T}}
(f^+-\mathbb{L}_{e+}'(\mathring{U},\mathring{\Phi})V_{\natural}).$$
Then using \eqref{h.nat:est}--\eqref{V.natural} and \eqref{es3c} with $\varepsilon=0$, we can apply the embedding and Moser-type calculus inequalities to obtain the tame energy estimate \eqref{tame} for the effective linear problem \eqref{ELP1}.
This completes the proof of Theorem \ref{thm:linear}.

\section{Nonlinear Analysis}\label{sec:Nash-Moser}
In this section, we employ a suitable Nash--Moser iteration scheme to prove Theorem \ref{thm:main}, that is, the solvability of the nonlinear problem \eqref{NP1}. See, for instance, \cite{AG07MR2304160} or \cite{S16MR3524197} for a more general presentation of this method.

\subsection{Approximate solutions}
To apply Theorem \ref{thm:linear}, which is valid for functions vanishing in the past,
we reduce the nonlinear problem \eqref{NP1} to that with zero initial data via the approximate solution.
The compatibility conditions on the initial data introduced below are necessary for constructing the approximate solution.

Let $m\geq 3$ be an integer. Suppose that the initial data $ U_0  \in H^{m+3/2}(\Omega^+)$ and $\varphi_0\in H^{m+2}(\mathbb{T}^{2})$ satisfy
$\|\varphi_0\|_{L^{\infty}(\mathbb{T}^{2})}<1$ and the hyperbolicity condition \eqref{hyperbolicity}.
It follows from \eqref{chi:def} that
\begin{align}
\nonumber
 \p_1\Phi_0\geq \frac{3-3\|\varphi_0\|_{L^{\infty}(\mathbb{T}^{2})} }{3+\|\varphi_0\|_{L^{\infty}(\mathbb{T}^{2})}}>0
\quad \textrm{for } \Phi_0(x):=x_1+\chi(x_1)\varphi_0(x').
\end{align}
Then the initial vacuum magnetic field $h_0\in \mathbb{R}^3$ is uniquely determined by the following div-curl system ({\it cf.}~\eqref{NP1b}--\eqref{NP1d} and Lemma \ref{h.nat:lem}):
\begin{align}
\nonumber
 L_-(\Phi_0)h_0 =0\ \ \textrm{in  }  \Omega^-,\quad  \ \
 h_0\cdot N_0=0 \ \ \textrm{on  }    \Sigma,\quad  \ \
h_0\times \mathbf{e}_1=\bm{j}_{\rm c}(0)\ \ \textrm{on  }    \Sigma^-,
\end{align}
where the operator $L_-$ is defined by \eqref{L-:def} and $N_0:=(1,-\p_2\varphi_0,-\p_3\varphi_0)^{\mathsf{T}}$.
Let us denote
$U_{(\ell)}  := \p_t^{\ell }U|_{t=0}$ and $\varphi_{(\ell)}:= \p_t^{\ell}\varphi|_{t=0}$
for any $\ell\in\mathbb{N}$.
Taking $\ell$ time derivatives of the interior equations \eqref{NP1a} and the first condition in \eqref{NP1c}, we evaluate the resulting identities at the initial time to determine $U_{(\ell)}  $ and $\varphi_{(\ell)}$ inductively.
Then we set $h_{(\ell)}:=\p_t^{\ell} h|_{t=0}$ as the unique solution of the elliptic problem that results from taking $\ell$ time derivatives of the equations \eqref{NP1b}, the third condition in \eqref{NP1c}, and the second condition in \eqref{NP1d}. More precisely, we have the following result (see \cite[Lemma 19]{ST14MR3151094} for the detailed proof).

\begin{lemma}\label{lem:CA1}
Suppose that $m\geq 3$ is an integer,
the surface current $\bm{j}_{\rm c}$ belongs to $H^{m+3/2}([0,T_0]\times \Sigma^-)$ for some $T_0>0$, and
the initial data $(U_0,\varphi_0)\in H^{m+3/2}(\Omega^+)\times H^{m+2}(\mathbb{T}^2)$ satisfy
$\|\varphi_0\|_{L^{\infty}(\mathbb{T}^{2})}<1$ and \eqref{hyperbolicity}.
Then the procedure described above determines
$U_{(\ell)}\in H^{m+3/2-\ell}(\Omega^+)$,
$\varphi_{(\ell)}\in H^{m+2-\ell}(\mathbb{T}^2)$,
and $h_{(\ell)}\in H^{m+3/2-\ell}(\Omega^-)$, for $\ell=0,1,\ldots,m$,
which satisfy
\begin{align}
\sum_{\ell=0}^{m}
\left( \|U_{(\ell)}\|_{H^{m+3/2-\ell}(\Omega^+) }
+\|\varphi_{(\ell)}\|_{H^{m+2-\ell}(\mathbb{T}^2) }+\|h_{(\ell)}\|_{H^{m+3/2-\ell}(\Omega^-) }
\right)\leq C(M_0),
\nonumber
\end{align}
for some positive constant $C(M_0)$ depending on
\begin{align}
 M_0:=
\|U_{0}\|_{H^{m+3/2}(\Omega^+) } +\|\varphi_{0}\|_{H^{m+2}(\mathbb{T}^2) } + \| \bm{j}_{\rm c}\|_{H^{m+3/2}([0,T_0]\times \Sigma^-)}.
 \label{M0:def}
\end{align} 
\end{lemma}

The compatibility conditions on the initial data are defined as follows.
\begin{definition}
 \label{def:1}
Suppose that all the conditions of Lemma \ref{lem:CA1} are satisfied.
The initial data $(U_0,\varphi_0)$ are said to be compatible up to order $m$,
if  $U_{(\ell)}$, $\varphi_{(\ell)}$, and $h_{(\ell)}$ satisfy the boundary conditions
$v_{1(\ell)}|_{\Sigma^+}=0$
and \begin{align}
 q_{(\ell)} =\, &
 \sum_{i=0}^{\ell -1}  \begin{pmatrix}
  \ell-1 \\ i
 \end{pmatrix} h_{(i)}\cdot h_{(\ell-i)}
 \nonumber\\
 &+ \mathfrak{s}\sum_{\substack{\alpha_{i}\in\mathbb{N}^{2}\\ |\alpha_1|+\cdots+\ell |\alpha_{\ell}|=\ell}}   \,\mathrm{D}_{x'}\cdot
 \bigg(\mathrm{D}_{\zeta}^{\alpha_1+\cdots+\alpha_\ell}\mathfrak{f}\big(\zeta_{(0)}\big)\ell! \prod_{i=1}^\ell\frac{1}{\alpha_{i}!}
 \bigg(\frac{\zeta_{(i)}}{i!}\bigg)^{\alpha_{i}}\bigg)\quad  \textrm{on }\Sigma,
 \label{compat1}
\end{align}
for $\ell=0,\ldots,m$,
where  $\zeta_{(i)}:=\mathrm{D}_{x'}\varphi_{(i)}\in\mathbb{R}^2$ and  $\mathfrak{f}(\zeta):={\zeta}/{\sqrt{1+|\zeta|^2}}.$
\end{definition}
The compatibility conditions \eqref{compat1} results from taking $\ell$ time derivatives of the second condition in \eqref{NP1c} ({\it cf.}~\cite[(3.4)]{TW21b}).
We can construct the approximate solution as in \cite[Lemma 21]{ST14MR3151094} and \cite[Lemma 5.2]{TW22MR4367917}.

\begin{lemma}\label{lem:app}
Suppose that all the conditions of Lemma \ref{lem:CA1} are satisfied.
Suppose further that the initial data $(U_0,\varphi_0)$ are compatible up to order $m$ and satisfy the constraints \eqref{inv1b}--\eqref{inv2b}. Then there exist positive constants $C(M_0)$ and $T_1(M_0)$ depending on $M_0$ \textnormal{(}{\it cf.}~\eqref{M0:def}\textnormal{)}, such that if $0<T\leq T_1(M_0)$, then we can find $(U^a,h^a,\varphi^a)$ that belongs to $H^{m+1}(\Omega_{T}^+)\times H^{m+1}(\Omega_{T}^-)\times H^{m+5/2}(\Sigma_{T})$ and satisfies
 \begin{alignat}{3}
 &\p_t^{\ell}\mathbb{L}_+(U^{a},\Phi^{a})\big|_{t=0}=0
 \quad&&\textnormal{in } \Omega^+ \quad \textnormal{for } \ell=0,\ldots,m-1,
  \label{app1a}\\
 & \mathbb{L}_-(h^{a},\Phi^{a})=0 \quad&&\textnormal{in } \Omega^-_T,
\label{app1b}\\
& \mathbb{B}_+(U^a,h^a,\varphi^a)=0\quad &&\textnormal{on }\Sigma_{T}^2\times \Sigma_{T}^+,
\label{app1c}
\\
&  \mathbb{B}_-(h^a,\varphi^a)=0\quad &&\textnormal{on }\Sigma_{T}\times \Sigma_{T}^-,
 \label{app1d} \\
& (U^{a},h^{a}, \varphi^{a})=(U_0, h_0, \varphi_0)\qquad &&\textnormal{if } t<0,
 \label{app1e}
\end{alignat}
where operators $\mathbb{L}_{\pm}$ and $\mathbb{B}_{\pm}$ are defined in \eqref{NP1a}--\eqref{NP1b} and \eqref{B.bb:def}, respectively,  and
$\Phi^{a}(t,x):=x_1 +\chi(x_1)\varphi^{a}(t,x')$.
 Moreover,
\begin{alignat}{3}
&\|U^{a}\|_{H^{m+1}(\Omega_T^+)}+\|h^{a}\|_{H^{m+1}(\Omega_T^-)} +\|\varphi^a\|_{H^{m+5/2}(\Sigma_T)}  \leq C (M_0),
 \label{app2a}  \\
 &  \rho_*<\inf_{\Omega_T^+} \rho(U^a)\leq \sup_{\Omega_T^+} \rho(U^a)<\rho^*,\quad
\|\varphi^a\|_{L^{\infty}(\Sigma_{T})}\leq \frac{3\|\varphi_0\|_{L^{\infty}(\mathbb{T}^{2})}+1}{4},
 \label{app2b} \\
 & H_1^{a}-H_2^{a}\p_{2} \varphi^{a}-H_3^{a}\p_{3} \varphi^{a}=0\quad \textnormal{on }  \Sigma_T,\quad
 H_1^a=0\quad   \textnormal{on }  \Sigma_T^+.
 \nonumber
\end{alignat}
\end{lemma}

The vector function $(U^a,h^a,\varphi^a)$ constructed above is called the {\it approximate solution} to the nonlinear problem \eqref{NP1}.
For
\begin{align}\nonumber 
 f^{a}:=\bigg\{\begin{aligned}
  & -\mathbb{L}_+(U^{a},\Phi^{a}) \quad &\textrm{if }t>0,\\
  & \; 0 \quad &\textrm{if }t<0,
 \end{aligned}
\end{align}
we have from \eqref{app1a} and \eqref{app2a} that
\begin{align}\label{f^a:est}
f^{a}\in H^{m}(\Omega_T^+),\quad
\|f^{a}\|_{ H^{m}(\Omega_T^+)}\leq\delta_0\left(T\right),
\end{align}
where $\delta_0(T)\to 0$ as $T\to 0$.
Using \eqref{app1a}--\eqref{app1e} implies that
$(U,\hat{h},\varphi)$ solves the nonlinear problem \eqref{NP1} on the time interval $[0,T]$,
if $(V,h,\psi):=(U,\hat{h},\varphi)-(U^{a},h^{a},\varphi^a)$ satisfies
\begin{align} \label{P.new}
\left\{
 \begin{aligned}
&\mathcal{L}_+(V,\Psi):=\mathbb{L}_+(U^a+V,\Phi^a+\Psi)-\mathbb{L}_+(U^a,\Phi^a)=f^a &\quad &\textnormal{in }\Omega_T^+,\\
&\mathcal{L}_-(h,\Psi):=\mathbb{L}_-(h^a+h,\Phi^a+\Psi)=0 &&\textnormal{in }\Omega_T^-,\\
&\mathcal{B}_+(V,h,\psi):=\mathbb{B}_+(U^a+V,h^a+h,\varphi^a+\psi)=0&&\textrm{on }\Sigma_T^2\times \Sigma_T^+,\\
&\mathcal{B}_-(h,\psi):=\mathbb{B}_-(h^a+h,\varphi^a+\psi)=0&&\textrm{on }\Sigma_T\times \Sigma_T^-,\\
&(V,h,\psi)=0\ &&\textrm{if }t< 0,
 \end{aligned}\right.
\end{align}
for $\Psi(t,x):=\chi(x_1)\psi(t,x')$.
Note here that as in \eqref{ELP1c} the notation $\Sigma_{T}^2\times \Sigma_{T}^+$ means that
the first two components of the corresponding vector equation are taken on $\Sigma_{T}$ and the third one on $\Sigma_{T}^+$.

\subsection{Nash--Moser iteration}

We first quote the properties on the smoothing operators from
\cite{A89MR976971,CS08MR2423311,T09ARMAMR2481071}.
Denote by $\mathscr{F}_*^s(\Omega_T^+)$ (resp.~$\mathscr{F}^s(\Omega_T^-)$) the class of $H_*^{s}(\Omega_T^+)$ (resp.~$H^{s}(\Omega_T^-)$) vanishing in the past.
\begin{proposition} \label{pro:smooth}
Let $T>0$ and $m\in \mathbb{N}$ with $m\geq 3$.
Then there is a family of smoothing operators $\{\mathcal{S}_{\theta}\}_{\theta\geq 1}:  \,\mathscr{F}_*^3(\Omega_T^+) \to \bigcap_{s\geq 3}\mathscr{F}_*^s(\Omega_T^+)$, such that
 \begin{subequations}\label{smooth.p1}
  \begin{alignat}{2}
   &  \|\mathcal{S}_{\theta} u\|_{H_*^{k}(\Omega_T^+)}\lesssim_m \theta^{(k-j)_+}\|u\|_{H_*^{j}(\Omega_T^+)}
   && \textrm{for  \;}k,j=1,\ldots,m,    \label{smooth.p1a}  \\[1.5mm]
   &  \|\mathcal{S}_{\theta} u-u\|_{H_*^{k}(\Omega_T^+)}\lesssim_m \theta^{k-j}\|u\|_{H_*^{j}(\Omega_T^+)}
   && \textrm{for  \;}1\leq k\leq j \leq m,   \label{smooth.p1b} \\
   &  \left\|\frac{\d}{\d \theta}\mathcal{S}_{\theta} u\right\|_{H_*^{k}(\Omega_T^+)}
   \lesssim_m \theta^{k-j-1}\|u\|_{H_*^{j}(\Omega_T^+)}
   &\quad&\textrm{for \;}k,j=1,\ldots,m,    \label{smooth.p1c}
  \end{alignat}
 \end{subequations}
 where $k,j\in\mathbb{N}$ and $(k-j)_+:=\max\{0,\, k-j \}$.
 Moreover, there exist two families of smoothing operators (still denoted by $\mathcal{S}_{\theta}$) acting respectively on $\mathscr{F}^3(\Omega_T^-)$ and functions defined on $\Sigma_T$, and satisfying the properties in \eqref{smooth.p1} with norms $\|\cdot\|_{H^{j}(\Omega_T^-)}$ and   $\|\cdot\|_{H^{j}(\Sigma_T)}$, respectively.
\end{proposition}

We follow \cite{CS08MR2423311,T09ARMAMR2481071,TW21b,TW21MR4201624,ST14MR3151094} to describe the iteration scheme for problem \eqref{P.new}.

\vspace{2mm}
\noindent{\bf Assumption\;(A-1)}: {\it  Set $ (V_0,h_0, \psi_0)=0$. Let $(V_k,h_k,\psi_k)$ be already given, vanish in the past, and satisfy
$V_{k,2}|_{\Sigma_{T}^+}=0$ and $h_k\times \mathbf{e}_1|_{\Sigma_{T}^-}=0$
for $k=0,1,\ldots,{n}$. Define $\Psi_k:=\chi( x_1)\psi_k$.}

\vspace{2mm}
The differences $\delta V_{{n}}:=V_{n+1}-V_n$, $\delta h_{{n}}:=h_{n+1}-h_n$, and $\delta \psi_{{n}}:=\psi_{n+1}-\psi_n$ will be determined through the effective linear problem
\begin{align}
\left\{\begin{aligned}
&\mathbb{L}'_{e+}(U^a+V_{{n}+1/2},\Phi^a+\Psi_{{n}+1/2})\delta \dot{V}_{{n}}=f_{{n}}^+
&\ &\textrm{in } \Omega^+_T, \\
&L_-( \Phi^a+\Psi_{{n}+1/2})\delta \dot{h}_{{n}}=f_{{n}}^-
&&\textrm{in } \Omega^-_T, \\
&\mathbb{B}'_{n+1/2}(\delta \dot{V}_{{n}},\delta \dot{h}_{{n}},\delta\psi_n)=g_n^+
&&\textrm{on } \Sigma_T^2\times \Sigma_{T}^+,\\
& \mathbb{B}'_{e-}(h^a+h_{{n}+1/2}, \varphi^a+\psi_{{n}+1/2}) (\delta\dot{h}_{{n}},\delta\psi_n)=g_n^-
&&\textrm{on } \Sigma_T\times \Sigma_{T}^-,\\
&(\delta \dot{V}_{{n}},\delta \dot{h}_{{n}},\delta\psi_{{n}}) =0 &&
\textrm{if } t<0,
\end{aligned} \right.
 \label{effective.NM}
\end{align}
where $(V_{{n}+1/2},h_{n+1/2}, \psi_{{n}+1/2})$ is a suitable modified state to be specified in Proposition \ref{pro:modified} so that $(U^a+V_{{n}+1/2},h^a+h_{n+1/2},\varphi^a+\psi_{{n}+1/2})$ satisfies \eqref{bas1a}--\eqref{bas1f}, $\Psi_{{n}+1/2}:=\chi(x_1)\psi_{{n}+1/2}$,
and
\begin{align}
\label{B.bb':2}
&\mathbb{B}'_{n+1/2}:=\mathbb{B}'_{e+}(U^a+V_{{n}+1/2}, h^a+h_{{n}+1/2}, \varphi^a+\psi_{{n}+1/2}), \\
 \label{good.NM}
&\delta \dot{V}_{{n}}:=\delta V_{{n}}-\frac{\p_1 (U^a+V_{{n}+1/2})}{\p_1 (\Phi^a+\Psi_{{n}+1/2})}\delta\Psi_{{n}},
 \quad
\delta \dot{h}_{{n}}:=\delta h_{{n}}-\frac{\p_1 (h^a+h_{{n}+1/2})}{\p_1 (\Phi^a+\Psi_{{n}+1/2})}\delta\Psi_{{n}}.
\end{align}
Source terms $f_n^{\pm}, g_n^{\pm}$ will be chosen via the accumulated error terms at Step ${n}$.

\vspace{2mm}
\noindent{\bf Assumption (A-2)}: {\it Set $f_0^+:=\mathcal{S}_{\theta_0}f^a$ and $(e_0^{\pm},\tilde{e}_0^{\pm},f_0^-,g_0^{\pm}):=0$ for $\theta_0\geq 1$ sufficiently large. Let $(e_k^{\pm},\tilde{e}_k^{\pm},f_k^{\pm},g_k^{\pm})$ be given and vanish in the past for $k=1,\ldots,{n}-1$.}

\vspace{1mm}
Under {\bf Assumptions (A-1)}--\textbf{(A-2)}, we set the accumulated error terms 
by
\begin{align}  \label{E.E.tilde}
 E_{{n}}^{\pm}:=\sum_{k=0}^{{n}-1}e_{k}^{\pm},\quad \widetilde{E}_{{n}}^{\pm}:=\sum_{k=0}^{{n}-1}\tilde{e}_{k}^{\pm},
\end{align}
and compute the source terms $f_n^{\pm}, g_n^{\pm}$ from
\begin{align} \label{source}
\sum_{k=0}^{{n}} f_k^+ +\mathcal{S}_{\theta_{{n}}}E_{{n}}^+=\mathcal{S}_{\theta_{{n}}}f^a,\quad
\sum_{k=0}^{{n}}f_k^{-}+\mathcal{S}_{\theta_{{n}}}{E}_{{n}}^{-}=0,\quad
\sum_{k=0}^{{n}}g_k^{\pm}+\mathcal{S}_{\theta_{{n}}}\widetilde{E}_{{n}}^{\pm}=0,
\end{align}
where  $\mathcal{S}_{\theta_{{n}}}$
are the smoothing operators given in Proposition \ref{pro:smooth} with $\theta_{{n}}:=(\theta^2_0+{n})^{1/2}$.
Once $f_n^{\pm}$ and $g_n^{\pm}$ are specified, applying Theorem \ref{thm:linear} to problem \eqref{effective.NM} can determine $(\delta \dot{V}_{{n}},\delta\dot{h}_n, \delta \psi_{{n}})$.
Then we obtain $\delta V_{{n}}$ and $\delta h_{{n}}$ from \eqref{good.NM}.

To define the error terms, we decompose
\begin{align}
\nonumber&\mathcal{L}_+(V_{{n}+1}, \Psi_{{n}+1})-\mathcal{L}_+(V_{{n}}, \Psi_{{n}})
 = \mathbb{L}_+'(U^a+V_{{n}}, \Phi^a+ \Psi_{{n}})(\delta V_{{n}},\delta \Psi_{{n}})+e_{{n}+}'\\
 \nonumber&\quad = \mathbb{L}_+'(U^a+\mathcal{S}_{\theta_{{n}}}V_{{n}}, \Phi^a+\mathcal{S}_{\theta_{{n}}} \Psi_{{n}})(\delta V_{{n}},\delta \Psi_{{n}})+e_{{n}+}'+e_{{n}+}''\\
 \nonumber&\quad = \mathbb{L}_+'(U^a+V_{{n}+1/2}, \Phi^a+ \Psi_{{n}+1/2})(\delta V_{{n}},\delta \Psi_{{n}})+e_{{n}+}'+e_{{n}+}''+e_{{n}+}'''\\
 \label{decom1}&\quad = \mathbb{L}_{e+}'(U^a+V_{{n}+1/2}, \Phi^a+ \Psi_{{n}+1/2})\delta \dot{V}_{{n}}+e_{{n}+}'+e_{{n}+}''+e_{{n}+}'''+e_{{n}+}^{*},
 \\[1mm]
\nonumber&\mathcal{L}_-(h_{{n}+1}, \Psi_{{n}+1})-\mathcal{L}_-(h_{{n}}, \Psi_{{n}})
= \mathbb{L}_-'(h^a+h_{{n}}, \Phi^a+ \Psi_{{n}})(\delta h_{{n}},\delta \Psi_{{n}})+e_{{n}-}'\\
\nonumber&\quad = \mathbb{L}_-'(h^a+\mathcal{S}_{\theta_{{n}}}h_{{n}}, \Phi^a+\mathcal{S}_{\theta_{{n}}} \Psi_{{n}})(\delta h_{{n}},\delta \Psi_{{n}})+e_{{n}-}'+e_{{n}-}''\\
\nonumber&\quad = \mathbb{L}_-'(h^a+h_{{n}+1/2}, \Phi^a+ \Psi_{{n}+1/2})(\delta h_{{n}},\delta \Psi_{{n}})+e_{{n}-}'+e_{{n}-}''+e_{{n}-}'''\\
\label{decom2}&\quad = {L}_{-}(\Phi^a+ \Psi_{{n}+1/2})\delta \dot{h}_{{n}}+e_{{n}-}'+e_{{n}-}''+e_{{n}-}'''+e_{{n}-}^{*},
\end{align}
and
\begin{align}
\nonumber&\mathcal{B}_+(V_{{n}+1},h_{n+1},\psi_{{n}+1})-\mathcal{B}_+(V_{{n}},h_n, \psi_{{n}})  \\
\nonumber&\quad
= \mathbb{B}_+'(U^a+V_{{n}},h^a+h_n,\varphi^a+\psi_{{n}})(\delta V_{{n}},\delta h_n, \delta\psi_{{n}})+\tilde{e}_{{n}+}'\\
 \nonumber&\quad = \mathbb{B}_+'(U^a+\mathcal{S}_{\theta_{{n}}}V_{{n}},h^a+\mathcal{S}_{\theta_{{n}}}h_{{n}},\varphi^a+\mathcal{S}_{\theta_{{n}}}\psi_{{n}})(\delta V_{{n}},\delta h_n, \delta\psi_{{n}})+\tilde{e}_{{n}+}'+\tilde{e}_{{n}+}''\\
 \label{decom3}&\quad =\mathbb{B}'_{n+1/2}(\delta \dot{V}_{{n}} ,\delta\dot{h}_n,\delta\psi_{{n}})+\tilde{e}_{{n}+}'+\tilde{e}_{{n}+}''+\tilde{e}_{{n}+}''',\\
\nonumber&\mathcal{B}_-(h_{n+1},\psi_{{n}+1})-\mathcal{B}_-(h_n, \psi_{{n}})
= \mathbb{B}_-'(h^a+h_n,\varphi^a+\psi_{{n}})(\delta h_n, \delta\psi_{{n}})+\tilde{e}_{{n}-}'\\
\nonumber&\quad = \mathbb{B}_-'(h^a+\mathcal{S}_{\theta_{{n}}}h_{{n}},\varphi^a+\mathcal{S}_{\theta_{{n}}}\psi_{{n}})(\delta h_n, \delta\psi_{{n}})+\tilde{e}_{{n}-}'+\tilde{e}_{{n}-}''\\
\label{decom4}&\quad =\mathbb{B}'_{e-}(h^a+h_{{n}+1/2}, \varphi^a+\psi_{{n}+1/2}) (\delta\dot{h}_n,\delta\psi_{{n}})+\tilde{e}_{{n}-}'+\tilde{e}_{{n}-}''+\tilde{e}_{{n}-}''',
\end{align}
where $\mathbb{L}_{\pm}'$, $\mathbb{L}_{e+}'$, $\mathbb{B}_{\pm}'$, and $\mathbb{B}'_{n+1/2}$ is given in \eqref{Alinhac1}--\eqref{Alinhac2}, \eqref{ELP1a}, \eqref{B'+.bb:def}--\eqref{B'-.bb:def}, and \eqref{B.bb':2}, respectively.
The description of the iteration scheme is completed by setting
\begin{align} \label{e.e.tilde}
 e_{{n}}^{\pm}:=e_{{n}\pm}'+e_{{n}\pm}''+e_{{n}\pm}'''+e_{n\pm}^*,\quad
 \tilde{e}_{{n}}^{\pm}:=\tilde{e}_{{n}\pm}'+\tilde{e}_{{n}\pm}''+\tilde{e}_{{n}\pm}'''.
\end{align}

Let us formulate the inductive hypothesis.
Set $m\in \mathbb{N}$ with $m\geq {13}$ and $\widetilde{\alpha}:=m-{5}$.
The initial data $(U_0,\varphi_0)$ are supposed to satisfy all the conditions of Lemma \ref{lem:app}, which implies estimates \eqref{app2a}--\eqref{f^a:est}.
Moreover, {\bf Assumptions (A-1)}--\textbf{(A-2)} are supposed to hold.
For some integer $ {\alpha }\in (6,\widetilde{\alpha})$ and constant $\boldsymbol{\epsilon}>0$ to be chosen later on, our inductive hypothesis reads
\begin{align*}
 (\mathbf{H}_{{n}-1}) \left\{\begin{aligned}
 \textrm{(a)}\,\,  &
 \|(\delta V_k, \delta h_k, \delta\psi_k)\|_{s}+\|\delta \Psi_k\|_{H^s(\Omega_{T})}+\|\mathrm{D}_{x'}\delta\psi_k\|_{H^{s}(\Sigma_T)}
\leq \boldsymbol{\epsilon} \theta_k^{s-{\alpha }-1}\varDelta_k
 \\
 &\quad \quad \textrm{for all } k=0,\ldots,{n}-1\textrm{ and }s=6,\ldots,\widetilde{\alpha} ;\\[1mm]
 \textrm{(b)}\,\, &\max\big\{\|\mathcal{L}_+( V_k,  \Psi_k)-f^a\|_{H_*^s(\Omega_{T}^+)},\, \|\mathcal{L}_-( h_k,  \Psi_k)\|_{H^{s+1}(\Omega_{T}^-)}\big\} \leq 2 \boldsymbol{\epsilon} \theta_k^{s-{\alpha }-1}\\
& \quad \quad  \textrm{for all } k=0,\ldots,{n}-1\textrm{ and } s= 6,\ldots,\widetilde{\alpha}-2;\\[1mm]
 \textrm{(c)}\,\,  &\|(\mathcal{B}_+( V_k, h_k, \psi_k),\,\mathcal{B}_-(h_k, \psi_k))\|_{H^{s} \times H^{s+1}} \leq  \boldsymbol{\epsilon} \theta_k^{s-{\alpha }-1}
 \\
 & \quad \quad  \textrm{for all } k=0,\ldots,{n}-1\textrm{ and } s=7,\ldots,{\alpha },
 \end{aligned}\right.
\end{align*}
for $\varDelta_{k}:=\theta_{k+1}-\theta_k$ with $\theta_{k}:=(\theta^2_0+k)^{1/2}$, where $\|(U,h,\varphi)\|_{s}$ and $\|(g^+,g^-)\|_{H^s\times H^{s+1}}$ are defined by \eqref{norm:ring} and \eqref{norm:g}, respectively.
Using hypothesis $(\mathbf{H}_{{n}-1})$ leads to the following result as in \cite[Lemmas 6--7]{CS08MR2423311}.
\begin{lemma}
\label{lem:triangle}
 If $\theta_0$ is large enough, then
 \begin{align}
 &\|(V_k,h_k,\psi_k) \|_{s}
 +\| \Psi_k\|_{H^s(\Omega_{T})}
 \leq \bigg\{\begin{aligned}
  &\boldsymbol{\epsilon} \theta_k^{(s-{\alpha })_+}   &&\textrm{if }s\neq {\alpha },\\
  &\boldsymbol{\epsilon} \log \theta_k   &&\textrm{if }s= {\alpha },
 \end{aligned}
\nonumber  
\\[0.5mm]
 & \big\|\big(({I}-\mathcal{S}_{\theta_k})V_k,\, ({I}-\mathcal{S}_{\theta_k})h_k ,\, ({I}-\mathcal{S}_{\theta_k})\psi_k\big)\big\|_{s}
 +\big\|({I}-\mathcal{S}_{\theta_k})\Psi_k  \big\|_{H^s(\Omega_{T})}
 \lesssim \boldsymbol{\epsilon} \theta_k^{s-{\alpha }},
\nonumber 
 \end{align}
for $k=0,\ldots,{n}-1$ and  $s=6,\ldots,\widetilde{\alpha}$,
where $\|\cdot \|_{s}$ is defined by \eqref{norm:ring}. Moreover,
 \begin{align}
 &
\big\|\big(\mathcal{S}_{\theta_k}V_k ,\, \mathcal{S}_{\theta_k}h_k ,\,\mathcal{S}_{\theta_k}\psi_k\big)\big\|_{s}
+\big\|\mathcal{S}_{\theta_k}\Psi_k  \big\|_{H^s(\Omega_{T})}
\lesssim  \left\{\begin{aligned}
  &\boldsymbol{\epsilon} \theta_k^{(s-{\alpha })_+}  &&\textrm{if }s\neq {\alpha },\\
  &\boldsymbol{\epsilon} \log \theta_k  &&\textrm{if }s= {\alpha },
 \end{aligned}\right.  \nonumber  
 \end{align}
 for $k=0,\ldots,{n}-1$ and $s=6,\ldots,\widetilde{\alpha}+6$.
\end{lemma}

\subsection{Estimate of the error terms}
This section is devoted to obtaining the estimate for the error terms $E_{k}^{\pm}$ and $\widetilde{E}_{k}^{\pm}$ defined by \eqref{E.E.tilde} and \eqref{e.e.tilde}.
For this purpose, we need estimates for the second derivatives $\mathbb{L}_{\pm}''$ and $\mathbb{B}_{\pm}''$ of  the operators $\mathbb{L}_{\pm}$ and $\mathbb{B}_{\pm}$.
In particular, for $\mathbb{B}_{\pm}$ defined by \eqref{B.bb:def}, we have
\begin{align}
\nonumber
&\mathbb{B}_+''\big(\mathring{\varphi}\big)\big((V,h,\psi),(\widetilde{V},\tilde{h},\tilde{\psi}) \big) \\  & \quad =
\begin{pmatrix}
\tilde{v}'   \cdot \mathrm{D}_{x'} \psi + v'  \cdot \mathrm{D}_{x'}  \tilde{\psi} \\
 \mathfrak{s}\mathrm{D}_{x'}\cdot
 \left(
 \dfrac{\mathring{\zeta}\cdot\tilde{\zeta}}{|\mathring{N}|^3}\zeta
 -\dfrac{\tilde{\zeta}\cdot {\zeta}}{|\mathring{N}|^3}\mathring{\zeta}
 -\dfrac{\mathring{\zeta}\cdot {\zeta}}{|\mathring{N}|^3}\tilde{\zeta}
 +\dfrac{3(\mathring{\zeta}\cdot {\zeta})(\mathring{\zeta}\cdot\tilde{\zeta})}{|\mathring{N}|^5}\mathring{\zeta}
 \right)-h\cdot\tilde{h}\\ 0
\end{pmatrix},
\label{B''_+:def}
\\[0.5mm]
&\mathbb{B}_-''\big((h,\psi),(\tilde{h},\tilde{\psi}) \big)=
\begin{pmatrix}
-\tilde{h}'   \cdot \mathrm{D}_{x'} \psi - h'  \cdot \mathrm{D}_{x'}  \tilde{\psi} \\[0.5mm]
0
\end{pmatrix},
\label{B''_-:def}
\end{align}
with $\zeta:=\mathrm{D}_{x'}\psi$, $\mathring{\zeta}:=\mathrm{D}_{x'}\mathring{\varphi}$, and $\tilde{\zeta}:=\mathrm{D}_{x'}\tilde{\psi}$.
We employ the embedding and Moser-type calculus inequalities to derive the following result ({\it cf.}~\cite[Proposition 23]{ST14MR3151094}).
\begin{proposition}  \label{pro:tame2}
Let $T>0$, $s\in\mathbb{N}$ with $s\geq 6$, and 
\begin{align*}
(V_i,h_i,\Psi_i,\psi_i)\in H_*^{s+2}(\Omega_T^+)\times H^{s+2}(\Omega_T^-)\times H^{s+2}(\Omega_T)\times H^{s+2}(\Sigma_T)
\ \textrm{for } i=0,1,2.
\end{align*}
If
$\|(V_0,\Psi_0)\|_{H_*^{6}(\Omega_T^+)} +\|(h_0,\Psi_0)\|_{H^{4}(\Omega_T^-)}  +\|{\psi}_0\|_{H^{3}(\Sigma_T)}\leq \widetilde{K}$
for some $\widetilde{K}>0$, then
\begin{align*}
&\big\|\mathbb{L}_+''\big(V_0,{\Psi}_0 \big) \big((V_1,\Psi_1),(V_2,\Psi_2) \big)\big\|_{H_*^s(\Omega_{T}^+)} \lesssim_{\widetilde{K}}
 \sum_{i+ j=3} \|(V_i,\Psi_i)\|_{H_*^{6}(\Omega_{T}^+)}   \\
&\qquad\qquad\quad    \times \|(V_j,\Psi_j)\|_{H_*^{s+2}(\Omega_{T}^+)}
+ \| (V_0,{\Psi}_0 )\|_{H_*^{s+2}(\Omega_{T}^+)} \prod_{i=1,2} \|(V_i,\Psi_i)\|_{H_*^{6}(\Omega_{T}^+)},\\[0.5mm]
&\big\|\mathbb{L}_-''\big({h}_0,{\Psi}_0 \big) \big((h_1,\Psi_1),(h_2,\Psi_2) \big)\big\|_{H^{s+1}(\Omega_{T}^-)} \lesssim_{\widetilde{K}}
\sum_{i+ j=3} \|(h_i,\Psi_i)\|_{H^{4}(\Omega_{T}^-)}   \\
&\qquad\qquad\quad    \times \|(h_j,\Psi_j)\|_{H^{s+2}(\Omega_{T}^-)}
+ \| ({{h}}_0,{\Psi}_0 )\|_{H^{s+2}(\Omega_{T}^-)} \prod_{i=1,2} \|(h_i,\Psi_i)\|_{H^{4}(\Omega_{T}^-)},
\end{align*}
and
\begin{align}
&\left\|\left(\mathbb{B}_+''\big(\psi_0 \big)
 \big((V_1,h_1,\psi_1),(V_2,h_2,\psi_2) \big),\, \mathbb{B}_-''\big((h_1,\psi_1),(h_2,\psi_2) \big) \right)\right\|_{H^{s}\times H^{s+1}}
\nonumber\\[1.5mm]
&\lesssim_{\widetilde{K}}
\|\psi_0\|_{H^{s+2}(\Sigma_T)} \prod_{i=1,2} \|\psi_i\|_{H^{3}(\Sigma_T)}
+\sum_{i+ j=3}
\Big\{ \|\psi_i\|_{H^{3}(\Sigma_T)}\|\psi_j\|_{H^{s+2}(\Sigma_T)}
+\|h_i\|_{H^{2}(\Sigma_T)}
 \nonumber\\
&\quad\ \  \times \|h_j\|_{H^{s}(\Sigma_T)}  + \|(V_i,h_i)\|_{H^{s+1}(\Sigma_T)}\|\psi_j\|_{H^{3}(\Sigma_T)}
 + \|(V_i,h_i)\|_{H^{2}(\Sigma_T)} \|\psi_j\|_{H^{s+2}(\Sigma_T)}\Big\},
 \nonumber
\end{align}
where the norm $\|\cdot\|_{H^s\times H^{s+1}}$ is defined by \eqref{norm:g}.
\end{proposition}

We first employ Proposition \ref{pro:tame2} to estimate the quadratic error terms $e'_{k\pm}$ and $\tilde{e}_{k\pm}'$ defined in \eqref{decom1}--\eqref{decom4}.

\begin{lemma}\label{lem:quad}
If $\theta_0\geq 1$ is large enough and $\boldsymbol{\epsilon}>0$ is sufficiently small, then
 \begin{align}\nonumber 
  \|e_{k+}'\|_{H_*^s(\Omega_{T}^+)}+\|e_{k-}'\|_{H^{s+1}(\Omega_{T}^-)}+ \|(\tilde{e}_{k+}',\, \tilde{e}_{k-}')\|_{H^{s}\times H^{s+1}}
  \lesssim \boldsymbol{\epsilon}^2 \theta_k^{\varsigma_1(s)-1}\varDelta_k,
 \end{align}
with $\varsigma_1(s):=\max\{(s+2-{\alpha })_++10-2{\alpha },\,s+6-2{\alpha } \},$
for all $k\in \{0,\ldots,{n}-1\}$  and $s\in \{6,\ldots,\widetilde{\alpha}-2\}$.
\end{lemma}
\begin{proof}  Rewriting the quadratic error term $e_{k-}'$ as
\begin{align*}
e_{k-}' =\int_{0}^{1}(1-\tau)
\mathbb{L}_-''\big(h^a+h_{k}+\tau \delta  h_{k},
\Phi^a+\Psi_{k} +\tau \delta \Psi_{k}\big)
\big((\delta h_{k},\delta\Psi_{k}),(\delta h_{k},\delta\Psi_{k})\big)\d\tau,
\end{align*}
we employ Proposition \ref{pro:tame2} and hypothesis $(\mathbf{H}_{{n}-1})$  to infer
\begin{align*}
 \|e_{k-}'\|_{H^{s+1}(\Omega_{T}^-)}
 \lesssim \;& \|(\delta h_{k},\delta\Psi_{k})\|_{H^{4}(\Omega_{T}^-)}\|(\delta h_{k},\delta\Psi_{k})\|_{H^{s+2}(\Omega_{T}^-)} \\
&+\|(\delta h_{k},\delta\Psi_{k})\|_{H^{4}(\Omega_{T}^-)} ^2
\|( h^a, \Phi^a, h_{k}, \Psi_{k},\delta h_{k},\delta\Psi_{k})\|_{H^{s+2}(\Omega_{T}^-)}\\
\lesssim \;&
\boldsymbol{\epsilon}^2\theta_{k}^{s+6-2{\alpha }}\varDelta_k^2
+\boldsymbol{\epsilon}^2\theta_k^{10-2{\alpha }}\varDelta_k^2\big(1+\|(h_{k}, \Psi_{k})\|_{H^{s+2}(\Omega_{T}^-)}\big)
\end{align*}
for all $s\in\{6,\ldots,\widetilde{\alpha}-2\}$.
In view of Lemma \ref{lem:triangle},
we analyze the cases $s+2\neq \alpha$ and $s+2= \alpha$ separately to obtain the estimate for $e_{k-}'$.
And the estimates for $e_{k+}'$ and $\tilde{e}_{k\pm}'$ follow in an entirely similar way.
\end{proof}

Then we obtain the following result concerning the estimate of  the first substitution error terms $e_{k\pm}''$ and $\tilde{e}_{k\pm }''$ given in \eqref{decom1}--\eqref{decom4}.
\begin{lemma} \label{lem:1st}
If $\theta_0\geq 1$ is large enough and $\boldsymbol{\epsilon}>0$ is sufficiently small, then
 \begin{align}  \nonumber 
  \|e_{k+}''\|_{H_*^s(\Omega_{T}^+)}+\|e_{k-}''\|_{H^{s+1}(\Omega_{T}^-)}+ \|(\tilde{e}_{k+}'',\, \tilde{e}_{k-}'')\|_{H^{s}\times H^{s+1}}
  \lesssim \boldsymbol{\epsilon}^2 \theta_k^{\varsigma_2(s)-1}\varDelta_k,
 \end{align}
with $\varsigma_2(s):=\max\{(s+2-{\alpha })_++12-2{\alpha },\, s+8-2{\alpha } \},$
for all $k\in \{0,\ldots,{n}-1\}$  and $s\in \{6,\ldots,\widetilde{\alpha}-2\}$.
\end{lemma}
\begin{proof}
Applying Proposition \ref{pro:tame2} to the first substitution error term
\begin{align}
 \nonumber {e}_{k-}''=\int_{0}^{1}\;&\mathbb{L}_-''\big(h^a+\mathcal{S}_{\theta_k}h_k+\tau({I}-\mathcal{S}_{\theta_k})h_k,\,
 \Phi^a+\mathcal{S}_{\theta_k}\Psi_k \\
 & +\tau({I}-\mathcal{S}_{\theta_k})\Psi_k\big)\big((\delta h_k ,\delta \Psi_k),\, (({I}-\mathcal{S}_{\theta_k})h_k,({I}-\mathcal{S}_{\theta_k})\Psi_k) \big)\d\tau,\nonumber
\end{align}
we use hypothesis $(\mathbf{H}_{{n}-1})$ and Lemma \ref{lem:triangle} to deduce
\begin{align*}
 \|e_{k-}''\|_{H^{s+1}(\Omega_{T}^-)}
 \lesssim \;&
 \boldsymbol{\epsilon}^2\theta_{k}^{s+7-2{\alpha }}\varDelta_k
 +\boldsymbol{\epsilon}^2\theta_k^{11-2{\alpha }}\varDelta_k\big(1+\|(\mathcal{S}_{\theta_k} h_{k}, \mathcal{S}_{\theta_k} \Psi_{k})\|_{H^{s+2}(\Omega_{T}^-)}\big)
\end{align*}
for all $s\in\{6,\ldots,\widetilde{\alpha}-2\}$.
Then the estimate for $e_{k-}''$ follows by using Lemma \ref{lem:triangle} again.
The similar argument applies also for the terms $e_{k+}''$ and $\tilde{e}_{k\pm }''$.
\end{proof}

The construction and estimate of the modified state in \cite[\S 12.4]{ST14MR3151094} are independent of the second boundary condition in \eqref{NP1c} ({\it cf.}~\cite{TW21MR4201624,TW21b}), which yields the following proposition for our problem.
\begin{proposition}\label{pro:modified}
Let ${\alpha }\geq 8$.
If $\theta_0\geq 1$ is large enough and $\boldsymbol{\epsilon},T>0$ are sufficiently small, then
there exist functions $V_{n+1/2}$, $h_{n+1/2}$, and $\psi_{n+1/2}$ vanishing in the past,
such that $(U^a+V_{n+1/2},h^a+h_{n+1/2}, \varphi^a+\psi_{n+1/2})$ satisfies \eqref{bas1a}--\eqref{bas1f}
for the approximate solution $(U^a, h^a, \varphi^a)$ constructed in Lemma \textnormal{\ref{lem:app}}.
Moreover,
\begin{alignat*}{3}
& \psi_{n+1/2}=\mathcal{S}_{\theta_n}\psi_{{n}}, &&
 \|\mathcal{S}_{\theta_n}\Psi_n-\Psi_{n+1/2}\|_{H^s(\Omega_{T})}\lesssim \boldsymbol{\epsilon} \theta_n^{s-{\alpha }}
 \quad   \textrm{for } s=6,\ldots,\widetilde{\alpha}+6,  \\
&v_{n+1/2}'=\mathcal{S}_{\theta_n} v_{n}',\quad&& \|\mathcal{S}_{\theta_n}V_n-V_{n+1/2}\|_{H_*^s(\Omega_{T}^+)}
 +\|\mathcal{S}_{\theta_n}h_n-h_{n+1/2}\|_{H^{s+1}(\Omega_{T}^-)}
 \lesssim \boldsymbol{\epsilon} \theta_n^{s+2-{\alpha }}
\end{alignat*}
for $s=6,\ldots,\widetilde{\alpha}+{4} $ and $\Psi_{n+1/2}:=\chi(x_1)\psi_{{n}+1/2}.$
\end{proposition}

We have the estimate for the second substitution error terms $e_{k\pm}'''$ and $\tilde{e}_{k\pm }'''$ defined in \eqref{decom1}--\eqref{decom4}.
\begin{lemma} \label{lem:2nd}
 Let ${\alpha }\geq 8$.
If $\theta_0\geq 1$ is large enough and $\boldsymbol{\epsilon},T>0$ are sufficiently small, then
 \begin{align} \nonumber 
&\|e_{k+}'''\|_{H_*^s(\Omega_{T}^+)}+\|e_{k-}'''\|_{H^{s+1}(\Omega_{T}^-)}\lesssim \boldsymbol{\epsilon}^2 \theta_k^{\varsigma_3(s)-1}\varDelta_k,\\
& \|(\tilde{e}_{k+}''',\, \tilde{e}_{k-}''')\|_{H^{s}\times H^{s+1}}\lesssim \boldsymbol{\epsilon}^2 \theta_k^{\varsigma_2(s)-1}\varDelta_k,
\nonumber 
 \end{align}
with $\varsigma_3(s):=\max\{(s+2-{\alpha })_++14-2{\alpha },\, s+10-2{\alpha } \},$
for all $k\in \{0,\ldots,{n}-1\}$  and $s\in \{6,\ldots,\widetilde{\alpha}-2\}$.
\end{lemma}
\begin{proof}
First we infer from Proposition \ref{pro:modified} that
\begin{align}
\nonumber \tilde{e}_{k+}'''=\,&\int_{0}^{1}
 \mathbb{B}_+''\big(\varphi^a+\psi_{k+1/2}\big)
 \big((\delta V_k ,\delta h_k ,\delta \psi_k),\, (\mathcal{S}_{\theta_k} V_k-V_{k+1/2},\mathcal{S}_{\theta_k} h_k-h_{k+1/2},0)\big) \d\tau,\\
\nonumber \tilde{e}_{k-}'''=\,&
\int_{0}^{1}
\mathbb{B}_-''\big(( \delta h_k ,\delta \psi_k),\, ( \mathcal{S}_{\theta_k} h_k-h_{k+1/2},0)\big) \d\tau,
\end{align}
which combined with \eqref{B''_+:def}--\eqref{B''_-:def} yield
\begin{align*}
\nonumber \tilde{e}_{k+}'''=
\begin{pmatrix}
 0 \\ -\delta h_k\cdot (\mathcal{S}_{\theta_k} h_k-h_{k+1/2}) \\ 0
\end{pmatrix},  \quad
\nonumber \tilde{e}_{k-}'''=
\begin{pmatrix}
 (h_{k+1/2}'-\mathcal{S}_{\theta_k} h_k')\cdot \mathrm{D}_{x'}\delta\psi_k \\ 0
\end{pmatrix}.
\end{align*}
Applying the Moser-type calculus inequalities to the above identities,
we use hypothesis $(\mathbf{H}_{{n}-1})$, Lemma \ref{lem:triangle}, and Proposition \ref{pro:modified} to
get the estimate for $\tilde{e}_{k\pm}'''$.
Apply a similar argument to deduce the estimate for $e_{k\pm }'''$.
\end{proof}

Using \eqref{Alinhac1}--\eqref{Alinhac2}, we can rewrite the last error terms $e_{{n}\pm}^{*}$ in \eqref{decom1}--\eqref{decom2} as
\begin{align*}
&e_{{n}+}^{*}=\frac{\p_1\mathbb{L}_{+}(U^{a}+V_{{n}+1/2},\Phi^{a}+\Psi_{{n}+1/2})
}{\p_1(\Phi^{a}+\Psi_{{n}+1/2} )} \delta \Psi_{{n}},\\[1mm]
&e_{{n}-}^{*}=\frac{\p_1\mathbb{L}_{-}(h^{a}+h_{{n}+1/2},\Phi^{a}+\Psi_{{n}+1/2})
}{\p_1(\Phi^{a}+\Psi_{{n}+1/2} )} \delta \Psi_{{n}}.
\end{align*}
As in \cite[\S 7.6]{CS08MR2423311} or \cite[\S 4]{TW21MR4201624},
we can get the following result by use of the embedding and Moser-type calculus inequalities,
hypothesis ($\mathbf{H}_{n-1}$), and Proposition \ref{pro:modified}.
\begin{lemma}\label{lem:last}
Let $\widetilde{\alpha}\geq {\alpha }+2$ and ${\alpha }\geq 8$.
If $\theta_0\geq 1$ is large enough and $\boldsymbol{\epsilon},T>0$ are sufficiently small, then
 \begin{align} \nonumber 
\|e_{k+}^*\|_{H_*^s(\Omega_{T}^+)}+\|e_{k-}^*\|_{H^{s+1}(\Omega_{T}^-)}
  \lesssim \boldsymbol{\epsilon}^2 \theta_k^{\varsigma_4 (s)-1}\varDelta_k,
 \end{align}
with $ \varsigma_4(s):=\max\{(s-{\alpha })_++18-2{\alpha },\, s+12-2{\alpha }\},$
for all $k\in \{0,\ldots,{n}-1\}$  and $s\in \{6,\ldots,\widetilde{\alpha}-2\}$.
\end{lemma}

Similar to \cite[Lemma 4.12]{TW21MR4201624},
we utilize Lemmas \ref{lem:quad}--\ref{lem:last} to obtain the following estimates for the accumulated error terms $E_{n}^{\pm}$ and $\widetilde{E}_n^{\pm}$ defined by \eqref{E.E.tilde} and \eqref{e.e.tilde}.
\begin{lemma} \label{lem:sum2}
Let $\widetilde{\alpha}={\alpha }+3$ and ${\alpha }\geq 12$.
If $\theta_0\geq 1$ is large enough and $\boldsymbol{\epsilon},T>0$ are sufficiently small, then
 \begin{align}\nonumber 
\|E_{n}^+\|_{H_*^{{\alpha }+1}(\Omega_{T}^+)}+\|E_{n}^-\|_{H^{{\alpha }+2}(\Omega_{T}^-)}
\lesssim \boldsymbol{\epsilon}^2 \theta_{{n}},\quad
\|(\widetilde{E}_{{n}}^+,\, \widetilde{E}_{{n}}^-)\|_{H^{{\alpha }+1}\times H^{{\alpha }+2}}\lesssim \boldsymbol{\epsilon}^2,
 \end{align}
where the norm $\|\cdot\|_{H^s\times H^{s+1}}$ is defined by \eqref{norm:g}.
\end{lemma}

\subsection{Proof of Theorem \ref{thm:main}}

Similar to \cite[Lemma 4.13]{TW21MR4201624}, we can deduce the following estimates for source terms $f_{{n}}^{\pm}$ and $g_{{n}}^{\pm}$ computed from \eqref{source}.

\begin{lemma} \label{lem:source}
Let $\widetilde{\alpha}={\alpha }+3$ and ${\alpha }\geq 12$.
If $\theta_0\geq 1$ is large enough and $\boldsymbol{\epsilon},T>0$ are sufficiently small, then
 \begin{align}
  &
\|f_{{n}}^+\|_{H_*^s(\Omega_{T}^+)}+\|f_{{n}}^-\|_{H^{s+1}(\Omega_{T}^-)}
  \lesssim \varDelta_{{n}}\big(\theta_{{n}}^{s-{\alpha }-1} \|f^a\|_{{\alpha },*,T}
  +\boldsymbol{\epsilon}^2 \theta_{{n}}^{s-{\alpha }-1} +\boldsymbol{\epsilon}^2\theta_{{n}}^{\varsigma_4(s)-1}\big),
  \nonumber \\ 
  &
  \|(g_{{n}}^+,\, g_{{n}}^-)\|_{H^{s+1}\times H^{s+2}}
  \lesssim  \boldsymbol{\epsilon}^2 \varDelta_{{n}}\big(\theta_{{n}}^{s-{\alpha }-1}
  +\theta_{{n}}^{\varsigma_2(s+1)-1}\big)\nonumber
\qquad \textrm{for all }s\in\{6,\ldots,\widetilde{\alpha}\}.
 \end{align}
\end{lemma}

Similar to \cite[Lemma 16]{CS08MR2423311} and
\cite[Lemma 15]{T09ARMAMR2481071},
applying the tame estimate \eqref{tame} to the problem \eqref{effective.NM},
we can use Proposition \ref{pro:modified} and Lemma \ref{lem:source} to derive the estimate (a) in hypothesis $(\mathbf{H}_{{n}})$.
\begin{lemma}\  \label{lem:Hl1}
 Let $\widetilde{\alpha}={\alpha }+3$ and ${\alpha }\geq 12$.
If $\boldsymbol{\epsilon},T>0$ and $\frac{1}{\boldsymbol{\epsilon}}\|f^a\|_{H_*^{\alpha}(\Omega_{T}^+)}$ are sufficiently small, and $\theta_0\geq 1$ is large enough, then
 \begin{align} \nonumber  
\|(\delta V_n, \delta h_n, \delta\psi_n)\|_{s}+\|\delta \Psi_n\|_{H^s(\Omega_{T})}+\|\mathrm{D}_{x'}\delta\psi_n\|_{H^{s}(\Sigma_T)}
\leq \boldsymbol{\epsilon} \theta_{{n}}^{s-{\alpha }-1}\varDelta_{{n}}
 \end{align}
 for all $s\in\{6,\ldots,\widetilde{\alpha}\}$.
\end{lemma}

The next lemma gives the other estimates in hypothesis $(\mathbf{H}_{{n}})$, whose proof is similar to that of  
\cite[Lemma 16]{T09ARMAMR2481071}.

\begin{lemma}\ \label{lem:Hl2}
Let $\widetilde{\alpha}={\alpha }+3$ and ${\alpha }\geq 12$.
If $\boldsymbol{\epsilon},T>0$ and $\frac{1}{\boldsymbol{\epsilon}}\|f^a\|_{H_*^{\alpha}(\Omega_{T}^+)}$ are sufficiently small, and $\theta_0\geq 1$ is large enough, then
 \begin{alignat}{3}\label{Hl.b}
\|\mathcal{L}_+( V_n,  \Psi_n)-f^a\|_{H_*^s(\Omega_{T}^+)}\leq 2 \boldsymbol{\epsilon} \theta_{{n}}^{s-{\alpha }-1},
\quad
\|\mathcal{L}_-( h_n,  \Psi_n)\|_{H^{s+1}(\Omega_{T}^-)}
\leq 2 \boldsymbol{\epsilon} \theta_{{n}}^{s-{\alpha }-1}
 \end{alignat}
for all $s\in \{6,\ldots,\widetilde{\alpha}-2\}$, and
\begin{align}
 \label{Hl.c}
 \|(\mathcal{B}_+( V_n, h_n, \psi_n),\,\mathcal{B}_-(h_n, \psi_n))\|_{H^{s} \times H^{s+1}}
\leq  \boldsymbol{\epsilon} \theta_{{n}}^{s-{\alpha }-1}
\quad  \textrm{for all } s\in \{7,\ldots,{\alpha }\}.
\end{align}
\end{lemma}

Thanks to Lemmas \ref{lem:Hl1}--\ref{lem:Hl2},
hypothesis $(\mathbf{H}_{{n}})$ follows from $(\mathbf{H}_{{n}-1})$ provided that $\widetilde{\alpha}={\alpha }+3$ and ${\alpha }\geq 12$ hold,
$\boldsymbol{\epsilon}, T>0$ and $\frac{1}{\boldsymbol{\epsilon}}\|f^a\|_{H_*^{\alpha}(\Omega_{T}^+)}$ are sufficiently small,
and $\theta_0 \geq 1$ is large enough.
Fixing the constants ${\alpha }\geq 12$, $\widetilde{\alpha}={\alpha }+3$,
$\boldsymbol{\epsilon}>0$, and $\theta_0\geq1$,
we can show hypothesis $(\mathbf{H}_{0})$ as in \cite[Lemma 17]{T09ARMAMR2481071}.

\begin{lemma}\ \label{lem:H0}
 If time $T>0$ is small enough, then hypothesis $(\mathbf{H}_0)$ is satisfied.
\end{lemma}

We are ready to conclude the proof of our main result.

\vspace{2mm}
\noindent  {\bf Proof of Theorem {\rm\ref{thm:main}}.}\quad
Suppose that the initial data $(U_0,\varphi_0)$ satisfy all the assumptions in Theorem {\rm\ref{thm:main}}.
Set $\widetilde{\alpha}=m-{5}$ and ${\alpha }=\widetilde{\alpha}-3\geq 12$.
Then the initial data $(U_0,\varphi_0)$ are compatible up to order $m=\widetilde{\alpha}+{5}$.
Taking $\boldsymbol{\epsilon}, T>0$ small enough and $\theta_0\geq 1$ suitably large can verify all the requirements of Lemmas \ref{lem:Hl1}--\ref{lem:H0} due to \eqref{app2a}--\eqref{f^a:est}.
Therefore, we can find some $T>0$, such that hypothesis $(\mathbf{H}_{{n}})$ is satisfied for all ${n}\in\mathbb{N}$, implying
\begin{align*}
\sum_{n=0}^{\infty}\left(
\|\delta V_n\|_{H_*^s(\Omega_{T}^+)}+\|\delta h_n\|_{H^s(\Omega_{T}^-)}
+\|(\delta\psi_n,\mathrm{D}_{x'}\delta\psi_n)\|_{H^{s}(\Sigma_T)}
\right)
\lesssim \sum_{n=0}^{\infty}\theta_n^{s-{\alpha }-2} <\infty
\end{align*}
for all $s\in \{6,\ldots, {\alpha }-1\}.$
So the sequence $(V_{n},h_n,\psi_n)$ converges to some limit $(V,h, \psi)$ in $H_*^{{\alpha }-1}(\Omega_T^+)\times H^{{\alpha }-1}(\Omega_T^-) \times H^{{\alpha }-1}({\Sigma_T})$.
Furthermore, $V_n\to V$ in $H^{\lfloor ({\alpha }-1)/2\rfloor}(\Omega_T^+)$ as $n\to \infty$ and $\mathrm{D}_{x'}\psi \in H^{{\alpha }-1}({\Sigma_T})$.
Passing to the limit in \eqref{Hl.b}--\eqref{Hl.c} for $s=m-{9}$ implies \eqref{P.new}, and
hence $(U, \hat{h}, \varphi)=(U^a+V, h^a+h, \varphi^a+\psi)$
is a solution of problem \eqref{NP1} on $[0,T]$.
The uniqueness of solutions to problem \eqref{NP1} can be achieved by a standard argument (see, {\it e.g.}, \cite[\S 13]{ST14MR3151094} or \cite[\S 3.5]{TW21b}).
The proof is complete.
\qed

\bigskip



{\footnotesize 
  }

\end{document}